\documentclass[11pt,a4paper]{amsart}
\usepackage{graphicx}
\usepackage{amsmath,amsfonts}
\usepackage{amsthm,amssymb,latexsym}
\usepackage[active]{srcltx}
\usepackage[usenames]{color}

\newcommand{\bfs}{\boldsymbol}

\newtheorem{theorem}{Theorem}[section]
\newtheorem{corollary}[theorem]{Corollary}
\newtheorem{lemma}[theorem]{Lemma}

\newtheorem*{fact*}{Fact}
\newtheorem{proposition}[theorem]{Proposition}
\theoremstyle{definition}

\newtheorem*{claim}{Claim}
\newtheorem{remark}[theorem]{Remark}
\numberwithin{equation}{section}

\newtheorem*{algorithm2}{Classical  factorization algorithm}
\newtheorem*{algorithm3}{ERF algorithm}
\newtheorem*{algorithm4}{DDF Algorithm}
\newtheorem*{algorithm5}{EDF algorithm}
\newcommand{\N}{\mathbb N}
\newcommand{\Z}{\mathbb Z}

\newcommand{\A}{\mathbb A}
\newcommand{\F}{\mathbb F}
\newcommand{\K}{\mathbb K}

\newcommand{\Pp}{\mathbb P}

\newcommand{\fq}{\F_{\hskip-0.7mm q}}
\newcommand{\fqk}{\F_{\hskip-0.7mm q^k}}

\newcommand{\fqtwo}{\F_{\hskip-0.7mm q^{2}}}
\newcommand{\fqi}{\F_{\hskip-0.7mm q^i}}
\newcommand{\cfq}{\overline{\F}_{\hskip-0.7mm q}}
\def\ifm#1#2{\relax \ifmmode#1\else#2\fi}

\newcommand{\klk}    {\ifm {,\ldots,} {$,\ldots,$}}

\newcommand{\plp}    {\ifm {+\cdots+} {$+\ldots+$}}

\newcommand{\wt}    {\ifm {{\sf wt}} {{$\sf wt$}}}

\begin{document}

\title[Factorization on nonlinear families]{Factorization patterns
on nonlinear families of univariate polynomials over a finite field}
\author[G. Matera]{
Guillermo Matera${}^{1,2}$}
\author[M. P\'erez]{
Mariana P\'erez${}^{1,3}$}
\author[M. Privitelli]{
Melina Privitelli${}^{2,4}$}

\address{${}^{1}$Instituto del Desarrollo Humano,
Universidad Nacional de Gene\-ral Sarmiento, J.M. Guti\'errez 1150
(B1613GSX) Los Polvorines, Buenos Aires, Argentina}
\email{gmatera@ungs.edu.ar}
\address{${}^{2}$ National Council of Science and Technology (CONICET),
Ar\-gentina}
\address{${}^{3}$Instituto de Ciencias,
Universidad Nacional de Hurlingham, Av. Gdor. Vergara 2222
(B1688GEZ) Hurlingham, Buenos Aires, Argentina}
\email{mariana.perez@unahur.edu.ar}
\address{${}^{4}$Instituto de Ciencias,
Universidad Nacional de Gene\-ral Sarmiento, J.M. Guti\'errez 1150
(B1613GSX) Los Polvorines, Buenos Aires, Argentina}
\email{mprivite@ungs.edu.ar}

\thanks{The authors were partially supported by the grants
PIP CONICET 11220130100598 and PIO CONICET-UNGS 14420140100027}


\date{\today}%

\begin{abstract}
We estimate the number $|\mathcal{A}_{\bfs\lambda}|$ of elements on
a nonlinear family $\mathcal{A}$ of monic polynomials of $\fq[T]$ of
degree $r$ having factorization pattern
$\bfs\lambda:=1^{\lambda_1}2^{\lambda_2}\cdots r^{\lambda_r}$. We
show that $|\mathcal{A}_{\bfs\lambda}|=
\mathcal{T}(\bfs\lambda)\,q^{r-m}+\mathcal{O}(q^{r-m-{1}/{2}})$,
where $\mathcal{T}(\bfs\lambda)$ is the proportion of elements of
the symmetric group of $r$ elements with cycle pattern $\bfs\lambda$
and $m$ is the codimension of $\mathcal{A}$. We provide explicit
upper bounds for the constants underlying the
$\mathcal{O}$--notation in terms of $\bfs\lambda$ and $\mathcal{A}$
with ``good'' behavior. We also apply these results to analyze the
average--case complexity of the classical factorization algorithm
restricted to $\mathcal{A}$, showing that it behaves as good as in
the general case.
\end{abstract}

\maketitle

%
%
\section{Introduction}
The distribution of factorization patterns on univariate polynomials
over a finite field $\fq$ is a classical subject of combinatorics.
Let $\bfs \lambda:=1^{\lambda_1}2^{\lambda_2}\cdots r^{\lambda_r}$
be a factorization pattern for polynomials of degree $r$, namely
$\lambda_1,\ldots,\lambda_r\in\Z_{\ge 0}$ satisfy
$\lambda_1+2\lambda_2+\cdots+r\lambda_r=r$. A seminal article of S.
Cohen (\cite{Cohen70}) shows that the proportion of elements of
$\fq[T]$ of degree $r$ is roughly the proportion
$\mathcal{T}({\bfs{\lambda}})$ of permutations with cycle pattern
$\bfs \lambda$ in the $r$th symmetric group $\mathbb{S}_r$ (an
element of $\mathbb{S}_r$ has cycle pattern $\bfs \lambda$ if it has
exactly $\lambda_i$ cycles of length $i$ for $1\le i\le r$).

In particular, the number of irreducible polynomials, or more
generally the distribution of factorization patterns, of polynomials
of ``given forms'' has been considered in a number of recent
articles (see, e.g., \cite{Pollack13}, \cite{BaBaRo15}, \cite{Ha16},
\cite{CeMaPe17}). In \cite{Cohen72} a subset of the set of
polynomials of degree $r$ is called {uniformly distributed} if the
proportion of elements with factorization pattern $\bfs\lambda$ is
roughly $\mathcal{T}(\bfs\lambda)$ for every $\bfs\lambda$. The main
result of this paper (\cite[Theorem 3]{Cohen72}) provides a
criterion for a linear family of polynomials of
$\fq[T]$ of given degree to be uniformly distributed in the sense above. 
\cite{BaBaRo15}, \cite{Ha16} and \cite{CeMaPe17} provide explicit
estimates on the number of elements with factorization pattern
$\bfs\lambda$ on certain linear families of $\fq[T]$, such as the
set of polynomials with some prescribed coefficients.

In \cite[Problem 2.2]{GaHoPa99} the authors ask for estimates on the
number of polynomials of a given degree with a given factorization
pattern lying in {\em nonlinear} families of polynomials with
coefficients parameterized by an affine variety defined over $\fq$.
Except for general results (see, e.g., \cite{ChDrMa92} and
\cite{FrHaJa94}), very little is known on the asymptotic behavior of
such a number. In this article we address this question, providing a
general criterion for a nonlinear family $\mathcal{A}\subset\fq[T]$
to be uniform distributed in the sense of Cohen and explicit
estimates on the number of elements of $\mathcal{A}$ with a given
factorization pattern.

Then we apply our results on the distribution of factorization
patterns to analyze the behavior of the classical factorization
algorithm restricted to such families $\mathcal{A}$. The classical
factorization algorithm (see, e.g., \cite{GaGe99}) is not the
fastest one. Nevertheless, it is worth analyzing it, since it is
implemented in several software packages for symbolic computation,
and a number of scientific problems rely heavily on polynomial
factorization over finite fields.

A precise worst--case analysis is given in \cite{GaGe99}. On the
other hand, an average--case analysis for the set of elements of
$\fq[T]$ of a given degree is provided in \cite{FlGoPa01}. This
analysis relies on methods of analytic combinatorics which cannot be
extended to deal with the nonlinear families we are interested in
this article. For this reason, we provide an analysis of its
average--case complexity when restricted to any nonlinear family
$\mathcal{A}$ satisfying our general criterion.

Now we describe precisely our results. Let $\cfq$ be the algebraic
closure of $\fq$. Let $m$ and $r$ be positive integers with $m<r$
and let $A_{r-1}\klk A_{0}$ be indeterminates over $\cfq$. For a
fixed $k$ with $0\leq k\leq r-1$, we denote $\fq[\bfs{A_k}]:=
\fq[A_{r-1},\ldots,A_{k+1},A_{k-1},\ldots,A_0]$. Let $G_1\klk G_m\in
\fq[\bfs{A_k}]$ and let $W:=\{G_1=0,\ldots,G_m=0\}$ be the set of
common zeros in $\cfq{\!}^r$ of $G_1,\ldots,G_m$. Denoting by
$\fq[T]_r$ the set of monic polynomials of degree $r$ with
coefficients in $\fq$, we consider the following family of
polynomials:
\begin{equation}\label{non-linear family A}
\mathcal{A}:=\{T^r+a_{r-1}T^{r-1}+\cdots+a_0\in \fq[T]_r:
G_i(a_{r-1},\ldots,a_{k-1},a_{k+1},\ldots,a_{0})=0\,\,(1\leq i \leq
m)\}.
\end{equation}

Consider the weight $\wt:\fq[\bfs{A_k}]\to\N_0$ defined by setting
$\wt(A_j):=r-j$ for $0\leq j \leq r-1$ and denote by $G_1^{\wt}\klk
G_m^{\wt}$ the components of highest weight of $G_1\klk G_m$. Let
$(\partial\bfs{G}/\partial \bfs{A})$ be the Jacobian matrix of
$G_1\klk G_m$ with respect to $\bfs{A_k}$. We shall assume that $G_1
\klk G_m$ satisfy the following conditions:
\begin{itemize}
\item[$({\sf H}_1)$] $G_1,\ldots,G_m$ form a regular sequence\footnote{This
means that $\{G_1=0,\ldots,G_i=0\}$ has dimension $r-i$ for $1\le
i\le m$; see Section \ref{subsec: complete intersections} for
details.} of $\fq[\bfs{A_k}]$.
\item[$({\sf H}_2)$]  $(\partial\bfs{G}/
\partial \bfs{A_k})$ has full rank on every point of the $W$.
\item[$({\sf H}_3)$]  $G_1^{\wt}\klk G_m^{\wt}$ satisfy
$({\sf H}_1)$ and $({\sf H}_2)$.
\end{itemize}

In what follows we identify the set $\cfq[T]_r$ of monic polynomials
of $\cfq[T]$ of degree $r$ with $\cfq{\!}^r$ by mapping each
$f_{\bfs a_0}:= T^ r + a_{r-1}T^ {r-1}+\cdots+a_0 \in \cfq[T]_r$ to
$\bfs a_0:=(a_{r-1},\ldots,a_0)\in\cfq{\!}^r$. For
$\mathcal{B}\subset \cfq[T]_r$, the set of elements of $\mathcal{B}$
which are not square--free is called the discriminant locus
$\mathcal{D}(\mathcal{B})$ of $\mathcal{B}$ (see \cite{FrSm84} and
\cite{MaPePr14} for the study of discriminant loci). For $f_{\bfs
a_0}\in \mathcal{B}$, let $\mathrm{Disc}(f_{\bfs
a_0}):=\mathrm{Res}(f_{\bfs a_0},f'_{\bfs a_0})$ denote the
discriminant of $f_{\bfs a_0}$, that is, the resultant of $f_{\bfs
a_0}$ and its derivative $f'_{\bfs a_0}$. Since $f_{\bfs a_0}$ has
degree $r$, by basic properties of resultants we have
$$\mathrm{Disc}(f_{\bfs a_0})= \mathrm{Disc}(F(\bfs A_0, T))|_{\bfs A_0=\bfs a_0} :=
\mathrm{Res}(F(\bfs A_0, T), F'(\bfs A_0, T), T)|_{\bfs A_0=\bfs
a_0},
$$
where the expression $\mathrm{Res}$ in the right--hand side denotes
resultant with respect to $T$. It follows that
$\mathcal{D}(\mathcal{B}):=\{\bfs a_0 \in \mathcal{B} :
\mathrm{Disc}(F(\bfs A_0, T))|_{\bfs A_0=\bfs a_0}= 0\}$. We shall
need further to consider first subdiscriminant loci. The first
subdiscriminant locus $\mathcal{S}_1(\mathcal{B})$ of
$\mathcal{B}\subset\cfq[T]_r$ is the set of $\bfs a_0\in
\mathcal{D}(\mathcal{B})$ for which the first subdiscriminant
$\mathrm{Subdisc}(f_{\bfs a_0}):=\mathrm{Subres}(f_{\bfs
a_0},f'_{\bfs a_0})$ vanishes, where $\mathrm{Subres}(f_{\bfs
a_0},f'_{\bfs a_0})$ denotes the first subresultant of $f_{\bfs
a_0}$ and $f'_{\bfs a_0}$. Since $f_{\bfs a_0}$ has degree $r$,
basic properties of subresultants imply
$$\mathrm{Subdisc}(f_{\bfs a_0}
)= \mathrm{Subdisc}(F(\bfs A_0, T))|_{\bfs A_0=\bfs a_0} :=
\mathrm{Subres}(F(\bfs A_0, T),F'(\bfs A_0, T), T))|_{\bfs A_0=\bfs
a_0},$$
where $\mathrm{Subres}$ in the right--hand side denotes first
subresultant with respect to $T$. We have
$\mathcal{S}_1(\mathcal{B}):=\{\bfs a_0\in\mathcal{D}(\mathcal{B}):
\mathrm{Subdisc}(F(\bfs A_0, T))|_{\bfs A_0=\bfs a_0}=0\}$. Our next
conditions require that the discriminant and the first
subdiscriminant locus intersect well $W$:
\begin{enumerate}
\item[$({\sf H}_4)$] $\mathcal{D}(W)$ has codimension at least one in $W$.
\item [$({\sf H}_5)$] $ (A_0\cdot \mathcal{S}_1)(W):=\{\bfs a_0\in W: a_0=0\}\cup \mathcal{S}_1(\mathcal{B})$  has codimension at least one in
$\mathcal{D}(W)$.
\item[$({\sf H}_6)$] $\mathcal{D}(V(G_1^{\wt}\klk G_m^{\wt}))$ has codimension at least
one in $V(G_1^{\wt}\klk G_m^{\wt})\subset \cfq{\!}^r$.
\end{enumerate}

We briefly discuss hypotheses $({\sf H}_1)$--$({\sf H}_6)$.
Hypothesis $({\sf H}_1)$--$({\sf H}_2)$ merely state that $W$ has
the expected dimension $r-m$ and it is smooth. These conditions are
satisfied for any sequence $G_1,\ldots,G_m\in\fq[\bfs A_k]$ as above
with general coefficients (see, e.g., \cite{Benoist12} or
\cite{GaMa18}). Hypothesis $({\sf H}_3)$ requires that
$G_1,\ldots,G_m$ behave properly ``at infinity'', which is also the
case for general $G_1,\ldots,G_m$. Hypotheses $({\sf H}_4)$--$({\sf
H}_5)$ require that ``most'' of the polynomials of $\mathcal{A}$ are
square--free, and among those which are not, only ``few'' of them
have roots with high multiplicity or several multiple roots. As we
are looking for criteria for uniform distribution, namely families
which behave as the whole set $\fq[T]_r$, it is clear that such a
behavior is to be expected. Further, it is required that ``few''
polynomials in the family under consideration have $0$ as a multiple
root, which is a common requirement for uniformly distributed
families (see, e.g., \cite{Cohen72}). Finally, hypothesis $({\sf
H}_6)$ requires that the discriminant locus at infinity is not too
large. We provide significant examples of families of polynomials
satisfying hypotheses $({\sf H}_1)$--$({\sf H}_6)$, which include in
particular the classical of polynomials with prescribed
coefficients.

Our main result shows that any family $\mathcal{A}$ satisfying
hypotheses $({\sf H}_1)$--$({\sf H}_6)$ is uniformly distributed in
the sense of Cohen, and provides explicit estimates on the number
$|\mathcal{A}_{\bfs\lambda}|$ of elements of $\mathcal{A}$ with
factorization pattern $\bfs\lambda$. In fact, we have the following
result (see Theorem \ref{theorem: estimate fact patterns} for a more
precise statement).
\begin{theorem}\label{theorem: estimate fact patterns into}
For $m<r$ and $\bfs\lambda$ a factorization pattern, we have
$$
\big||\mathcal{A}_{\bfs\lambda}|
-\mathcal{T}(\bfs\lambda)\,q^{r-m}\big|\le
q^{r-m-1}\big(\mathcal{T}(\bfs\lambda) (D\delta\, q^{\frac{1}{2}}+14
D^2 \delta^2+r^2\delta)+r^2\delta\big),
$$
where $\delta:=\prod_{i=1}^m \wt(G_i)$ and
$D:=\sum_{i=1}^m(\wt(G_i)-1)$.
\end{theorem}

Our methodology differs significantly from that of \cite{Cohen70}
and \cite{Cohen72}, as we express $|\mathcal{A}_{\bfs\lambda}|$ in
terms of the set of common of $\fq$--rational zeros of certain
symmetric multivariate polynomials defined over $\fq$. This allows
us to establish several facts concerning the geometry of the set of
zeros of such polynomials over $\cfq$. Combining these results with
estimates on the number of $\fq$--rational points of such set of
zeros (see, e.g., \cite{CaMa06} or \cite{CaMaPr15}), we obtain our
main results.

Then we consider the average--case complexity of the classical
factorization algorithm restricted to $\mathcal{A}$. This algorithm
works in four main steps. First it performs an ``elimination of
repeated factors''. Then it computes a (partial) factorization of
the result of the first step by splitting its irreducible factors
according to their degree (this is called the distinct--degree
factorization). The third step factorizes each of the factors
computed in the second step (the equal--degree factorization).
Finally, the fourth step consists of the factorization of the
repeated factors left aside in the first step (factorization of
repeated factors). The following result summarizes our estimates on
the average--case complexity of each of these steps (see Theorems
\ref{costo paso 1}, \ref{average de DDF}, \ref{average de EDF} and
\ref{costo en promedio cuarto paso} for more precise statements).
\begin{theorem}\label{th: average factorization}
Let $\delta_{\bfs G}:=\deg G_1\cdots \deg G_m$. Denote by
$E[\mathcal {X} _1] $, $E[\mathcal {X} _2]$, $E[\mathcal {X} _3]$
and $E[\mathcal {X} _4]$ the average cost on $\mathcal{A}$ of the
steps of elimination of repeated factors, distinct--degree
factorization, equal--degree factorization and factorization of
repeated factors.

For $q > 15\delta_{\bfs G}^{13/3}$, assuming that fast
multiplication is used, we have
\begin{align*}
 E [\mathcal {X} _1] \leq c \, \mathcal{U} (r) +o(1),\quad\!
&E[\mathcal{X}_2] \le\xi\,(2\, \tau_1 \lambda(q)+\tau_1+\tau_2\log
r)
\,M(r)\,(r+1)\big(1+o(1)\big),\\[1ex]
E[\mathcal{X}_3] \leq \tau\, M(r) \log q\,(1+o(1)),\quad\!
&E[\mathcal{X}_4] \leq \tau_1 M(r)(1+o(1)),
\end{align*}
where $M(r):=r\log r\log\log r$ is the fast--multiplication time
function, $\mathcal{U}(r):=M(r)\log r$ is the $\gcd$ time function,
$\lambda (q)$ is the number of multiplications required to compute
$q$--th powers using repeated squaring, $\xi\sim 0.62432945\dots$ is
the Golomb constant, and $c$, $\tau_1$, $\tau_2$ and $\tau$ are
constants independent of $q$ and $r$.
\end{theorem}

Here, the $o(1)$ terms goes to zero as $q$ tends to infinity, for
fixed $r$ and $\deg G_1,\ldots,\deg G_m$. See Theorems \ref{costo
paso 1}, \ref{average de DDF}, \ref{average de EDF} and \ref{costo
en promedio cuarto paso} for explicit expressions of these terms.

This result significantly strengthens the conclusions of the
average--case analysis of \cite{FlGoPa01}, in that it shows that
such conclusions are not only applicable to the whole set $\fq[T]_r$
of monic polynomials of degree $r$, but to any family
$\mathcal{A}\subset\fq[T]_r$ satisfying hypotheses $({\sf
H}_1)$--$({\sf H}_6)$. Further, our estimates improve by roughly a
factor $r$ those of \cite{FlGoPa01}, up to logarithmic terms, due to
the fact that we consider fast multiplication of polynomials,
instead of the classical polynomial multiplication considered
\cite{FlGoPa01}.

The paper is organized as follows. In Section \ref{sec: notions of
algebraic geometry} we collect the notions of algebraic geometry we
use. In Section \ref{sec: estimate of |A|} we obtain a lower bound
on the number of elements of the family $\mathcal{A}$ under
consideration. Section \ref{the factorization patterns} is devoted
to describe our algebraic--geometry approach to the distribution of
factorization patterns and to prove Theorem \ref{theorem: estimate
fact patterns into}. In Section \ref{sec: examples} we exhibit
examples of linear and nonlinear families of polynomials satisfying
hypotheses $({\sf H}_1)$--$({\sf H}_6)$. Finally, in Section
\ref{sec: average-case complexity of factorization} we perform the
average--case analysis of the classical polynomial factorization
restricted to $\mathcal{A}$, showing Theorem \ref{th: average
factorization}.

%

\section{Basic notions of algebraic geometry}
\label{sec: notions of algebraic geometry}
In this section we collect the basic definitions and facts of
algebraic geometry that we need in the sequel. We use standard
notions and notations which can be found in, e.g., \cite{Kunz85},
\cite{Shafarevich94}.

Let $\K$ be any of the fields $\fq$ or $\cfq$. We denote by $\A^r$
the affine $r$--dimensional space $\cfq{\!}^{r}$ and by $\Pp^r$ the
projective $r$--dimensional space over $\cfq{\!}^{r+1}$. Both spaces
are endowed with their respective Zariski topologies over $\K$, for
which a closed set is the zero locus of a set of polynomials of
$\K[X_1,\ldots, X_r]$, or of a set of homogeneous polynomials of
$\K[X_0,\ldots, X_r]$.

A subset $V\subset \Pp^r$ is a {\em projective variety defined over}
$\K$ (or a projective $\K$--variety for short) if it is the set of
common zeros in $\Pp^r$ of homogeneous polynomials $F_1,\ldots, F_m
\in\K[X_0 ,\ldots, X_r]$. Correspondingly, an {\em affine variety of
$\A^r$ defined over} $\K$ (or an affine $\K$--variety) is the set of
common zeros in $\A^r$ of polynomials $F_1,\ldots, F_{m} \in
\K[X_1,\ldots, X_r]$. We think a projective or affine $\K$--variety
to be equipped with the induced Zariski topology. We shall denote by
$\{F_1=0,\ldots, F_m=0\}$ or $V(F_1,\ldots,F_m)$ the affine or
projective $\K$--variety consisting of the common zeros of
$F_1,\ldots, F_m$.

In the remaining part of this section, unless otherwise stated, all
results referring to varieties in general should be understood as
valid for both projective and affine varieties.

A $\K$--variety $V$ is {\em irreducible} if it cannot be expressed
as a finite union of proper $\K$--subvarieties of $V$. Further, $V$
is {\em absolutely irreducible} if it is $\cfq$--irreducible as a
$\cfq$--variety. Any $\K$--variety $V$ can be expressed as an
irredundant union $V=\mathcal{C}_1\cup \cdots\cup\mathcal{C}_s$ of
irreducible (absolutely irreducible) $\K$--varieties, unique up to
reordering, called the {\em irreducible} ({\em absolutely
irreducible}) $\K$--{\em components} of $V$.

For a $\K$--variety $V$ contained in $\Pp^r$ or $\A^r$, its {\em
defining ideal} $I(V)$ is the set of polynomials of $\K[X_0,\ldots,
X_r]$, or of $\K[X_1,\ldots, X_r]$, vanishing on $V$. The {\em
coordinate ring} $\K[V]$ of $V$ is the quotient ring
$\K[X_0,\ldots,X_r]/I(V)$ or $\K[X_1,\ldots,X_r]/I(V)$. The {\em
dimension} $\dim V$ of $V$ is the length $n$ of a longest chain
$V_0\varsubsetneq V_1 \varsubsetneq\cdots \varsubsetneq V_n$ of
nonempty irreducible $\K$--varieties contained in $V$. 
We say that $V$ has {\em pure dimension} $n$ if every irreducible
$\K$--component of $V$ has dimension $n$. A $\K$--variety of $\Pp^r$
or $\A^r$ of pure dimension $r-1$ is called a $\K$--{\em
hypersurface}. A $\K$--hypersurface of $\Pp^r$ (or $\A^r$) can also
be described as the set of zeros of a single nonzero polynomial of
$\K[X_0,\ldots, X_r]$ (or of $\K[X_1,\ldots, X_r]$).

The {\em degree} $\deg V$ of an irreducible $\K$--variety $V$ is the
maximum of $|V\cap L|$, considering all the linear spaces $L$ of
codimension $\dim V$ such that $|V\cap L|<\infty$. More generally,
following \cite{Heintz83} (see also \cite{Fulton84}), if
$V=\mathcal{C}_1\cup\cdots\cup \mathcal{C}_s$ is the decomposition
of $V$ into irreducible $\K$--components, we define the degree of
$V$ as
$$\deg V:=\sum_{i=1}^s\deg \mathcal{C}_i.$$
The degree of a $\K$--hypersurface $V$ is the degree of a polynomial
of minimal degree defining $V$. 
%
We shall use the following {\em B\'ezout inequality} (see
\cite{Heintz83}, \cite{Fulton84}, \cite{Vogel84}): if $V$ and $W$
are $\K$--varieties of the same ambient space, then
\begin{equation}\label{eq: Bezout}
\deg (V\cap W)\le \deg V \cdot \deg W.
\end{equation}

Let $V\subset\A^r$ be a $\K$--variety, $I(V)\subset \K[X_1,\ldots,
X_r]$ its defining ideal and $x$ a point of $V$. The {\em dimension}
$\dim_xV$ {\em of} $V$ {\em at} $x$ is the maximum of the dimensions
of the irreducible $\K$--components of $V$ containing $x$. If
$I(V)=(F_1,\ldots, F_m)$, the {\em tangent space} $\mathcal{T}_xV$
to $V$ at $x$ is the kernel of the Jacobian matrix $(\partial
F_i/\partial X_j)_{1\le i\le m,1\le j\le r}(x)$ of $F_1,\ldots, F_m$
with respect to $X_1,\ldots, X_r$ at $x$. We have
$\dim\mathcal{T}_xV\ge \dim_xV$ (see, e.g., \cite[page
94]{Shafarevich94}). The point $x$ is {\em regular} if
$\dim\mathcal{T}_xV=\dim_xV$; otherwise, $x$ is called {\em
singular}. The set of singular points of $V$ is the {\em singular
locus} $\mathrm{Sing}(V)$ of $V$; it is a closed $\K$--subvariety of
$V$. A variety is called {\em nonsingular} if its singular locus is
empty. For projective varieties, the concepts of tangent space,
regular and singular point can be defined by considering an affine
neighborhood of the point under consideration.

Let $V$ and $W$ be irreducible affine $\K$--varieties of the same
dimension and $f:V\to W$ a regular map with $\overline{f(V)}=W$,
where $\overline{f(V)}$ denotes the closure of $f(V)$ with respect
to the Zariski topology of $W$. Such a map is called {\em dominant}.
Then $f$ induces a ring extension $\K[W]\hookrightarrow \K[V]$ by
composition with $f$. We say that the dominant map $f$ is {\em
finite} if this extension is integral, namely each element
$\eta\in\K[V]$ satisfies a monic equation with coefficients in
$\K[W]$. A dominant finite
morphism is necessarily closed. Another fact 
we shall use is that the preimage $f^{-1}(S)$ of an irreducible
closed subset $S\subset W$ under a dominant finite morphism $f$ is
of pure dimension $\dim S$ (see, e.g., \cite[\S 4.2,
Proposition]{Danilov94}).
%
%
\subsection{Rational points}
Let $\Pp^r(\fq)$ be the $r$--dimensional projective space over $\fq$
and $\A^r(\fq)$ the $r$--dimensional $\fq$--vector space $\fq^n$.
For a projective variety $V\subset\Pp^r$ or an affine variety
$V\subset\A^r$, we denote by $V(\fq)$ the set of $\fq$--rational
points of $V$, namely $V(\fq):=V\cap \Pp^r(\fq)$ in the projective
case and $V(\fq):=V\cap \A^r(\fq)$ in the affine case. For an affine
variety $V$ of dimension $n$ and degree $\delta$, we have the
following bound (see, e.g., \cite[Lemma 2.1]{CaMa06}):
\begin{equation}\label{eq: upper bound -- affine gral}
   |V(\fq)|\leq \delta\, q^n.
\end{equation}
On the other hand, if $V$ is a projective variety of dimension $n$
and degree $\delta$, we have the following bound (see
\cite[Proposition 12.1]{GhLa02a} or \cite[Proposition 3.1]{CaMa07};
see \cite{LaRo15} for more precise upper bounds):
 \begin{equation}\label{eq: upper bound -- projective gral}
|V(\fq)|\leq \delta\, p_n,
\end{equation}
where $p_n:=q^n+q^{n-1}+\cdots+q+1=|\Pp^n(\fq)|$.
%
%
\subsection{Complete intersections}\label{subsec: complete intersections}
Elements $F_1,\ldots, F_m$ in $\mathbb{K}[X_1,\ldots,X_r]$ or
$\mathbb{K}[X_0,\ldots,X_r]$ form a \emph{regular sequence} if $F_1$
is nonzero and no $F_i$ is zero or a zero divisor in the quotient
ring $\mathbb{K}[X_1,\ldots,X_r]/ (F_1,\ldots,F_{i-1})$ or
$\mathbb{K}[X_0,\ldots,X_r]/ (F_1,\ldots,F_{i-1})$ for $2\leq i \leq
m$. In that case, the (affine or projective) $\mathbb{K}$--variety
$V:=V(F_1,\ldots,F_{r-n})$ is called a {\em set--theoretic complete
intersection}. We remark that $V$ is necessarily of pure dimension
$r-m$. Further, $V$ is called an {\em (ideal--theoretic) complete
intersection} if its ideal $I(V)$ over $\K$ can be generated by $m$
polynomials. We shall frequently use the following criterion to
prove that a variety is a complete intersection (see, e.g.,
\cite[Theorem 18.15]{Eisenbud95}).
\begin{theorem}\label{theorem: eisenbud 18.15}
Let $F_1,\ldots,F_m\in\mathbb{K}[X_1,\ldots,X_r]$ be polynomials
which form a regular sequence and let
$V:=V(F_1,\ldots,F_m)\subset\A^r$. Denote by
$(\partial\bfs{F}/\partial \bfs{X})$ the Jacobian matrix of
$F_1,\ldots,F_m$ with respect to $X_1,\ldots,X_r$. If the subvariety
of $V$ defined by the set of common zeros of the maximal minors of
$(\partial\bfs{F}/\partial \bfs{X})$ has codimension at least one in
$V$, then $F_1,\ldots,F_m$ define a radical ideal. In particular,
$V$ is a complete intersection.
\end{theorem}

If $V\subset\Pp^r$ is a complete intersection defined over $\K$ of
dimension $r-m$, and $F_1 ,\ldots, F_m$ is a system of homogeneous
generators of $I(V)$, the degrees $d_1,\ldots, d_m$ depend only on
$V$ and not on the system of generators. Arranging the $d_i$ in such
a way that $d_1\geq d_2 \geq \cdots \geq d_m$, we call $(d_1,\ldots,
d_m)$ the {\em multidegree} of
$V$. In this case, a stronger version of 
\eqref{eq: Bezout} holds, called the {\em B\'ezout theorem} (see,
e.g., \cite[Theorem 18.3]{Harris92}):
\begin{equation}\label{eq: Bezout theorem}
\deg V=d_1\cdots d_m.
\end{equation}

A complete intersection $V$ is called {\em normal} if it is {\em
regular in codimension 1}, that is, the singular locus
$\mathrm{Sing}(V)$ of $V$ has codimension at least $2$ in $V$,
namely $\dim V-\dim \mathrm{Sing}(V)\ge 2$ (actually, normality is a
general notion that agrees on complete intersections with the one we
define here). A fundamental result for projective complete
intersections is the Hartshorne connectedness theorem (see, e.g.,
\cite[Theorem VI.4.2]{Kunz85}): if $V\subset\Pp^r$ is a complete
intersection defined over $\K$ and $W\subset V$ is any
$\K$--subvariety of codimension at least 2, then $V\setminus W$ is
connected in the Zariski topology of $\Pp^r$ over $\K$. Applying the
Hartshorne connectedness theorem with $W:=\mathrm{Sing}(V)$, one
deduces the following result.
\begin{theorem}\label{theorem: normal complete int implies irred}
If $V\subset\Pp^r$ is a normal complete intersection, then $V$ is
absolutely irreducible.
\end{theorem}
%
%
%
\section{Estimates on the number of elements of $\mathcal{A}$}
\label{sec: estimate of |A|}
Let  $X_1 \klk X_r$ be indeterminates over $\cfq$. Denote by $\Pi_1
\klk \Pi_{r}$ the elementary symmetric polynomials of $\fq[X_1 \klk
X_r]$. Observe that $f:=T^r+ a_{r-1}T^{r-1}+\ldots +a_0 \in
\mathcal{A}$ if and only if there exists $\bfs x\in \A^r$ such that
$a_j=(-1)^{r-j} \Pi_{r-j}(\bfs x)$ for $0 \leq j \leq r-1$ and
$$R_i:=G_i(-\Pi_1(\bfs x) \klk {(-1)^{r-k-1}\Pi_{r-k-1}(\bfs x)},
{(-1)^{r-k+1}\Pi_{r-k+1}(\bfs x)} \klk
(-1)^{r}\Pi_{r}(\bfs x))=0$$
for $1\leq i \leq m$. Thus, we associate to $\mathcal{A}$ the
polynomials $R_1\klk R_m\in\fq[X_1 \klk X_r]$ and the variety $V
\subset \A^r$ defined by $R_1 \klk R_m$.

Our estimates on the distribution of factorization patterns in
$\mathcal{A}$ require asymptotically--tight estimates on the number
of $\fq$--rational points of $V$, and for the average--case analysis
of the classical factorization algorithm restricted to $\mathcal{A}$
we need asymptotically--tight lower bounds on the number of elements
of $\mathcal{A}$. For this purpose, we shall prove several facts
concerning the geometry of the affine varieties $V$ and $W$.

Hypothesis $({\sf H}_1)$ implies that $W$ is a set--theoretic
complete intersection of dimension $r-m$. Furthermore, by  $({\sf
H}_2)$ it follows that the subvariety of $W$ defined by the set of
common zeros of the maximal minors of $(\partial\bfs{G}/\partial
\bfs{A_k})$ has codimension at least one in $W$.  Applying Theorem
\ref{theorem: eisenbud 18.15} we deduce the following result.
\begin{lemma}
$W\subset \A^r$ is a complete intersection of dimension $r-m$.
\end{lemma}

%
Consider the following surjective morphism of affine
$\fq$--varieties:
\begin{align}\label{definition of Pir}
{\bfs\Pi^r}: \A^r & \rightarrow  \A^{r}
\\ \notag
\bfs{x} & \mapsto  (-\Pi_1(\bfs{x}),\ldots,(-1)^{r}\Pi_{r}(\bfs{x})).
\end{align}
It is easy to see that  $\bfs\Pi^r$ is a finite, dominant morphism
with $\bfs \Pi^r(V)=W$. 
By hypothesis $({\sf H}_1)$ the variety
$W^j:=V(G_1,\ldots,G_j)\subset \A^r$ has pure dimension $r-j$ for
$1\leq j \leq m$. This implies that
$V^j:=(\bfs\Pi^r)^{-1}(W^j)=V(R_1\klk R_j)$ has pure dimension $r-j$
for $1\leq j \leq m$. We conclude that $R_1\klk R_{m}$ form a
regular sequence of $\fq[X_1\klk X_r]$, namely we have the following
result.
\begin{lemma}\label{lemma: V is complete inters}
$V$ is a set--theoretic complete intersection of dimension $r-m$.
\end{lemma}

Next we study the singular locus of $V$. 
%
%
For this purpose, we make some remarks concerning the Jacobian
matrix of $(\partial\bfs\Pi^r/\partial{\bfs X})$ of $\bfs\Pi^r$ with
respect to $X_1,\dots,X_r$. Denote by $A_r$ the $(r\times
r)$--Vandermonde matrix
$$
A_r:=(X_j^{i-1})_{1\leq i,j\leq r}.
$$
Taking into account the following well--known identities 
(see, e.g., \cite{LaPr02}):
$$
\frac{\partial \Pi_i}{\partial X_{j}}= \Pi_{i-1}-X_{j} \Pi_{i-2} +
X_{j}^2 \Pi_{i-3} +\cdots+ (-1)^{i-1} X_{j}^{i-1}\quad (1\leq i,j
\leq r),
$$
we conclude that $(\partial\bfs\Pi^r/\partial{\bfs X})$ can be
factored as
\begin{equation} \label{eq: factorization Jacobian elem sim pols}
\left(\frac{\partial {\bfs\Pi^r}}{\partial \bfs X}\right):=B_r\cdot
A_r
:=
\left(
\begin{array}{ccccc}
-1 & \ 0 & 0 &  \dots & 0
\\
\ \ \Pi_1 &  -1 & 0 &  &
\\
-\Pi_2 & \ \ \Pi_1 & -1 & \ddots & \vdots
\\
\vdots &\vdots  & \vdots & \ddots & 0
\\
(-1)^{r}\Pi_{r-1} & (-1)^{r-1}\Pi_{r-2} & (-1)^{r-2}\Pi_{r-3} & \cdots &\!\! -1
\end{array}
\!\!\right)
\cdot
A_r.
\end{equation}
Since $\det B_r=(-1)^r$, we see that
$$\det \left(\frac{\partial {\bfs\Pi^r}}{\partial \bfs X}\right)
=(-1)^{r} \prod_{1\le i < j\le r}(X_j-X_i).$$

A critical point in the study of the singular locus of $V$ is the
analysis of the zero locus of the $(r-1)\times(r-1)$ minors of
$\left({\partial {\bfs\Pi^r}}/{\partial \bfs X}\right)$. For this
purpose, we have the following result.
\begin{proposition}\label{prop: determinant minors order r-1}
Fix $k$ with $ 0 \leq k \leq r-1$ as in the introduction and $l$
with $ 1 \leq l  \leq r$. Denote by $M_{r-k,l}$ the $(r-1)\times
(r-1)$--matrix obtained by deleting the row $r-k$ and the column $l$
of $(\partial{\bfs\Pi^r}/\partial{\bfs X})$. Then
\begin{equation} \label{det maximal menor}
\det M_{r-k,l}=(-1)^{r-k-1}\Delta_l\cdot X_l^k,
\end{equation}
where $\Delta_l:=\prod_{1\le i < j\le r, \, \,   i,j \neq
l}(X_j-X_i).$
\end{proposition}
\begin{proof}
According to the factorization (\ref{eq: factorization Jacobian elem
    sim pols}), we have
$$M_{r-k,l}=B_{r}^{r-k} \cdot A_{r}^l,$$
where $B_{r}^{r-k}$ is the $(r-1)\times r$--submatrix of $B_r$
obtained by deleting its $(r-k)$th row and $A_{r}^l$ is the $r
\times (r-1)$--submatrix of $A_{r}$ obtained by deleting its $l$th
column. By the Cauchy--Binet formula, it follows that
$$\det M_{r-k,l}=\sum_{j=1}^r \det B_r^{r-k,j}\cdot \det A_r^{j,l},$$
where $B_{r}^{r-k,j}$ is the $(r-1)\times (r-1)$--matrix obtained by
removing the $j$th column of $B_r^{r-k}$ and $A_r^{j,l}$ is the
$(r-1)\times (r-1)$--matrix obtained by removing the $j$th row of
$A_r^l$.

From \cite[Lemma 2.1]{Ernst00} we deduce that
\begin{equation}\label{det A}
\det A_r^{j,l}= \Delta_l \cdot \Pi_{r-j}^*,
\end{equation}
where $\Pi_{r-j}^*=\Pi_{r-j}(X_1 \klk X_{l-1},X_{l+1} \klk X_r).$

Next, we obtain an explicit expression of $\det B_{r}^{r-k,j}$ for
$1\le j \le r$. Observe that $B_r^{r-k}$ has a block structure:
\begin{equation}\label{formula Br-1,l}
B_{r}^{r-k}:=\left(
                 \begin{array}{cc}
                  B_{r-k-1}  &  \mathbf{0}\\
                  *  &  \mathcal{T}_k^*\\
                 \end{array}
               \right),
\end{equation}
where $B_{r-k-1}$ is the $(r-k-1)\times(r-k-1)$ principal submatrix
of $B_r$ consisting on its first $r-k-1$ rows and columns and
$\mathcal{T}_k^*$ is the  $k\times(k+1)$--matrix
$$
\mathcal{T}_k^*:=\left(
\begin{array}{ccccccc}
\Pi_1 &\!\! -1 & 0 &  \dots & 0 & 0
\\
-\Pi_2 & \!\! \ddots & \ddots&   &\vdots & \vdots
\\
\vdots &\!\!\ddots  & \ddots & \ddots & 0 & 0
\\
\vdots &  & \ddots & \ddots &\!\! -1 &0
\\
(-1)^{k+1}\Pi_k \!\!&\!\! \dots & \dots & -\Pi_2& \Pi_1 &\!\!-1
\end{array}
\right).
$$
From \eqref{formula Br-1,l} we readily deduce that
\begin{equation}\label{det B}
\det B_r^{r-k,j}=\left\{\begin{array}{cl}
                   0 & \textrm{for }1\le j \le r-k-1, \\
                   (-1)^{r-1} & \textrm{for }j=r-k, \\
                   (-1)^{r-i-1} \det \mathcal{T}_i & \textrm{for
                   }j=r-k+i,\,1\le i\le k,
                 \end{array}\right.
\end{equation}
where $\mathcal{T}_i$ is the following $i \times i$
Toeplitz--Hessenberg matrix:
$$
\mathcal{T}_i:=\left(
\begin{array}{ccccc}
\Pi_1 & -1 & 0 &  \dots & 0
\\
-\Pi_2\ \ &  \ddots & \ddots&  & \vdots
\\
\vdots &\ddots  & \ddots & \ddots & 0
\\
\vdots &  & \ddots & \ddots &\!\! -1\,
\\
(-1)^{i+1}\Pi_{i} & \dots & \dots & -\Pi_2&\!\! \Pi_1
\end{array}
\!\right).
$$

By the Trudi formula (see \cite[Ch. VII]{Muir60}; see also
\cite[Theorem 1]{Merca13}) we deduce the following identity (see
\cite[Section 4]{Merca13}):
$$\det \mathcal{T}_i=H_i,$$
where $H_i:=H_i(X_1 \klk X_r)$ is the $i$th complete homogeneous
symmetric function. Therefore, combining \eqref{det A} and
\eqref{det B} we conclude that
\begin{align*}
\det M_{r-k,l} = \Delta_l\sum_{j=r-k}^r \det B_r^{r-k,j} \cdot
\Pi_{r-j}^* &=\Delta_l  \sum_{i=0}^k \det B_r^{r-k,\,i+r-k} \cdot
\Pi_{k-i}^* \\ \notag & =\Delta_l \sum_{i=0}^k (-1)^{r-i-1} H_i
\cdot \Pi^*_{k-i}.
\end{align*}

We claim that
\begin{equation}\label{Sk}
S(k):=\sum_{i=0}^k (-1)^{r-i-1} H_i \cdot
\Pi^*_{k-i}=(-1)^{r-k-1}X_l^k, \quad k=0 \klk r-1.
\end{equation}
We prove the claim arguing by induction on $k$. Since $H_0=
\Pi_0^*=1$, the case $k = 0$ follows immediately. Assume now that
\eqref{Sk} holds for $k-1$ with $k>0$, namely
\begin{equation}\label{HI}
(-1)^{r-1}\sum_{i=0}^{k-1}(-1)^i H_i\cdot
\Pi_{k-1-i}^*=(-1)^{r-k}X_l^{k-1}.
\end{equation}
%
It is well known that (see, e.g., \cite[$7$.\S $1$, Exercise
$10$]{CoLiOs92})
$$\sum_{i=0}^k (-1)^i H_i\cdot \Pi_{k-i}=0.$$
Since $\Pi_{k-i}^*=\Pi_{k-i}(X_1 \klk X_{l-1},X_{l+1} \klk X_r)$, we
deduce that $\Pi_{k-i}= X_l \cdot \Pi_{k-i-1}^*+ \Pi_{k-i}^*$. As a
consequence, it follows that
$$
\sum_{i=0}^k(-1)^i H_i \cdot \Pi_{k-i}^*=X_l \sum_{i=0}^{k-1}
(-1)^{i-1} H_i \cdot \Pi_{k-i-1}^*.
$$
Combining this identity and the inductive hypothesis \eqref{HI}, we
conclude that
$$
S(k)=-X_l  \sum_{i=0}^{k-1} (-1)^{r-i-1} H_i \cdot \Pi_{k-i-1}^* =-
X_l\, (-1)^{r-k} X_l^{k-1} =(-1)^{r-k-1}X_l^k.
$$
This concludes the proof of the proposition.
\end{proof}

Denote by $(\partial \bfs{R}/\partial \bfs{X}):=(\partial
R_i/\partial X_j)_{1\le i\le m,1\leq j \le r}$ the Jacobian matrix
of $R_{1}\klk R_{m}$ with respect to  $X_1\klk X_r$.
\begin{theorem}\label{theorem: dimension singular locus V}
The set of $\bfs{x}\in V$ for which $(\partial \bfs{R}/\partial
\bfs{X})(\bfs{x})$ does not have full rank, has codimension at least
$2$. In particular, the singular locus $\Sigma$ of $V$ has
codimension at least $2$.
\end{theorem}
\begin{proof}
By the chain rule, we have the equality
    $$
    \left(\frac{\partial \bfs{R}}{\partial
        \bfs{X}}\right)=\left(\frac{\partial \bfs{G}}{\partial
        \bfs{A}}\circ\bfs\Pi\right)\cdot
    \left(\frac{\partial\bfs\Pi}{\partial \bfs{X}}\right),
    $$
where $\bfs \Pi:=(-\Pi_1,\ldots,
(-1)^{r-k-1}\Pi_{r-k-1},(-1)^{r-k+1}\Pi_{r-k+1},\ldots,(-1)^{r}\Pi_{r})$.
Fix a point $\bfs{x}:=(x_1,\ldots,x_r)\in V$ such that $(\partial
\bfs{R}/\partial \bfs{X})(\bfs{x})$ does not have full rank, and let
$\bfs{v}\in\A^{m}$ be a nonzero element in the left kernel of
$(\partial \bfs{R}/\partial \bfs{X})(\bfs{x})$. We have
    $$\bfs{0}=
    \bfs{v}\cdot \left(\frac{\partial \bfs{R}}{\partial
        \bfs{X}}\right)(\bfs{x})=\bfs{v}\cdot\left(\frac{\partial
        \bfs{G}}{\partial \bfs{A}}\right)\big(\bfs\Pi(\bfs{x})\big)\cdot
    \left(\frac{\partial\bfs\Pi}{\partial \bfs{X}}\right)(\bfs{x}).$$
Since by hypothesis ($ {\sf H}_2$) the Jacobian matrix
$(\partial\bfs{G}/\partial \bfs{A})\big(\bfs\Pi(\bfs{x})\big)$ has
full rank, we see that $\bfs{w}:=\bfs{v}\cdot\left({\partial
\bfs{G}}/{\partial \bfs{A}}\right)\big(\bfs
\Pi(\bfs{x})\big)\in\A^{r-1}$ is nonzero. As $\bfs{w}\cdot
\left({\partial\bfs\Pi}/{\partial \bfs{X}}\right)(\bfs{x})=
\bfs{0}$, all the maximal minors of
$\left({\partial\bfs\Pi}/{\partial \bfs{X}}\right)(\bfs{x})$ must be
zero. These minors are the determinants $\det M_{r-k,l}(\bfs{x})$,
where $M_{r-k,l}$ are the matrices of Proposition \ref{prop:
determinant minors order r-1}.

Since $\det M_{r-k,l}(\bfs x)=0$ for $1\leq l\leq r$, Proposition
\ref{prop: determinant minors order r-1} implies
$$x_i^k\Delta_i(x)=x_j^{k}\Delta_{j}(x)=0\quad (1\leq i <j\leq r).$$
It follows that $\bfs x$ cannot have its $r$ coordinates pairwise
distinct. As a consequence, either $\bfs x$ has $r-1$
pairwise--distinct coordinates, one of them being equal to zero, or
$\bfs x$ has at most $r-2$ pairwise--distinct coordinates. Let
$$g:=(T-x_1) \cdots (T-x_r)=T^r -\Pi_1(\bfs x)T^{r-1}+ \cdots
+(-1)^{r} \Pi_{r}(\bfs x).$$
Observe that $\bfs \Pi^r(\bfs x) \in W$. If there is a coordinate
$x_i=0$, then the constant coefficient of $g$ is zero. On the other
hand, if $\bfs x$ has at most $r-2$ pairwise--distinct coordinates,
then there exist $i,j,l,h \in \{1 \klk r\}$ with $i<j, l<h$ and
$\{i,j\} \cap \{k,l\}= \emptyset$ such that  $x_i=x_j$ and
$x_h=x_l$. If $x_i\not= x_h$, then $g$ has two distinct multiple
roots, while in the case $x_i=x_h$, $g$ has a root of multiplicity
at least $4$. In both cases $g$ and $g'$ have a common factor of
degree at least 2, which implies that
$$ \mathrm{Disc}(g)=0, \, \, \mathrm{Subdisc}(g)=0,$$
namely $g\in\mathcal{S}_1(W)$. In either case, $\bfs \Pi^r(\bfs x)
\in (A_0\cdot\mathcal{S}_1)(W)$. According to (${\sf H}_4$) and
(${\sf H}_5$), $(A_0\cdot\mathcal{S}_1)(W)$ has codimension at least
$2$ in $W$. Since $\bfs \Pi^r$ is a finite morphism, we have that
$(\bfs \Pi ^r)^{-1}\big((A_0\cdot\mathcal{S}_1)(W)\big)$ has
codimension at least $2$ in $V$. In particular, the set of points
$\bfs{x}\in V$ with $\mathrm{rank}(\partial \bfs{R}/\partial
\bfs{X})(\bfs{x})< m$ is contained in a subvariety of codimension
$2$ of $V$.

Now let $\bfs x$ be an arbitrary point of $\Sigma$. By Lemma
\ref{lemma: V is complete inters} we have $\dim T_{\bfs x}V >r-m$.
It follows that $\mathrm{rank}(\partial \bfs{R}/\partial
\bfs{X})(\bfs{x})<m$, for otherwise we would have $\dim T_{\bfs x}V
\leq r-m$, contradicting the hypothesis that $\bfs x$ is a singular
point of $V$. Therefore, from the first assertion the theorem
follows.
\end{proof}

From Lemma \ref{lemma: V is complete inters} and  Theorem
\ref{theorem: dimension singular locus V} we obtain further
consequences concerning the polynomials $R_i$ and the variety $V$.
Theorem \ref{theorem: dimension singular locus V} shows in
particular that the set of points $\bfs{x}\in V$ for which
$(\partial \bfs{R}/\partial \bfs{X})(\bfs{x})$ does not have full
rank has codimension at least one in $V$. Since $R_{1}\klk R_{m}$
form a regular sequence, by Theorem \ref{theorem: eisenbud 18.15} we
conclude that $R_{1}\klk R_{m}$ define a radical ideal of
$\fq[X_1\klk X_r]$, and thus $V$ is a complete intersection.
In other words, we have the following result.
\begin{corollary}\label{coro: radicality and degree Vr}
$R_{1}\klk R_{m}$ define a radical ideal and $V$ is a complete
intersection. 
\end{corollary}
%
%
\subsection{The geometry of the projective closure}
Consider the embedding of $\A^r$ into the projective space $\Pp^r$
defined by the mapping $(x_1,\ldots, x_r)\mapsto(1:x_1:\dots:x_r)$.
The closure $\mathrm{pcl}(V)\subset\Pp^r$ of the image of $V$ under
this embedding in the Zariski topology of $\Pp^r$ is called the
projective closure of $V$. The points of $\mathrm{pcl}(V)$ lying in
the hyperplane $\{X_0=0\}$ are called the points of
$\mathrm{pcl}(V)$ at infinity.

Denote by $F^h\in\fq[X_0,\ldots,X_r]$ the homogenization of each
$F\in\fq[X_1,\ldots,X_r]$, and let $(R_{1}\klk R_{m})^h$ be the
ideal generated by all the polynomials $F^h$ with $F\in (R_{1}\klk
R_{m})$. We have  that $(R_{1}\klk R_{m})^h$ is radical because
$(R_{1}\klk R_{m})$ is a radical ideal (see, e.g., \cite[\S I.5,
Exercise 6]{Kunz85}). It is well known that $\mathrm{pcl}(V)$ is the
$\fq$--variety of $\mathbb{P}^r$ defined by $(R_{1}\klk R_{m})^h$
(see, e.g., \cite[\S I.5, Exercise 6]{Kunz85}). Furthermore,
$\mathrm{pcl} (V)$ has pure dimension $r-m$ (see, e.g.,
\cite[Propositions I.5.17 and II.4.1]{Kunz85}) and degree equal to
$\deg V$ (see, e.g., \cite[Proposition 1.11]{CaGaHe91}).

Next we discuss the behavior of $\mathrm{pcl} (V)$ at infinity.
Consider the decomposition of each $R_i$ into its homogeneous
components, namely
$$R_i=R_i^{d_i}+R_i^{d_i-1}+\cdots+R_i^{0},$$
where each $R_i^j\in\fq[X_1\klk X_r]$ is homogeneous of degree $j$
or zero, $R_i^{d_i}$ being nonzero for $1\le i\le m$. The
homogenization of each $R_i$ is the polynomial
\begin{equation}\label{eq: Ri homogenization}
R_i^h=R_i^{d_i}+R_i^{d_i-1}X_0+\cdots+R_i^{0}X_0^{d_i}.
\end{equation}
It follows that $R_i^h(0,X_1\klk X_r)=R_i^{d_i}$ for $1\leq i \leq
m$. To express each $R_i^{d_i}$ in terms of the component
$G_i^{\wt}$ of highest weight of $G_i$, let $A_0^{i_0}\cdots
A_{k-1}^{i_{k-1}}A_{k+1}^{i_{k+1}}\cdots A_{r-1}^{i_{r-1}}$ be a
monomial arising with nonzero coefficients in the dense
representation of $G_i$. Then its weight
$$\wt(A_0^{i_0}\cdots
A_{k-1}^{i_{k-1}}A_{k+1}^{i_{k+1}}A_{r-1}^{i_{r-1}})=\mathop{\sum_{j=0}^{r-1}}_
{j\not=k} (r-j) i_j$$
equals the degree of the corresponding monomial $\Pi_r^{i_0} \cdots
\Pi_{r-k+1}^{i_{k-1}}\Pi_{r-k-1}^{i_{k+1}}\cdots \Pi_{1}^{i_{r-1}}$
of $R_i$. We deduce the following result.
\begin{lemma}\label{lemma: rel between R^di and S^wt}
$R_i^{d_i}=G_i^{\wt}(-\Pi_1 \klk
(-1)^{r-k-1}\Pi_{r-k-1},(-1)^{r-k+1}\Pi_{r-k+1} \klk
(-1)^{r}\Pi_{r})$ for $1\leq i \leq m$. In particular, $\deg
R_i=\wt(G_i)$ for $1\leq i \leq m$.
\end{lemma}

Denote by $(\partial \bfs{R}^{\bfs{d}}/\partial \bfs{X}):=(\partial
R_i^{d_i}/\partial X_j)_{1\le i\le m,1\leq j \le r}$ the Jacobian
matrix of $R_{1}^{d_1}\klk R_{m}^{d_m}$ with respect to  $X_1\klk
X_r$. Let $\Sigma^{\infty}\subset\mathbb{P}^r$ be the singular locus
of $\mathrm{pcl}(V)$ at infinity, namely the set of singular points
of $\mathrm{pcl}(V)$ lying in the hyperplane $\{X_0=0\}$. We have
the following result.
\begin{lemma}\label{lemma: singular locus Vr at infinity}
The set of points $\bfs{x} \in
V(R_1^{d_1},\ldots,R_m^{d_m})\subset\Pp^{r-1}$ for which $(\partial
\bfs{R}^{\bfs{d}}/\partial \bfs{X})(\bfs{x})$ has not full rank, has
codimension at least 1 in $V(R_1^{d_1}\klk R_m^{d_m})$. In
particular, the singular locus $\Sigma^{\infty}\subset\mathbb{P}^r$
at infinity has dimension at most $r-m-2$.
\end{lemma}
\begin{proof}
Consider the affine variety $V_{\mathrm{aff}}(R_1^{d_1}\klk
R_m^{d_m})\subset\A^r$ defined by $R_1^{d_1}\klk R_m^{d_m}$.
Hypothesis (${\sf H}_3$) asserts that $G_1^{\wt}\klk G_m^{\wt}$
satisfy hypotheses (${\sf H}_1$) and (${\sf H}_2$). Therefore, Lemma
\ref{lemma: V is complete inters} proves that
$V_{\mathrm{aff}}(R_1^{d_1}\klk R_m^{d_m})$ is a set--theoretic
complete intersection of dimension $r-m$. Denote by
$\Sigma_{\mathrm{aff}}^\infty$ the set of points $\bfs x\in
V_{\mathrm{aff}}(R_1^{d_1}\klk R_m^{d_m})$ as in the statement of
the lemma. Arguing as in the proof of Theorem \ref{theorem:
dimension singular locus V} we conclude that any $\bfs
x\in\Sigma_{\mathrm{aff}}^\infty$ cannot have its $r$ coordinates
pairwise distinct. This implies that $\bfs
\Pi^r(\Sigma_{\mathrm{aff}}^\infty)$ is contained in the
discriminant locus $\mathcal{D}(V(G_1^{\wt}\klk G_m^{\wt}))$. By
hypothesis $({\sf H}_6)$ we have that $\mathcal{D}(V(G_1^{\wt}\klk
G_m^{\wt}))$ has codimension at least 1 in $V(G_1^{\wt}\klk
G_m^{\wt})=\bfs\Pi^r(V_{\mathrm{aff}}(R_1^{d_1}\klk R_m^{d_m}))$.
Since $\bfs \Pi^r$ is a finite morphism, we deduce that
$\Sigma_{\mathrm{aff}}^\infty$ has codimension at least 1 in
$V_{\mathrm{aff}}(R_1^{d_1}\klk R_m^{d_m})$. The first assertion of
the lemma follows.

Now let $\bfs{x}:=(0:x_1:\dots: x_r)$ be an arbitrary point of
$\Sigma^{\infty}$. Since each $R_i^{h}$ vanishes identically in
$\mathrm{pcl}(V)$, we have $R_i^{h}(\bfs{x})=R_i^{d_i}(x_1\klk
x_r)=0$ for $1\le i\le m$. Further, $(\partial
\bfs{R}^{\bfs{d}}/\partial \bfs{X})(\bfs{x})$ does not have full
rank, since otherwise we would have $\dim
\mathcal{T}_{\bfs{x}}(\mathrm{pcl}(V))\le r-m$, which would imply
that $\bfs{x}$ is a nonsingular point of $\mathrm{pcl}(V)$,
contradicting thus the hypothesis on $\bfs{x}$. It follows that
$\Sigma^{\infty}$ has codimension at least 1 in
$V(R_1^{d_1},\ldots,R_m^{d_m})$, and thus dimension at most $r-m-2$.
\end{proof}

Our next result concerns the projective variety
$V(R_1^{d_1},\ldots,R_m^{d_m})\subset \Pp^{r-1}$.
\begin{lemma}\label{Lemma: property variety Ri^di}
$V(R_1^{d_1},\ldots,R_m^{d_m})\subset \Pp^{r-1}$ is a complete
intersection of dimension $r-m-1$, degree $\prod_{i=1}^{m}d_i$ and
singular locus of dimension at most $r-m-2$.
\end{lemma}
\begin{proof}
Since $G_1^{\wt}\klk G_m^\wt$ satisfy hypothesis (${\sf H}_1$),
Lemma \ref{lemma: V is complete inters} shows that
$V(R_1^{d_1},\ldots,R_m^{d_m})$ is of pure dimension $r-m-1$.
Furthermore, Lemma \ref{lemma: singular locus Vr at infinity} shows
that the set of $\bfs{x} \in V(R_1^{d_1},\ldots,R_m^{d_m})$ for
which $(\partial \bfs{R}^{\bfs{d}}/\partial \bfs{X})(\bfs{x})$ has
not full rank, has codimension at least 1 in $V(R_1^{d_1}\klk
R_m^{d_m})$. Then Theorem \ref{theorem: eisenbud 18.15} proves that
$R_1^{d_1},\ldots,R_m^{d_m}$ define a radical ideal, and therefore
$V(R_1^{d_1}\klk R_m^{d_m})$ is a complete intersection.

In particular, the singular locus of $V(R_1^{d_1} \klk R_m^{d_m})$
is the set of $\bfs{x} \in V(R_1^{d_1},\ldots,R_m^{d_m})$ for which
$(\partial \bfs{R}^{\bfs{d}}/\partial \bfs{X})(\bfs{x})$ has not
full rank, and hence it has dimension at most $r-m-2$. Finally, the
B\'ezout theorem \eqref{eq: Bezout theorem} proves the assertion
concerning the degree.
\end{proof}

Now we prove our main result concerning $\mathrm{pcl}(V)$.
\begin{theorem}\label{theorem: proj closure of Vr is abs irred}
The identity $\mathrm{pcl}(V)=V(R_1^{h},\ldots,R_m^{h})$ holds and
$\mathrm{pcl}(V)$ is a normal complete intersection of dimension
$r-m$ and degree $\prod_{i=1}^r d_i$.
\end{theorem}
\begin{proof}
Observe that the following inclusions hold:
    \begin{align*}
    V(R_1^{h},\ldots,R_m^{h})\cap\{X_0\not=0\}&
    \subset V(R_1,\ldots,R_m),\\
    V(R_1^{h},\ldots,R_m^{h})\cap\{X_0=0\}&\subset
    V(R_{1}^{d_1},\ldots,R_{m}^{d_m}).
    \end{align*}
Lemma \ref{Lemma: property variety Ri^di} proves that
$V(R_1^{d_1},\ldots,R_m^{d_m})\subset \Pp^{r-1} $ is a complete
intersection of dimension $r-m-1$ and singular locus of codimension
at least $1$. On the other hand, Lemma \ref{lemma: V is complete
inters} and Theorem \ref{theorem: dimension singular locus V} show
that $V(R_1,\ldots,R_m)\subset\A^r$ is of pure dimension $r-m$ and
its singular locus has codimension at least $2$. We conclude that
the same holds with $V(R_1^{h},\ldots,R_m^{h})\subset\Pp^r$. Since
it is defined by $m$ polynomials, it is a set--theoretic complete
intersection. Further, by Theorem \ref{theorem: dimension singular
locus V} and Lemma \ref{lemma: singular locus Vr at infinity} the
set of points $\bfs x\in V(R_1^h,\ldots,R_m^h)$ for which $(\partial
\bfs{R}^{\bfs{h}}/\partial \bfs X)(\bfs x)$ has not full rank, has
codimension at least 2 in $V(R_1^h,\ldots,R_m^h)$. Then Theorem
\ref{theorem: eisenbud 18.15} proves that $R_1^{h},\ldots,R_m^{h}$
define a radical ideal and therefore $V(R_1^{h},\ldots,R_m^{h})$ is
a normal complete intersection. By Theorem \ref{theorem: normal
complete int implies irred} it follows that
$V(R_1^{h},\ldots,R_m^{h})$ is absolutely irreducible.

It is clear that $\mathrm{pcl}(V)\subset V(R_1^{h},\ldots,R_m^{h})$.
Being both of pure dimension $r-m$ and $V(R_1^{h},\ldots,R_m^{h})$
absolutely irreducible, the identity of the statement of the theorem
follows. Finally, since $R_1^h,\ldots,R_m^h$ define a radical ideal,
the B\'ezout theorem \eqref{eq: Bezout theorem} proves the assertion
on the degree.
\end{proof}

We end the section with the following result, which allows us to
control the number of $\fq$--rational points of $\mathrm{pcl}(V)$ at
infinity.
\begin{remark}\label{remark: dimension of pcl(V) at infinity}
$V_{ \infty}:=\mathrm{pcl}(V)\cap \{X_0=0\}\subset \Pp ^{r-1}$ has
dimension $r-m-1$. Indeed, recall that $\mathrm{pcl}(V)$ has pure
dimension $r-m$. Hence, each irreducible component of
$\mathrm{pcl}(V)\cap \{X_0=0\}$ has dimension at least $r-m-1$. From
(\ref{eq: Ri homogenization}) we deduce that $\mathrm{pcl}(V)\cap
\{X_0=0\}\subset V(R_1^{d_1}\klk R_m^{d_m})$. By Lemma \ref{Lemma:
property variety Ri^di} we have that $V(R_1^{d_1},\ldots,R_m^{d_m})$
has dimension $r-m-1$. It follows that $\mathrm{pcl}(V)\cap
\{X_0=0\}$ also has dimension $r-m-1$.
\end{remark}
%
%
\subsection{Estimates on the number of $\fq$--rational points of $W$}
The results on $V$ allows us to estimate the number of
$\fq$--rational points of $W$. We start with the following result.
\begin{corollary}\label{theo: W absolutely irreducible}
$W\subset \A^r$ is absolutely irreducible.
\end{corollary}
\begin{proof}
By Theorems \ref{theorem: proj closure of Vr is abs irred} and
\ref{theorem: normal complete int implies irred} we have that $\mathrm{pcl}(V)$ is absolutely irreducible. As a
consequence, $V$ is absolutely irreducible. Since $\bfs\Pi^r(V)=W$,
the assertion follows.
\end{proof}

As $|\mathcal{A}|=|W(\fq)|$, we obtain estimates on the number of
elements of $\mathcal{A}$. Combining Corollary \ref{theo: W
absolutely irreducible} with \cite[Theorem 7.1]{CaMa06}, for $q
> \delta_{\bfs G}:=\deg(G_1)\cdots\deg(G_m)$ we have the following
estimate:
$$
\big||\mathcal{A}|-q^{r-m}\big| \le (\delta_{\bfs G}-1)(\delta_{\bfs
G}-2)q^{r-m-{1}/{2}}+5 \delta_{\bfs G}^{13/3}q^{r-m-1}.
$$
On the other hand, according to \cite[Corollary 7.2]{CaMa06}, if $q
>15\delta_{\bfs G}^{13/3}$, then
$$
\big||\mathcal{A}|-q^{r-m}\big| \le (\delta_{\bfs G}-1)(\delta_{\bfs
G}-2)q^{r-m-{1}/{2}}+7\delta_{\bfs G}^2q^{r-m-1}.
$$
We easily deduce the following result. 
%
\begin{theorem} \label{estimation A}
For $q >15\delta_{\bfs G}^{13/3}$, we have
$$
|\mathcal{A}|\geq q^{r-m}\bigg(1-\frac{3\delta_{\bfs
G}^{{13}/{6}}}{q^{{1}/{2}}}\bigg)\textrm{ and
}|\mathcal{A}|^{-1}\leq q^{m-r}\bigg(1+\frac{15\delta_{\bfs
G}^{{13}/{6}}}{q^{{1}/{2}}}\bigg).
$$
Further,
$$
|\mathcal{A}|\geq \frac{1}{2}q^{r-m}.
$$
\end{theorem}
%
%
\section{The distribution of factorization patterns in $\mathcal{A}$}
\label{the factorization patterns}
Let $\lambda_1,\dots,\lambda_r$ be nonnegative integers such that
$\lambda_1+2\lambda_2+\cdots+r\lambda_r=r$. Denote by ${\mathcal
P}_{\bfs \lambda}$ the set of $f\in \fq[T]_r$ with factorization
pattern $\bfs \lambda:=1^{\lambda_1}2^{\lambda_2}\cdots
r^{\lambda_r}$, namely having exactly $\lambda_i$ monic irreducible
factors over $\fq$ of degree $i$ (counted with multiplicity) for
$1\le i\le r$. Further, for $\mathcal{S}\subset\fq[T]_r$ we denote
$\mathcal{S}_{\bfs
\lambda}:=\mathcal{S}\cap\mathcal{P}_{\bfs\lambda}$. In this section
we estimate the number $|\mathcal{A}_{\bfs\lambda}|$ of elements of
$\mathcal{A}$ with factorization pattern $\bfs\lambda$, where
$\mathcal{A}\subset\fq[T]_r$ is the family of \eqref{non-linear
family A}.
%
%
\subsection{Factorization patterns and roots}
Following the approach of \cite{CeMaPe17}, we show that the set
$\mathcal{A}_{\bfs\lambda}$ can be expressed in terms of certain
symmetric polynomials.

Let $f\in\fq[T]_r$ and $m\in \fq[T]$ a monic irreducible factor of
$f$ of degree $i$. Then $m$ is the minimal polynomial of a root
$\alpha$ of $f$ with $\fq(\alpha)=\fqi$. Denote by $\mathbb G_i$ the
Galois group $\mbox{Gal}(\fqi,\fq)$ of $\fqi$ over $\fq$. We may
express $m$ in the following way:
$$m=\prod_{\sigma\in\mathbb G_i}(T-\sigma(\alpha)).$$
Hence, each irreducible factor $m$ of $f$ is uniquely determined by
a root $\alpha$ of $f$ (and its orbit under the action of the Galois
group of $\cfq$ over $\fq$), and this root belongs to a field
extension of $\fq$ of degree $\deg m$. Now, for
$f\in\mathcal{P}_{\bfs \lambda}$, there are $\lambda_1$ roots of $f$
in $\fq$, say $\alpha_1,\dots,\alpha_{\lambda_1}$ (counted with
multiplicity), which are associated with the irreducible factors of
$f$ in $\fq[T]$ of degree 1; we may choose $\lambda_2$ roots of $f$
in $\fqtwo\setminus\fq$ (counted with multiplicity), say
$\alpha_{\lambda_1+1},\dots, \alpha_{\lambda_1+\lambda_2}$, which
are associated with the $\lambda_2$ irreducible factors of $f$ of
degree 2, and so on. From now on we assume that a choice of
$\lambda_1\plp\lambda_r$ roots $\alpha_1\klk\alpha_{\lambda_1
    \plp\lambda_r}$ of $f$ in $\cfq$ is made in such a way that each
monic irreducible factor of $f$ in $\fq[T]$ is associated with one
and only one of these roots.

Our aim is to express the factorization of $f$ into irreducible
factors in $\fq[T]$ in terms of the coordinates of the chosen
$\lambda_1\plp \lambda_r$ roots of $f$ with respect to certain bases
of the corresponding extensions $\fq\hookrightarrow\fqi$ as
$\fq$--vector spaces. To this end, we express the root associated
with each irreducible factor of $f$ of degree $i$ in a normal basis
$\Theta_i$ of the field extension $\fq\hookrightarrow \fqi$.

Let $\theta_i\in \fqi$ be a normal element and $\Theta_i$ the normal
basis of the extension $\fq\hookrightarrow\fqi$ generated by
$\theta_i$, i.e.,
$$\Theta_i=\left \{\theta_i,\cdots, \theta_i^{q^{i-1}}\right\}.$$
The Galois group $\mathbb G_i$ is cyclic and the Frobenius map
$\sigma_i:\fqi\to\fqi$, $\sigma_i(x):=x^q$ is a generator of
$\mathbb{G}_i$. Thus, the coordinates in the basis $\Theta_i$ of all
the elements in the orbit of a root $\alpha_k\in\fqi$ of an
irreducible factor of $f$ of degree $i$ are the cyclic permutations
of the coordinates of $\alpha_k$ in the basis $\Theta_i$.

The vector that gathers the coordinates of all the roots
$\alpha_1\klk\alpha_{\lambda_1+\dots+\lambda_r}$ we choose to
represent the irreducible factors of $f$ in the normal bases
$\Theta_1\klk \Theta_r$ is an element of $\fq^r$, which is denoted
by ${\bfs x}:=(x_1,\dots,x_r)$. Set
\begin{equation}\label{eq: fact patterns: ell_ij}
\ell_{i,j}:=\sum_{k=1}^{i-1}k\lambda_k+(j-1)\,i
\end{equation}
for $1\le j \le \lambda_i$ and $1\le i \le r$. Observe that the
vector of coordinates of a root
$\alpha_{\lambda_1\plp\lambda_{i-1}+j}\in\fqi$ is the sub-array
$(x_{\ell_{i,j}+1},\dots,x_{\ell_{i,j}+i})$ of $\bfs x$. With these
notations, the $\lambda_i$ irreducible factors of $f$ of degree $i$
are the polynomials
\begin{equation}\label{eq: fact patterns: gij}m_{i,j}=\prod_{\sigma\in\mathbb G_i}
\Big(T-\big(x_{\ell_{i,j}+1}\sigma(\theta_i)+\dots+
x_{\ell_{i,j}+i}\sigma(\theta_i^{q^{i-1}})\big)\Big)
\end{equation}
for $1\le j \le \lambda_i$. In particular,
\begin{equation}\label{eq: fact patterns: f factored with g_ij}
f=\prod_{i=1}^r\prod_{j=1}^{\lambda_i}m_{i,j}.
\end{equation}

Let $X_1\klk X_r$ be indeterminates over $\cfq$, set $\bfs
X:=(X_1,\dots,X_r)$ and consider the polynomial $M\in
\fq[\bfs{X},T]$ defined as
\begin{equation}\label{eq: fact patterns: pol G}
M:=\prod_{i=1}^r\prod_{j=1}^{\lambda_i}M_{i,j},\quad
M_{i,j}:=\prod_{\sigma\in\mathbb G_i}
\Big(T-\big(X_{\ell_{i,j}+1}\sigma(\theta_i)+
\dots+X_{\ell_{i,j}+i}\sigma(\theta_i^{q^{i-1}})\big)\Big),
\end{equation}
where the $\ell_{i,j}$ are defined as in \eqref{eq: fact patterns:
    ell_ij}. Our previous arguments show that $f\in\fq[T]_r$ has
factorization pattern ${\bfs \lambda}$ if and only if there exists
$\bfs x\in\fq^r$ with $f=M({\bfs x},T)$.

To discuss how many elements $\bfs x\in\fq^r$ yield an arbitrary
polynomial $f=M(\bfs x,T)\in\mathcal{P}_{\bfs\lambda}$, 
%
we introduce the notion of an array of type $\bfs \lambda$. Let
$\ell_{i,j}$ $(1\le i\le r,\ 1\le j\le\lambda_i)$ be defined as in
\eqref{eq: fact patterns: ell_ij}. We say that ${\bfs x}=(x_1,\dots,
x_r)\in\fq^r$ is of {\em type
    $\bfs \lambda$} if and only if each sub-array $\bfs
x_{i,j}:=(x_{\ell_{i,j}+1},\dots,x_{\ell_{i,j}+i})$ is a cycle of
length $i$. The following result relates the set
$\mathcal{P}_{\boldsymbol{\lambda}}$ with the set of elements of
$\fq^r$ of type $\bfs \lambda$ (see \cite[Lemma 2.2]{CeMaPe17}).
\begin{lemma}
\label{lemma: fact patterns: G(x,T) with fact pat lambda} For any
${\bfs x}=(x_1,\dots, x_r)\in \fq^r$, the polynomial $f:=M({\bfs
x},T)$ has factorization pattern $\bfs \lambda$ if and only if
${\bfs x}$ is of type $\bfs \lambda$. Furthermore, for each
square--free polynomial $f\in \mathcal{P}_{\bfs \lambda}$ there are
$w({\bfs\lambda}):=\prod_{i=1}^r i^{\lambda_i}\lambda_i!$ different
${\bfs x}\in \fq^r $ with $f=M({\bfs x},T)$.
\end{lemma}

Consider the polynomial $M$ of \eqref{eq: fact patterns: pol G} as
an element of $\fq[\bfs X][T]$. We shall express the coefficients of
$M$ by means of the vector of linear forms $\bfs Y:=(Y_1\klk Y_r)$,
with $Y_i\in\cfq[\bfs X]$ defined in the following way for $1\le
i\le r$:
\begin{equation}\label{eq: fact patterns: def linear forms Y}
(Y_{\ell_{i,j}+1},\dots,Y_{\ell_{i,j}+i})^{t}:=A_{i}\cdot
(X_{\ell_{i,j}+1},\dots, X_{\ell_{i,j}+i})^{t} \quad(1\le j\le
\lambda_i,\ 1\le i\le r),
\end{equation}
where $A_i\in\fqi^{i\times i}$ is the matrix
$$A_i:=\left(\sigma(\theta_i^{q^{h}})\right)_{\sigma\in {\mathbb G}_i,\, 0\le h\le i-1}.$$
According to (\ref{eq: fact patterns: pol G}), we may express the
polynomial $M$ as
$$M=\prod_{i=1}^r\prod_{j=1}^{\lambda_i}\prod_{s=1}^i(T-Y_{\ell_{i,j}+s})=
\prod_{i=1}^r(T-Y_i)=T^r+\sum_{i=1}^r(-1)^i\,(\Pi_i(\bfs Y))\,
T^{r-i},$$
where $\Pi_1(\bfs Y)\klk \Pi_r(\bfs Y)$ are the elementary symmetric
polynomials of $\fq[\bfs Y]$. By \eqref{eq: fact patterns: pol G} we
see that $M$ belongs to $\fq[{\bfs X},T]$, which in particular
implies that $\Pi_i(\bfs Y)$ belongs to $\fq[{\bfs X}]$ for $1\le
i\le r$. Combining these arguments with Lemma \ref{lemma: fact
patterns: G(x,T) with fact pat lambda} we obtain the following
result.
\begin{lemma}
    \label{lemma: fact patterns: sym pols and pattern lambda} A
    polynomial $f:=T^r+a_{r-1}T^{r-1}\plp a_0\in\fq[T]_r$ has
    factorization pattern $\bfs \lambda$ if and only if there exists
    $\bfs{x}\in\fq^r$ of type $\bfs \lambda$ such that
    \begin{equation}\label{eq: fact patterns: sym pols and pattern lambda}
    a_i= (-1)^{{r-i}}\,\Pi_{r-i}(\bfs Y(\bfs x)) \quad(0\le i\le r-1).
    \end{equation}
    In particular, for $f$ square--free, there are $w(\bfs \lambda)$
    elements $\bfs x$ for which (\ref{eq: fact patterns: sym pols and
        pattern lambda}) holds.
\end{lemma}

Recall that the family $\mathcal{A}$ of \eqref{non-linear family A}
is defined by polynomial $G_1 \klk G_m \in \cfq[\bfs{A_k}]$, for a
fixed $k$ with $0 \leq k \leq r-1$. As a consequence, we may express
the condition that an element of $\mathcal{A}$ has factorization
pattern $\bfs \lambda$ in terms of the elementary symmetric
polynomials $\Pi_1 \klk \Pi_{r-k-1},\Pi_{r-k+1} \klk \Pi_{r}$ of
$\fq[\bfs Y]$.
\begin{corollary}\label{coro: fact patterns: systems pattern lambda}
A polynomial $f:=T^r+a_{r-1}T^{r-1}\plp a_0\in\fq[T]_r$ belongs to
$\mathcal{A}_{\bfs \lambda}$ if and only if there exists
$\bfs{x}\in\fq^r$ of type $\bfs \lambda$ such that \eqref{eq: fact
        patterns: sym pols and pattern lambda} and
    \begin{equation}\label{eq: fact patterns: systems pattern lambda}
    G_j\big(-\Pi_{1}\klk
    (-1)^{r-k-1}\Pi_{r-k-1},(-1)^{r-k+1}\Pi_{r-k+1}\klk (-1)^r \Pi_{r}
    \big)(\bfs Y(\bfs x))=0 \quad(1\le j\le m)
    \end{equation}
hold, where $G_1\klk G_m$ are the polynomials defining the family
$\mathcal{A}$. In particular, if $f:=M(\bfs
x,T)\in\mathcal{A}_{\bfs\lambda}$ is square--free, then there are
$w(\bfs \lambda)$ elements $\bfs x$ for which (\ref{eq: fact
patterns: systems pattern lambda}) holds.
\end{corollary}
%
%
\subsection{The number of polynomials in $\mathcal{A}_{\bfs \lambda}$}
Given a factorization pattern $\bfs{\lambda}$, in this section we
estimate the number of elements of $\mathcal{A}_{\bfs\lambda}$. For
this purpose, in Corollary \ref{coro: fact patterns: systems pattern
lambda} we associate to $\mathcal{A}_{\bfs \lambda}$ the polynomials
$R_1\klk R_m\in\fq[\bfs X]$ defined as follows:
\begin{equation}\label{eq: geometry: def R_j}
R_j:= G_j\big(-\Pi_{1}\klk
    (-1)^{r-k-1}\Pi_{r-k-1},(-1)^{r-k+1}\Pi_{r-k+1}\klk (-1)^r \Pi_{r}
    \big)(\bfs Y(\bfs x)).
\end{equation}
%
%
Let $V:=V(R_1 \klk R_m) \subset \A^r$ be the variety defined by
$R_1\klk R_m$. Since $G_1 \klk G_m$ satisfy hypotheses $({\sf
H}_1)$--$({\sf H}_6)$, by Lemma \ref{lemma: V is complete inters},
Corollary \ref{coro: radicality and degree Vr},  Theorem
\ref{theorem: proj closure of Vr is abs irred} and Remark
\ref{remark: dimension of pcl(V) at infinity} we obtain the
following result.

\begin{theorem}\label{Theo:geometry projective clousure}
    Let $ m, r$ be positive integers with $m<r$.
    \begin{enumerate}
        \item $V\subset \A^r$ is a complete intersection of dimension $r-m$.
        \item The projective variety $\mathrm{pcl}(V)\subset \Pp^r$ is a normal
        complete intersection of dimension $r-m$ and degree $\prod_{i=1}^m d_i$,
        where $d_i:=\deg(R_i)=\wt(G_i)$
        for $1\le i\le m$.
        \item  $V_{\infty} :=\mathrm{pcl}(V)\cap \{Y_0=0\}\subset \Pp^{r-1}$ has dimension $r-m-1$.
    \end{enumerate}
\end{theorem}

Now we estimate the number of $\fq$--rational points of $V$.
According to Theorem \ref{Theo:geometry projective clousure},
$\mathrm{pcl}(V)\subset\Pp^r$ is a normal complete intersection
defined over $\fq$, of dimension $r-m$ and multidegree $\bfs
d:=(d_1\klk d_m)$. Therefore, \cite[Corollary 8.4]{CaMaPr15} implies
that the following estimate holds (see \cite{CaMa07},
\cite{GhLa02a}, \cite{GhLa02} and \cite{MaPePr16} for further
explicit estimates of this type):
$$
\big||\mathrm{pcl}(V)(\fq)|-p_{r-m}\big|\le (\delta
(D-2)+2)q^{r-m-\frac{1}{2}}+14 D^2 \delta^2q^{r-m-1}.
$$
where $p_{r-m}:=q^{r-m}+\cdots+ q+1=|\Pp^{r-m}(\fq)|$,
$\delta:=d_1\cdots d_m$ and $D:=\sum_{i=1}^m(d_i-1)$.

On the other hand, the B\'ezout inequality \eqref{eq: Bezout}
implies $\deg V_{\infty}\le\delta$. Then by Theorem
\ref{Theo:geometry projective clousure} and \eqref{eq: upper bound
-- projective gral} we have
$$
\big|V_{\infty}(\fq)\big|\le \delta p_{r-m-1}.
$$
It follows that
\begin{align}\label{ineq: estimation Vr}
\big||V(\fq)|-q^{r-m}\big|& =
\big||\mathrm{pcl}(V)(\fq)|-|V_{\infty} (\fq)|-
p_{r-m}+p_{r-m-1}\big|\nonumber\\[1ex]
& \le
\big||\mathrm{pcl}(V)(\fq)|-p_{r-m}\big|+
\big|V_{\infty}(\fq)\big|+ 2{q^{r-m-1}}
\nonumber\\[1ex]
& \le \big((\delta(D-2)+2)q^{\frac{1}{2}}+14D^2 \delta^2+2
\delta+2\big)q^{r-m-1}.
\end{align}

Let $V^{=}$ be the subvariety of $V$ defined as
$$V^{ =}:=\mathop{\bigcup_{1\le i\le r}}_{ 1\leq j_1 <j_2\leq
    \lambda _i,\, \,  1\leq k_1 <k_2 \leq i}
V\cap\{Y_{\ell_{i,j_1}+k_1}=Y_{\ell_{i,j_2}+k_2}\},
$$
where $Y_{\ell_{i,j}+k}$ are the linear forms of \eqref{eq: fact
patterns: def linear forms Y}. Let $V^{ \neq}(\fq):=V(\fq)\backslash
V^{ =}(\fq)$.
We claim that $V\cap\{Y_{\ell_{i,j_1}+k_1}=Y_{\ell_{i,j_2}+k_2}\}$
has dimension at most $r-m-1$ for every $1\le i\le r$, $1\leq j_1
<j_2\leq \lambda _i$ and $1\leq k_1 <k_2 \leq i$. Indeed, let $\bfs
x\in V\cap\{Y_{\ell_{i,j_1}+k_1}=Y_{\ell_{i,j_2}+k_2}\}$ for
$i,j_1,j_2,k_1,k_2$ as above. By \eqref{eq: fact patterns: pol G} we
conclude that $M(\bfs x,T)$ is not square--free, and therefore
$\Pi^r(\bfs Y(\bfs x))\in\mathcal{D}(W)$. Since $G_1\klk G_m$
satisfy $({\sf H}_4)$, it follows that $\dim \mathcal{D}(W)\le
r-m-1$, and the fact that $\Pi^r$ is a finite morphism implies that
$\dim(\Pi^r)^{-1}(\mathcal{D}(W))\le r-m-1$. This proves our claim.

The claim implies $\dim V^{ =}\le r-m-1$. By the B\'ezout inequality
\eqref{eq: Bezout} we have
$$\deg V^{ =}\le \deg V\sum_{i=1}^r\frac{i^2\lambda_i^2}{4}\le \frac{r^2}{4}\delta.$$
As a consequence, by \eqref{eq: upper bound -- affine gral} we see
that
\begin{equation}\label{eq: estimates: upper bound V^=(fq)}
|V^{ =}(\fq)|\le \deg V^{ =}\,q^{r-m-1}\le \frac{r^2\delta}{4}\,
q^{r-m-1}.
\end{equation}

Finally, combining \eqref{ineq: estimation Vr} and \eqref{eq:
estimates: upper bound V^=(fq)} we obtain the following result.
\begin{theorem}\label{ineq:estimation Vr distintos}
    For $m<r$ we have
    \begin{align*}
    \big||V^{ \neq}(\fq)|-q^{r-m}\big|&\le q^{r-m-1}
    \Big((\delta (D-2)+2)q^{\frac{1}{2}}+14 D^2 \delta^2+2\delta+2+r^2\delta/4\Big),
    \end{align*}
    where $\delta:=\prod_{i=1}^m \wt(G_i)$ and $D:=\sum_{i=1}^m (\wt(G_i)-1).$
\end{theorem}
\begin{proof}
By \eqref{eq: estimates: upper bound V^=(fq)}, $|V^{ =}(\fq)|\le
r^2\delta\, q^{r-m-1}/4$. Then, from \eqref{ineq: estimation Vr} we
deduce that
    \begin{align*}
    \big||V^{ \neq}(\fq)|-q^{r-m}\big|& \leq
    \big||V(\fq)|-q^{r-m}\big|+\big|V^{ =}(\fq)\big|\nonumber\\[1ex]
    & \le \big((\delta(D-2)+2)q^{\frac{1}{2}}+14D^2 \delta^2+2 \delta+2\big)q^{r-m-1}
    + \frac{r^2\delta}{4} q^{r-m-1}.
    \end{align*}
This shows the statement of the theorem.
\end{proof}

Next we use Corollary \ref{coro: fact patterns: systems pattern
lambda} to relate $|V(\fq)|$ to the quantity
$|\mathcal{A}_{\bfs{\lambda}}|$. More precisely, let $\bfs x:=(\bfs
x_{i,j}:1\le i\le r,1\le j\le \lambda_i)\in\fq^r$ be an
$\fq$--rational zero of $R_1\klk R_m$ of type $\bfs\lambda$. Then
$\bfs x$ is associated with $f\in\mathcal{A}_{\bfs\lambda}$ having
$Y_{\ell_{i,j}+k}(\bfs x_{i,j})$ as an $\fqi$--root for $1\le i\le
r$, $1\le j\le\lambda_i$ and $1\le k\le i$, where $Y_{\ell_{i,j}+k}$
are the linear forms of \eqref{eq: fact patterns: def linear forms
Y}.

Let $\mathcal{A}_{\bfs\lambda}^{sq}:=\{f\in
\mathcal{A}_{\bfs\lambda}: f \mbox{ is square--free}\}$ and
$\mathcal{A}_{\bfs\lambda}^{nsq}:=\mathcal{A}_{\bfs\lambda}\setminus
\mathcal{A}_{\bfs\lambda}^{sq}$. Corollary \ref{coro: fact patterns:
systems pattern lambda} shows that any element $f\in
\mathcal{A}_{\bfs\lambda}^{sq}$ is associated with
$w(\bfs\lambda):=\prod_{i=1}^r i^{\lambda_i}\lambda_i!$ common
$\fq$--rational zeros of $R_1\klk R_m$ of type $\bfs\lambda$.
Observe that $\bfs x\in\fq^r$ is of type $\bfs\lambda$ if and only
if $Y_{\ell_{i,j}+k_1}(\bfs x) \neq Y_{\ell_{i,j}+k_2}(\bfs x)$ for
$1\leq i\leq r$, $1\leq j\leq \lambda_i$ and $1\leq k_1 <k_2 \leq
i$. Furthermore, an $\bfs x\in\fq^r$ of type $\bfs\lambda$ is
associated with $f\in\mathcal{A}_{\bfs\lambda}^{sq}$ if and only if
$Y_{\ell_{i,j_1}+k_1}(\bfs x) \neq Y_{\ell_{i,j_2}+k_2}(\bfs x)$ for
$1\leq i\leq r$, $1\leq j_1<j_2\leq \lambda_i$ and $1\leq k_1 <k_2
\leq i$. As a consequence, we see that
$|\mathcal{A}_{\bfs\lambda}^{sq}| =\mathcal{T}(\bfs\lambda)
\big|V^{\neq}(\fq)\big|$, where $\mathcal{T}(\bfs\lambda):=1/w(\bfs
\lambda)$. This implies
$$
\big||\mathcal{A}_{\bfs\lambda}^{sq}|
-\mathcal{T}(\bfs\lambda)\,q^{r-m}\big| =
\mathcal{T}(\bfs\lambda)\,\big||V^{
     \neq}(\fq)|-q^{r-m}\big|.
$$
From Theorem \ref{ineq:estimation Vr distintos} we deduce that
\begin{align*}
\big||\mathcal{A}_{\bfs\lambda}^{sq}|
-\mathcal{T}(\bfs\lambda)\,q^{r-m}\big| &\le
\,\mathcal{T}(\bfs\lambda)q^{r-m-1}\big((\delta
(D-2)+2)q^{\frac{1}{2}}+14 D^2
\delta^2+2\delta+2+r^2\delta/4\big)\\
&\le \,\mathcal{T}(\bfs\lambda)q^{r-m-1}\big((\delta
(D-2)+2)q^{\frac{1}{2}}+14 D^2 \delta^2+r^2\delta\big).
\end{align*}
Now we are able to estimate $|\mathcal{A}_{\bfs \lambda}|$. We have
\begin{align}
\big||\mathcal{A}_{\bfs\lambda}|
-\mathcal{T}(\bfs\lambda)\,q^{r-m}\big|&=
\big||\mathcal{A}_{\bfs\lambda}^{sq}|+
|\mathcal{A}_{\bfs\lambda}^{nsq}|-\mathcal{T}(\bfs\lambda)q^{r-m}\big|\nonumber\\
&\le\mathcal{T}(\bfs\lambda)q^{r-m-1}\big((\delta
(D-2)+2)q^{\frac{1}{2}}+14 D^2 \delta^2+r^2\delta\big)+
|\mathcal{A}_{\bfs\lambda}^{nsq}|. \label{eq: estimates: estimate
    A_lambda aux}
\end{align}
It remains to bound $|\mathcal{A}_{\bfs\lambda}^{nsq}|$. To this
end, we observe that $f\in \mathcal{A}$ is not square--free if and
only if its discriminant is equal to zero, namely it belongs to the
discriminant locus $\mathcal{D}(W)$. By hypothesis $({\sf H}_4)$ the
discriminant locus $\mathcal{D}(W)$ has dimension at most $r-m-1$.
Further, by the B\'ezout inequality \eqref{eq: Bezout} we have
$$\deg\mathcal{D}(W)\le \deg W\cdot \deg \{\bfs a_0 \in \A^r :
\mathrm{Disc}(F(\bfs A_0, T))|_{\bfs A_0=\bfs a_0}= 0\}\le
\delta_{\bfs G}\, r(r-1)\le \delta\, r^2.$$
Then \eqref{eq: upper bound -- affine gral} implies
\begin{equation}\label{eq: estimates: upper bound discr locus}
|\mathcal{A}_{\bfs\lambda}^{nsq}|\le |\mathcal{A}^{nsq}|\le
\delta_{\bfs G}\, r(r-1) \,q^{r-m-1}\le \delta\, r^2q^{r-m-1}.
\end{equation}
Hence, combining \eqref{eq: estimates: estimate A_lambda aux} and
\eqref{eq: estimates: upper bound discr locus} we conclude that
\begin{align*}
\big||\mathcal{A}_{\bfs\lambda}|
-\mathcal{T}(\bfs\lambda)\,q^{r-m}\big|&\le
q^{r-m-1}\Big(\mathcal{T}(\bfs\lambda)\big((\delta
(D-2)+2)q^{\frac{1}{2}}\!+\!14 D^2
\delta^2+r^2\delta\big)+r^2\delta\Big).
\end{align*}
In other words, we have the following result.
\begin{theorem}\label{theorem: estimate fact patterns}
    For $m<r$, we have that
    \begin{align*}
    \big||\mathcal{A}_{\bfs\lambda}^{sq}|
    -\mathcal{T}(\bfs\lambda)\,q^{r-m}\big|& \leq \mathcal{T}(\bfs\lambda)q^{r-m-1}
    \big((\delta (D-2)+2)q^{\frac{1}{2}}+14 D^2 \delta^2+ r^2\delta\big),\\
    \big||\mathcal{A}_{\bfs\lambda}|
    -\mathcal{T}(\bfs\lambda)\,q^{r-m}\big|&\le
    q^{r-m-1}\Big(\mathcal{T}(\bfs\lambda)
    \big((\delta (D-2)+2)q^{\frac{1}{2}}+14 D^2 \delta^2+r^2\delta\big)+r^2\delta\Big),
    \end{align*}
    where $\delta:=\prod_{i=1}^m \wt(G_i)$ and $D:=\sum_{i=1}^m(\wt(G_i)-1)$.
\end{theorem}

As we show in Section \ref{subsec: examples CeMaPe17}, Theorem
\ref{theorem: estimate fact patterns} extends \cite[Theorem
4.2]{CeMaPe17}. More precisely, Theorem \ref{theorem: estimate fact
patterns} holds for families defined by linearly--independent linear
polynomials $G_1,\ldots,G_m\in \fq[A_{r-1},\ldots, A_2]$ with
$\mathrm{char}(\fq)$ not dividing $r(r-1)$, and
linearly--independent linear polynomials $G_1,\ldots,G_m\in
\fq[A_{r-1},\ldots, A_3]$ with $\mathrm{char}(\fq)>2$. The latter is
precisely \cite[Theorem 4.2]{CeMaPe17}.
\section{Examples of linear and nonlinear families}
\label{sec: examples}
In this section we exhibit examples of linear and nonlinear families
of polynomials satisfying hypotheses $({\sf H}_1)$--$({\sf H}_6)$.
Therefore, the estimate of Theorem \ref{theorem: estimate fact
patterns} is valid for these families.
%
%
\subsection{The linear families of \cite{CeMaPe17}}
\label{subsec: examples CeMaPe17}
Suppose that $\mathrm{char}(\fq)>3$. Let $r,m,n$ be positive
integers with $2\leq  n \leq r-m$ and
$L_1,\ldots,L_m\in\fq[A_{r-1},\ldots,A_n]$ linear forms which are
linearly independent. In \cite{CeMaPe17} the distribution of
factorization patterns of the following linear family is considered:
\begin{equation}\label{eq: family paper CMP17}
\mathcal{A}:=\left\{T^r+a_{r-1}T^{r-1}\plp a_0\in\fq[T]:
L_j(a_{r-1}\klk a_n)=0\quad (1\le j\le m)\right\}.
\end{equation}
Assume without loss of generality that the Jacobian matrix
$(\partial L_i/\partial A_j)_{1\le i\le m,\,n\le j\le r-1}$ is lower
triangular in row echelon form and denote by $1\le i_1<\cdots<i_m\le
r-n$ the positions corresponding to the pivots. We have the
following result.
\begin{lemma}\label{lemma: example MPP satisfies H1-H6}
If either $n=2$ and $\mathrm{char}(\fq)$ does not divide $r(r-1)$ or
$n \geq 3$, then $L_1 \klk L_m$ satisfy hypotheses $({\sf
H}_1)$--$({\sf H}_6)$.
\end{lemma}
\begin{proof}
It is clear that hypotheses $({\sf H}_1)$--$({\sf H}_2)$ hold.
Further, since the component of highest weight of $L_k$ is of the
form $L_k^{\wt}=b_{k,r-i_k}A_{r-i_k}$ for $1\le k\le m$, we conclude
that $({\sf H}_3)$ holds.

Now we analyze the validity of $({\sf H}_4)$.  Denote
$W:=V(L_1,\ldots,L_m)\subset\A^r$. It is clear that
$$\cfq[W]:=\cfq[A_{r-1} \klk A_0]/(L_1 \klk L_m) \simeq \cfq[A_k: k \in
\mathcal{J}]$$
is a domain, where $\mathcal{J}:=\{{r-1} \klk 0\}\setminus \{r-i_1
\klk r-i_m\}$. Therefore, it suffices to prove that the coordinate
class $\mathcal{R}$ defined by $\mathrm{Disc}(F(\bfs{A}_0,T))$ in
$\cfq[W]$ is a nonzero polynomial in $\cfq[A_k: k \in \mathcal{J}]$,
where $F(\bfs A_0,T):=T^r+A_{r-1}T^{r-1} +\dots+ A_0$ and $\bfs
A_0:=(A_{r-1} \klk A_0)$. If $\mathrm{char}(\fq)$ does not divide
$r(r-1)$, then the nonzero monomial $r^rA_0^{r-1}$ occurs in the
dense representation of $\mathcal{R}$. On the other hand, if  $
\mathrm{char}(\fq)$ divides $r$, then the nonzero  monomial $A_1^r$
occurs in the dense representation of $\mathcal{R}$. Finally, if
$\mathrm{char}(\fq)$ divides $r-1$, then we have the nonzero
monomial $A_0^{r-1}$ in the dense representation of $\mathcal{R}$.

Next we show that $({\sf H}_5)$ is fulfilled. For this purpose, we
first prove that $A_0, L_1 \klk L_m$,
$\mathrm{Disc}(F(\bfs{A}_0,T))$ form a regular sequence of
$\cfq[A_{r-1} \klk A_0]$. We observe that
$$\cfq[A_{r-1}\klk
A_0]/(A_0, L_1 \klk L_m) \simeq \cfq[A_k : k \in \mathcal{J}_1]$$
is domain, where $\mathcal{J}_1:=\mathcal{J}\setminus \{0\}$. Hence,
considering the class $\mathcal{R}_1$ of
$\mathrm{Disc}(F(\bfs{A}_0,T))$ as an element of $\cfq[A_k : k \in
\mathcal{J}_1]$, it is enough to prove that it is nonzero. Indeed,
if $\mathrm{char}(\fq)$ does not divide $r(r-1)$, then the monomial
$(-1)^{r-1}(r-1)^{r-1}A_1^r$ occurs in the dense representation
$\mathcal{R}_1$, while for $\mathrm{char}(\fq)$ dividing $r$, the
monomial $A_1^r$ appears in $\mathcal{R}_1$. Finally, for $n \geq 3$
and $\mathrm{char}(\fq)$ dividing $r-1$, we have the nonzero
monomial $(-1)^{r+1}A_1^2A_2^{r-1}$ in the dense representation of
$\mathcal{R}_1$.

Finally we prove that $L_1 \klk
L_m,\mathrm{Disc}(F(\bfs{A}_0,T)),\mathrm{Subdisc}(F(\bfs{A}_0,T))$
form a regular sequence in $\cfq[A_{r-1} \klk A_0]$. Recall that
$\cfq[A_{r-1} \klk A_0]/(L_1 \klk L_m) \simeq \cfq[A_k: k \in
\mathcal{J}]$ is a domain. Therefore, we may consider the classes
$\mathcal{R}$ and $\mathcal{S}_1$ of $\mathrm{Disc}(F(\bfs{A}_0,T))$
and $\mathrm{Subdisc}(F(\bfs{A}_0,T))$ modulo $(L_1 \klk L_m)$ as
elements of $\cfq[A_k: k \in \mathcal{J}]$. We have already shown
that $\mathcal{R}$ is nonzero. On the other hand, if
$\mathrm{char}(\fq)$ does not divide $r(r-1)$, then the nonzero
monomial $r(r-1)^{r-2}A_1^{r-2}$ occurs in the dense representation
of $\mathcal{S}_1$, while for $\mathrm{char}(\fq)$ dividing
$r(r-1)$, we have the nonzero monomial $2(-1)^r(r-2)^{r-2}A_2^{r-1}$
in the dense representation of $\mathcal{S}_1$. We conclude that
$\mathcal{S}_1$ is nonzero.

Further, \cite[Theorem A.3]{MaPePr14} or \cite[Teorema
3.1.7]{Perez16} show that $\mathcal{R}$ is an irreducible element of
$\cfq[A_k: k \in \mathcal{J}]$ and hence $ \mathbb{B}:=\cfq[A_k: k
\in \mathcal{J}]/(\mathcal{R})$ is a domain. Thus, it suffices to
see that the class of $\mathcal{S}_1$ in $\mathbb{B}$ is nonzero. If
not, then $\mathcal{S}_1$ would be a nonzero multiple of
$\mathcal{R}$ in $\cfq[A_k: k \in \mathcal{J}]$, which is not
possible because
$\text{max}\{\deg_{A_1}\mathcal{R},\deg_{A_2}\mathcal{R}\}=r$ and
$\text{max}\{\deg_{A_1}\mathcal{S}_1,\deg_{A_2}\mathcal{S}_1\}=
r-1$.

Finally, we prove that $({\sf H}_6)$ holds. The components of
highest weight of $L_1,\klk,L_m$ being of the form
$L_k^{\wt}=b_{k,r-i_k}A_{r-i_k}$ for $k=1 \klk m$, arguing as before
we readily see that $({\sf H}_6)$ holds.
\end{proof}

From Lemma \ref{lemma: example MPP satisfies H1-H6} it follows that
the family $\mathcal{A}$ of \eqref{eq: family paper CMP17} satisfies
the hypotheses of Theorem \ref{theorem: estimate fact patterns}.
Therefore, applying Theorem \ref{theorem: estimate fact patterns} we
obtain the following result.
\begin{theorem}
Suppose that $\mathrm{char}(\fq)>3$. Let $\mathcal{A}$ be the family
of \eqref{eq: family paper CMP17} and $\bfs{\lambda}$ a
factorization pattern. If either $\mathrm{char}(\fq)$ does not
divide $r(r-1)$ and $L_k\in \fq[A_{r-1},\ldots,A_2]$ for $1\leq
k\leq m$, or $L_k\in \fq[A_{r-1},\ldots,A_n]$ for $1\leq k\leq m$
and $ 3 \leq n \leq r-m$, then
\begin{align*}
\big||\mathcal{A}_{\bfs\lambda}^{sq}|
-\mathcal{T}(\bfs\lambda)\,q^{r-m}\big|& \leq
\mathcal{T}(\bfs\lambda)q^{r-m-1}
\big((\delta (D-2)+2)q^{\frac{1}{2}}+14 D^2 \delta^2+ r^2\delta\big),\\
\big||\mathcal{A}_{\bfs\lambda}|
-\mathcal{T}(\bfs\lambda)\,q^{r-m}\big|&\le
q^{r-m-1}\Big(\mathcal{T}(\bfs\lambda) \big((\delta
(D-2)+2)q^{\frac{1}{2}}+14 D^2
\delta^2+r^2\delta\big)+r^2\delta\Big),
\end{align*}
where $\delta:=\prod_{j=1}^m i_j$ and $D:=\sum_{j=1}^m(i_j-1)$.
\end{theorem}
%
\subsection{A linear family from \cite{GaHoPa99}}
In \cite{GaHoPa99} there are experimental results on the number of
irreducible polynomials on certain families over $\fq$. Further, the
distribution of factorization patterns on general families of
polynomials of $\fq[T]$ of a given degree is stated as an open
problem. In particular, the family of polynomials we now discuss is
considered.

Suppose that $\mathrm{char}(\fq)>3$. For positive integers $s$ and
$r$ with $3\leq s \leq r-2$, let
\begin{equation}\label{eq: linear family GaoPanarioHowell}
\mathcal{A}:=\{T^r+g(T)T+1: \,\, g\in \fq[T] \,\, \text{and}\, \deg
g\leq s-1\}.
\end{equation}
Observe that $\mathcal{A}$ is isomorphic to the set of
$\fq$-rational points of the affine $\fq$--subvariety of $\A^r$
defined by the polynomials
$$G_1:=A_{0}-1,\ G_2:=A_{s+1},\ldots, G_{r-s}:=A_{r-1}.$$

We show that hypotheses $({\sf H}_1)$--$({\sf H}_6)$ are fulfilled.
It is easy to see that  $({\sf H}_1)$ and $({\sf H}_2)$ hold, since
$G_1\klk G_{r-s}$ are linearly--independent polynomials of degree 1.
Furthermore, taking into account that
$$G_1^{\wt}=A_{0},\ G_2^{\wt}=A_{s+1},\ldots, G_{r-s}^{\wt}=A_{r-1},$$
we immediately conclude that hypothesis $({\sf H}_3)$ holds.

Now we analyze the validity of hypotheses $({\sf H}_4)$ and $({\sf
H}_5)$. Let $W\subset \A^r$ be the $\fq$--variety defined by the
polynomials $G_1,\ldots,G_{r-s}$, and denote by
$\mathcal{D}(W)\subset\A^r$ and $\mathcal{S}_1(W)\subset \A^r$ the
discriminant locus and the first subdiscriminant locus of $W$
respectively.

We first prove that $\mathcal{D}(W)$ has codimension one in $W$. It
is clear that $G_1,\ldots,G_{r-s}$ form a regular sequence of
$\fq[A_{r-1},\ldots,A_0]$. Observe that
$$\cfq[W]=\cfq[A_{r-1},\ldots,A_0]/(G_1,\ldots,G_{r-s})\simeq
\cfq[A_s,\ldots,A_1]$$
is a domain. As a consequence, we may consider the coordinate
function $\mathcal{R}$ defined by $\mathrm{Disc}(F(\bfs{A}_0,T))$ as
an element of $\cfq[A_s \klk A_1]$, where
$\bfs{A}_0:=(A_{r-1},\ldots,A_0)$ and
$F(\bfs{A}_0,T):=T^r+A_{r-1}T^{r-1}+\dots+ A_0$. We observe that
$\mathcal{R}\not=0$ in $\cfq[A_s \klk A_1]$, because $F(\bfs A_0,T)$
is not a separable polynomial, and therefore it is not a zero
divisor of $\cfq[W]$. It follows that $\mathcal{D}(W)$ has
codimension one in $W$, namely hypothesis $({\sf H}_4)$ holds.

Next we show that $(A_0\cdot\mathcal{S}_1)(W)$ has codimension at
least one in $\mathcal{D}(W)$. Since $G_1:=A_0-1$ vanishes on $W$,
the coordinate function of $\cfq[W]$ defined by $A_0$ is a unit,
which implies $(A_0\cdot\mathcal{S}_1)(W)=\mathcal{S}_1(W)$.

In what follows, we shall use the following elementary property.
\begin{lemma}\label{lemma: regular sequences 1}
Let $F_1 \klk F_m \in \cfq[A_{0} \klk A_{r-1}]$. If $F_1 \klk F_m$
form a regular sequence in $\cfq(A_0 \klk A_i)[A_{i+1} \klk
A_{r-1}]$, then $F_1 \klk F_m$ form a regular sequence in $\cfq[A_0
\klk A_{r-1}]$.
\end{lemma}

We shall also use the following property of regular sequences.
\begin{lemma} \label{lemma: regular sequences 2}
Let $F_1 \klk F_m \in \cfq[A_0 \klk A_{r-1}]$. For an assignment of
positive integer weights $\wt$ to the variables $A_0 \klk A_{r-1}$,
denote by $F_1^{\wt} \klk F_m^{\wt}$ the components of highest
weight of $F_1 \klk F_m$. If $F_1^{\wt} \klk F_m^{\wt}$ form a
regular sequence in $\cfq[A_0 \klk A_{r-1}]$, then $F_1 \klk F_m$
form a regular sequence in $\cfq[A_0 \klk A_{r-1}]$.
\end{lemma}
\begin{proof}
Let $V_j:=V(F_1\klk F_j)\subset \A^r$ for $1\le j\le m$. It is
enough to see that $V_j$ has codimension $j$ for $1 \leq j \leq m$.
By hypothesis, $V_j^{\wt}:=V(F_1^{\wt} \klk F_j^{\wt})\subset \A^r$
has pure dimension $r-j$. Therefore, there exist $1\le {k_1}<\cdots<
k_{r-j}\le m$ such that the variety $V:=V(F_1^{\wt} \klk
F_j^{\wt},A_{k_1} \klk A_{k_{r-j}} )\subset \A^r$ has dimension
zero. Consider the following morphism of affine $\fq$--varieties:
    \begin{align*}
    {\bfs\phi}: \A^r & \rightarrow  \A^{r}\\
    (a_0 \klk a_{r-1}) & \mapsto  (a_0^{\wt(0)},
    a_1^{\wt(1)},\ldots,a_{r-1}^{\wt(r-1)}),
    \end{align*}
where $\wt(0)\klk\wt(r-1)$ are the weights assigned to $A_0\klk
A_{r-1}$ respectively. It is clear that $\bfs \phi$ is a finite,
dominant morphism. 
Observe that, if $F\in\cfq[A_0,\ldots,A_{r-1}]$ is weighted
homogeneous, then $\bfs \phi(F)$ is homogeneous.

We have that $\bfs\phi(V) \subset \A^r$ is a zero--dimensional
affine cone. Since $\bfs\phi(V)$ is defined by the homogeneous
polynomials $F_i^{\wt}(A_0^{\wt(0)},\ldots,A_{r-1}^{\wt(r-1)})$,
$1\leq i \leq j$, and $A_{k_i}  ^{\wt(k_{i})}$, $1\leq i \leq r-j$,
it must be $\bfs \phi(V)=\{0\}$. Therefore, by, e.g.,
\cite[Proposition 18]{PaSa04}, the affine variety defined by the
polynomials
$$F_1(A_0^{\wt(0)},\ldots,A_{r-1}^{\wt(r-1)}),\ldots,
F_j(A_0^{\wt(0)},\ldots,A_{r-1}^{\wt(r-1)}),A_{k_1}^{\wt(k_1)} \klk
A_{k_{r-j}}^{\wt(k_{r-j})}$$
has dimension zero. Taking into account that $\bfs\phi$ is a finite
morphism, we conclude that the variety $\hat{V}_j\subset\A^r$
defined by $F_1,\ldots,F_j, A_{k_1},\ldots, A_{k_{r-j}}$ has also
dimension zero.

Finally, observe that the dimension of $V_j$ is at least $r-j$. On
the other hand, $0=\dim\hat{V}_j\ge\dim V_j-(r-j)$. This finishes
the proof of the lemma.
\end{proof}

It easy to see that $G_2 \klk G_{r-s}$ form a regular sequence in
$\cfq[A_{r-1} \klk A_0]$. Observe that $\cfq[A_{r-1} \klk A_0]/(G_2
\klk G_{r-s}) \simeq \cfq[A_{s}\klk A_0]$. Therefore, to conclude
that $({\sf H}_5)$ holds it suffices to prove that $\mathcal{G}_1$,
$\mathcal{S}_1$ and $\mathcal{R}$ form a regular sequence in
$\cfq[A_{s} \klk A_0]$, where $\mathcal{G}_1$, $\mathcal{R}$ and
$\mathcal{S}_1$ are the coordinate functions of $\cfq[A_{r-1} \klk
A_0]/(G_2 \klk G_{r-s})$ defined by $G_1$, $\mathrm{Disc}(F(\bfs
A_0,T))$ and $\mathrm{Subdisc}(F(\bfs A_0,T))$, respectively.
\begin{lemma}\label{lemma: H5 ej lineal: regular sequence}
$\mathcal{G}_1$, $\mathcal{S}_1$ and $\mathcal{R}$ form a regular
sequence in $\cfq[A_{s} \klk A_0]$.
\end{lemma}
\begin{proof}
We consider $\mathcal{R}, \mathcal{S}_1, \mathcal{G}_1$ as elements
of $\cfq(A_s \klk A_{i+1})[A_i,\ldots,A_0]$ for an appropriate
$i\in\{2,3\}$ and define a weight $\wt_i$ by setting
$$\wt_i(A_0):=r,\ \wt_i(A_1):=r-1,\ldots,\wt_i(A_i):=r-i.$$
Denote by $\mathcal{G}_1^{\wt_i}$, $\mathcal{R}^{\wt_i}$ and
$\mathcal{S}_1^{\wt_i}$ the components of highest weight of
$\mathcal{G}_1$, $\mathcal{R}$ and $\mathcal{S}_1$ respectively. We
have the following claim.
\begin{claim}
$\mathcal{G}_1^{\wt_i}$, $\mathcal{S}_1^{\wt_i}$ and
$\mathcal{R}^{\wt_i}$ form a regular sequence in $\cfq(A_s \klk
A_{i+1})[A_i,\ldots,A_0]$.
\end{claim}
\begin{proof}[Proof of Claim] Observe that
$$\cfq(A_s \klk A_{i+1})[A_i,\ldots,A_0]/(\mathcal{G}_1^{\wt_i}) \simeq
\cfq(A_s \klk A_{i+1})[A_i,\ldots,A_1]$$
is a domain. As a consequence, it suffices to prove that the
coordinate functions defined by $\mathcal{S}_1^{\wt_i}$ and
$\mathcal{R}^{\wt_i}$ in this quotient ring form a regular sequence.
With a slight abuse of notation, we shall also denote them by
$\mathcal{S}_1^{\wt_i}$ and $\mathcal{R}^{\wt_i}$.

The proof will be split into four parts, according to whether
$\mathrm{char}(\fq)$ divides $r$, $r-1$, $r-2$ or does not divide
$r(r-1)(r-2)$.\medskip

\noindent{\bf First case: $\mathrm{char}(\fq)$ divides $r$}. For
$i:=2$, it is easy to see that in $\cfq(A_s \klk A_3)[A_2,A_1]$,
\begin{align}
\mathcal{R}^{\wt_2}=A_1^r+(-1)^{r+1}2^{r-2}A_2^{r-1}A_1 ^2&\quad
\text{and} \quad\mathcal{S}_1^{\wt_2}=(2A_{2})^{r-1}.
\end{align}
Observe that $\mathcal{S}_1^{\wt_2}$ is a nonzero polynomial of
$\cfq(A_s \klk A_3)[A_2,A_1]$, and
$$\cfq(A_s \klk A_3)[A_2,A_1]/(\mathcal{S}_1^{\wt_2}) \simeq \cfq(A_s
\klk A_3)[A_1].$$
It follows that $\mathcal{R}^{\wt_2}$ is not a zero divisor in
$\cfq(A_s \klk A_3)[A_2,A_1]/(\mathcal{S}_1^{\wt_2})$, which
completes the proof of the claim in this case.
\medskip

\noindent{\bf Second case: $\mathrm{char}(\fq)$ divides $r-1$}. For
$i:=3$, we prove that $\mathcal{S}_1^{\wt_3}$ and
$\mathcal{R}^{\wt_3}$ form a regular sequence in $\cfq(A_s \klk
A_4)[A_3, A_2,A_1]$. Let $F:=T^r+A_3T^3+A_2T^2+A_1T$. It is easy to
see that $\mathcal{R}^{\wt_3}=\mathrm{Disc}(F)$ and
$\mathcal{S}_1^{\wt_3}=\mathrm{Subdisc}(F)$. Observe that
$F'=T^{r-1}+3A_3T^3+2A_2T^2+A_1$. By \cite[Lemma 7.1]{GeCzLa92} we
deduce that
$$\mathcal{R}^{\wt_3}=(-1)^{{r}({r-1})} \mathrm{Res}(F',G)\ \textrm{ and }
\ \mathcal{S}_1^{\wt_3}=(-1)^{(r-1)(r-2)}\mathrm{Subdisc}(F',G),$$
where $G:=-2A_3T^3-A_2T^2$ is the remainder of the division of $F$
by $F'$. Therefore, applying the Poisson formula, it is easy to see
that
$$\mathcal{R}^{\wt_3}=(-1)^{r+1}A_1^2A_2^{r-1}+
2^{r-1}A_1^2A_2^2A_3^{r-2}-2^{r-3}A_1^3A_3^{r-1}.$$
On the other hand, by, e.g., Theorem \cite[Theorem 2.5]{DaKrSz13},
we conclude that
\begin{align*}
\mathcal{S}_1^{\wt_3}&=2
A_2^{r-1}+(-1)^r2^{r-2}A_2^2A_3^{r-2}+2A_1A_2^{r-3}A_3+3
(-1)^{r+1}2^{r-2}A_1A_3^{r-1}\\
&=2\big(
A_2^{r-1}+A_1A_2^{r-3}A_3\big)+(-2)^{r-2}\big(A_2^2A_3^{r-2}-3
A_1A_3^{r-1}\big).
\end{align*}
In the second line we express $\mathcal{S}_1^{\wt_3}$ as the sum of
two homogeneous polynomials of degrees $r-1$ and $r$ without common
factors. Then \cite[Lemma 3.15]{Gibson98} proves that
$\mathcal{S}_1^{\wt_3}$ is an irreducible polynomial in $\cfq(A_s
\klk A_4)[A_3,A_2,A_1]$. Next suppose that $\mathcal{R}^{\wt_3}$ is
a zero divisor in $\cfq(A_s \klk
A_4)[A_3,A_2,A_1]/(\mathcal{S}_1^{\wt_3})$. Since
$\mathcal{S}_1^{\wt_3}$ is irreducible,  we have that
$\mathcal{R}^{\wt_3} \in (\mathcal{S}_1^{\wt_3})$, which is easily
shown to be not possible by a direct calculation. \medskip

\noindent{\bf Third case: $\mathrm{char}(\fq)$ divides $r-2$}. For
$i:=3$, we show that $\mathcal{S}_1^{\wt_3}$ and
$\mathcal{R}^{\wt_3}$ form a regular sequence in $\cfq(A_s \klk
A_4)[A_3, A_2,A_1]$. As in the previous case, if
$F:=T^r+A_3T^3+A_2T^2+A_1T$, then it can be seen that
$\mathcal{R}^{\wt_3}=\mathrm{Disc}(F)$ and
$\mathcal{S}_1^{\wt_3}:=\mathrm{Subdisc}(F)$. Since
$F'=2T^{r-1}+3A_3T^3+2A_2T^2+A_1$, from \cite[Lemma 7.1]{GeCzLa92}
it follows that
$$\mathcal{R}^{\wt_3}=(-1)^{{r}({r-1})} 2^{r-3}\mathrm{Res}(F',G)\ \textrm{ and }
\ \mathcal{S}_1^{\wt_3}=
(-1)^{(r-1)(r-2)}2^{r-3}\mathrm{Subdisc}(F',G),$$
where $G:=-\frac{1}{2}A_3T^3+\frac{1}{2}A_1T$ is the remainder the
division of $F$ by $F'$. By the Poisson formula we obtain

$$\mathcal{R}^{\wt_3}=\left\{
\begin{array}{rl}
4 A_1^3A_3^{r-1}-A_1^r
-2A_2A_1{\!\!}^{\frac{r+2}{2}}A_3{\!\!}^{\frac{r-2}{2}}-A_1^2A_2^2A_3^{r-2}&\textrm{
for }r\textrm{ even},\\[1ex]
4A_1^3A_3^{r-1}+A_1^r+4A_1^{\frac{r+3}{2}}A_3{\!\!}^{\frac{r-1}{2}}-A_1^2A_2^2A_3^{r-2}&\textrm{
for }r\textrm{ odd}.
\end{array}
\right.$$
In the same vein, by, e.g., \cite[Theorem 2.5]{DaKrSz13}, we see
that
$$\mathcal{S}_1^{\wt_3}=\left\{
\begin{array}{rl}
4A_2(A_1A_3)^{\frac{r-2}{2}}
+2A_2^2A_3^{r-2}+2A_1^{r-2}-6A_1A_3^{r-2}&\textrm{ for }r\textrm{
even},\\[1ex] 7(A_1A_3)^{\frac{r-1}{2}}
-2A_2^2A_3^{r-2}+2A_1^{r-2}+6A_1A_3^{r-1}&\textrm{ for }r\textrm{
odd}.
\end{array}
\right.$$
We observe that $\mathcal{S}_1^{\wt_3}$ is an irreducible polynomial
in $\cfq(A_s \klk A_4)[A_3,A_2,A_1]$. To see this it suffices to
apply the Eisenstein criterion, considering $\mathcal{S}_1^{\wt_3}$
as an element of the polynomial ring  $\cfq((A_s \klk
A_4)[A_3,A_1])[A_2]$ and the prime $(A_1)$. Next, suppose that
$\mathcal{R}^{\wt_3}$ is a zero divisor in $\cfq(A_s \klk
A_4)[A_3,A_2,A_1]/(\mathcal{S}_1^{\wt_3})$. Since
$\mathcal{S}_1^{\wt_3}$ is irreducible,  we have that
$\mathcal{R}^{\wt_3} \in (\mathcal{S}_1^{\wt_3})$, which can be
shown to be not possible by a direct calculation.
\medskip

\noindent{\bf Fourth case: $\mathrm{char}(\fq)$ does not divide
$r(r-1)(r-2)$}. For $i:=2$, we prove that $\mathcal{S}_1^{\wt_2}$
and $\mathcal{R}^{\wt_2}$ form a regular sequences in $\cfq(A_s \klk
A_3)[A_2,A_1]$. Arguing as before, we obtain
\begin{align*}
\mathcal{R}^{\wt_2}&=(1-r)^{r-1}A_1^r-(r-2)r^{-1}A_1^2A_2^{r-1},\\
\mathcal{S}_1^{\wt_2}&=r(r-1)^{r-2}A_1^{r-2}+2(2-r)^{r-2}A_2^{r-1}.
\end{align*}
By the Stepanov criterion (see, e.g., \cite[Lemma 6.54]{LiNi83}) we
deduce that $\mathcal{S}_1^{\wt_2}$ is an irreducible polynomial in
$\cfq(A_s \klk A_3)[A_2,A_1]$. Suppose that $\mathcal{R}^{\wt}$ is a
zero divisor in $\cfq(A_s \klk
A_3)[A_2,A_1]/(\mathcal{S}_1^{\wt_2})$. Since
$\mathcal{S}_1^{\wt_2}$ is irreducible, we have that
$\mathcal{R}^{\wt_2} \in (\mathcal{S}_1^{\wt_2})$, which can be seen
not to be the case by a direct calculation. Therefore, we deduce
that $\mathcal{S}_1^{\wt_2}$ and $\mathcal{R}^{\wt_2}$ form a
regular sequence in $\cfq(A_s \klk A_3)[A_2, A_1]$.
\end{proof}

By the claim and Lemma \ref{lemma: regular sequences 2} it follows
that $\mathcal{G}_1$, $\mathcal{S}_1$ and $\mathcal{R}$ form a
regular sequence in $\cfq(A_s \klk A_{i+1})[A_i,\ldots,A_0]$, and
Lemma \ref{lemma: regular sequences 1} implies that $\mathcal{G}_1$,
$\mathcal{S}_1$ and $\mathcal{R}$ form a regular sequence in
$\cfq[A_s \klk A_0]$.
\end{proof}

By Lemma \ref{lemma: H5 ej lineal: regular sequence} we conclude
that hypothesis $({\sf H}_5)$ holds. Finally, we prove that
hypothesis $({\sf H}_6)$ holds. The components of higher weight of
the polynomials $G_1,\ldots,G_{r-s}$ are $G_i^{\wt}=A_{s+i-1}$ for
$2\leq i \leq r-s$ and $G_1^{\wt}= A_0$. With the same arguments as
above, we see that $\mathcal{D}(W^{\wt})$ has codimension at least
one in $W^{\wt}$, where
$W^{\wt}:=V(G_1^{\wt},\ldots,G_{r-s}^{\wt})$.

Since the family \eqref{eq: linear family GaoPanarioHowell}
satisfies hypotheses $({\sf H}_1)$--$({\sf H}_6)$, from Theorem
\ref{theorem: estimate fact patterns} we deduce the following
result.
\begin{theorem}
Let $\mathcal{A}$ be the family \eqref{eq: linear family
GaoPanarioHowell} and $\bfs\lambda$ a factorization pattern. We have
    \begin{align*}
    \big||\mathcal{A}_{\bfs\lambda}^{sq}|
    -\mathcal{T}(\bfs\lambda)\,q^s\big|& \leq \mathcal{T}(\bfs\lambda)q^{s-1}
    \big((\delta (D-2)+2)q^{\frac{1}{2}}+14 D^2 \delta^2+ r^2\delta\big),\\
    \big||\mathcal{A}_{\bfs\lambda}|
    -\mathcal{T}(\bfs\lambda)\,q^s\big|&\le
    q^{s-1}\Big(\mathcal{T}(\bfs\lambda)
    \big((\delta (D-2)+2)q^{\frac{1}{2}}+14 D^2 \delta^2+r^2\delta\big)+r^2\delta\Big),
    \end{align*}
    where $\mathcal{A}_{\bfs\lambda}$ is the set of elements of $\mathcal{A}$
    with factorization pattern $\bfs\lambda$,
    $\mathcal{A}_{\bfs\lambda}^{sq}$ is
    the set of square--free elements of $\mathcal{A}_{\bfs\lambda}$,
    $\delta:=r\cdot(r-s-1)!$ and $D:=r-1+{(r-s-2)(r-s-1)}/{2}$.
\end{theorem}
\begin{proof}
We apply Theorem \ref{theorem: estimate fact patterns} with $m:=r-s$
to the polynomials
$$R_1:=(-1)^r\Pi_r-1,\ R_2:=(-1)^{r-s-1}\Pi_{r-s-1},\ldots, R_{r-s}:=-\Pi_1.$$
Therefore, we have
$$\delta:=\prod_{i=1}^{r-s} \deg R_i=r\cdot(r-s-1)!\textrm{ and }
D:=\sum_{i=1}^{r-s}(\deg R_i-1)=r-1+\frac{(r-s-2)(r-s-1)}{2}.$$
This finishes the proof.
\end{proof}
%
\subsection{A nonlinear family}
Let $r, t_1 \klk t_r$ be positive integers with $r$ even. Suppose
that $\mathrm{char}(\fq)>3$ does not divide
$(r-1)(r+1)\big((r-1)^{r-1}+r^r\big)$. Consider the polynomial $G\in
\fq[A_1,\ldots,A_r]$ defined in the following way:
$$G:=\sum_{t_1+2t_2 +\ldots +r t_r=r}
(-1)^{\Delta(t_1,\ldots,t_r)}\frac{(t_1 +\dots+ t_r)!}{t_1!  \dots
t_r!} A_r^{t_1} \cdots A_1^{t_r},$$
where $\Delta(t_1,t_2,\ldots,t_r):= r-\sum_{i=1}^r t_i$. The
polynomial $G$ arises as the determinant of the $n \times n$ generic
Toeplitz--Hessenberg matrix, namely
$$
G=\det    \left(
    \begin{array}{ccccc}
        A_r & 1 & 0 &  \dots & 0
        \\
        \vdots &\ddots  & \ddots & \ddots & \vdots
        \\
        \vdots &  & \ddots & \ddots & 0
        \\
        A_1 & \ldots & \ldots & A_r&\!\! 1
    \end{array}
    \!\!\right).
$$
This is the well--known Trudi formula (see \cite[Ch. VII]{Muir60};
see also \cite[Theorem 1]{Merca13}). We also remark that the
polynomial $H_r:=G(\Pi_r,\ldots,\Pi_1)$ is critical in the study of
deep holes of the standard Reed--Solomon codes (see
\cite[Proposition 2.2]{CaMaPr12}).

We consider the following family of polynomials:
\begin{equation}\label{eq: nonlinear family}
\mathcal{A}_{\mathcal{N}}:=\{T^{r+1} +a_rT^r+ \dots+a_0:
G(a_r,\ldots, a_1)=0\}.
\end{equation}
Observe $\mathcal{A}_{\mathcal{N}}$ may be seen as the set of
$\fq$--rational points of the $\fq$--variety $W:=V(G) \subset
\A^{r+1}$. Let $\wt$ be the weight defined by $\wt(A_i):=r+1-i$ for
$i=0 \klk r$. We shall see that this family of polynomials satisfies
hypotheses $({\sf H}_1)$--$({\sf H}_6)$.

It is clear that $({\sf H}_1)$ holds, because $G$ is nonzero.
Further, since $G$ is a monic element of $\fq[A_r,\ldots,A_2][A_1]$
of degree $1$ in $A_1$, we have that
$$\nabla G(\bfs a_0)=
\bigg(\frac{\partial{G}}{\partial{A_r}}(\bfs a_0)\klk
\frac{\partial{G}}{\partial A_2}(\bfs a_0),1\bigg) \neq 0$$
for any $\bfs a_0 \in W$. We deduce that hypothesis $({\sf H}_2)$
holds.

Next we consider hypothesis $({\sf H}_3)$. Given an arbitrary
nonzero monomial
$$m_{G}:=\frac{(t_1 +\dots+ t_r)!}{t_1!  \dots
t_r!} A_r^{t_1} \dots A_1^{t_r}$$
arising in the dense representation of $G$, it is easy to see that
$\wt(m_G)=r$. It follows that $G$ is weighted homogeneous of
weighted degree $r$. Then $G^{\wt}=G$, which readily implies that
hypothesis $({\sf H}_3)$ holds.

Now we analyze the validity of hypothesis $({\sf H}_4)$, namely that
the discriminant locus $\mathcal{D}(W) \subset \A^{n+1}$ of $W$ has
codimension at least 1 in $W$. For this purpose, it suffices to show
that $\{G, \mathcal{R}\}$ form a regular sequence in $\cfq[A_r \klk
A_0]$, where $\mathcal{R}:=\mathrm{Disc}(F(\bfs A_0, T))$, $F(\bfs
A_0,T):=T^{r+1}+A_rT^r +\dots+ A_0$ and $\bfs A_0:=(A_r \klk A_0)$.

We consider $G$ and $\mathcal{R}$ as elements of the polynomial ring
$\cfq(A_r,\ldots, A_2)[A_1, A_0]$ and the weight $\wt_1$ on
$\cfq(A_r,\ldots, A_2)[A_1, A_0]$ defined by setting
$$\wt_1(A_1):=r,\quad \wt_1(A_0):=r+1.$$
We claim that $G^{\wt_1}, \mathcal{R}^{\wt_1}$ form a regular
sequence in $\cfq(A_r,\klk A_2)[A_1,A_0]$. It is easy see that
$G^{\wt_1}=A_1$. Further, since $\cfq(A_r \klk A_2)[A_1,
A_0]/(G^{\wt_1}) \simeq \cfq(A_r \klk A_2)[A_0]$ is a domain, to
prove the claim it suffices to show that ${\mathcal{R}^{\wt_1}}$ is
nonzero modulo $(A_1)$. A direct calculation shows that
${\mathcal{R}^{\wt}}=(r+1)^{r+1}A_0^{r+1}$ modulo $(A_1)$, which
proves the claim. As a consequence of the claim and Lemma
\ref{lemma: regular sequences 2} we see that $G$ and $\mathcal{R}$
form a regular sequence in $\cfq(A_r \klk A_2)[A_1,A_0]$, and Lemma
\ref{lemma: regular sequences 1} implies that $G$ and $\mathcal{R}$
form a regular sequence in $\cfq[A_r \klk A_0]$. In other words,
hypothesis $({\sf H}_4)$ is satisfied.


Next we show that hypothesis $({\sf H}_5)$ holds. To this end, we
make the following claim.
\begin{claim}
$A_0$, $\mathcal{R}$ and $G$ form a regular sequence of
$\cfq[A_r,\ldots,A_0]$.
\end{claim}
\begin{proof}
Since $\cfq[A_r,\ldots,A_0]/(A_0)\simeq \cfq[A_r,\ldots,A_1]$ and
$G\in \cfq[A_r,\ldots,A_1]$, we have to show that $\mathcal{R}$
modulo $(A_0)$, and $G$, form a regular sequence in
$\cfq[A_r,\ldots,A_1]$. We consider $G$ and $\mathcal{R}$ modulo
$(A_0)$ as elements of $\cfq(A_{r-1},\ldots,A_2)[A_r,A_1]$, with the
weight $\wt_r$ defined by $\wt_r(A_r):=1$ and $\wt_r(A_1):=r$. We
claim that $G^{\wt_r}$ and $\mathcal{R}^{\wt_r}$ form a regular
sequence in $\cfq(A_{r-1},\ldots,A_2)[A_r,A_1]$. First we observe
that
$$G^{\wt_r}=A_1+A_r^r,$$
and the Stepanov criterion (see, e.g., \cite[Lemma 6.54]{LiNi83})
proves that $G^{\wt_r}$ is an irreducible polynomial of
$\cfq(A_{r-1},\ldots,A_2)[A_r,A_1]$. Thus, it is enough to prove
that $\mathcal{R}^{\wt_r}$ is a nonzero polynomial of
$\cfq(A_{r-1},\ldots,A_2)[A_r,A_1]/(G^{\wt_r})$.
%
%
%
We have
\begin{align*}
\mathcal{R}^{\wt_r}&= - (r-1)^{r-1} A_r^rA_1^r+
r^rA_1^{r+1}\\
&\equiv -\big((r-1)^{r-1}+r^r\big)A_r^{r+r^2}
\textrm{ modulo }G^{\wt_r}.
\end{align*}
We conclude that $G^{\wt_r}$ and $\mathcal{R}^{\wt_r}$ form a
regular sequence in $\cfq(A_{r-1},\ldots,A_2)[A_r,A_1]$. Combining
Lemmas \ref{lemma: regular sequences 2} and \ref{lemma: regular
sequences 1} as before we deduce that $G$ and $\mathcal{R}$ modulo
$(A_0)$ form a regular sequence in $\cfq[A_r \klk A_1]$, which
implies that $A_0$, $\mathcal{R}$ and $G$ form a regular sequence of
$\cfq[A_r,\ldots,A_0]$.
\end{proof}
Next we make a second claim.
\begin{claim}
$G$, $\mathcal{R}$ and $\mathcal{S}_1$ form a regular sequence of
$\cfq[A_r \klk A_0]$.
\end{claim}
%
\begin{proof}
We consider $G$, $\mathcal{R}$ and $\mathcal{S}_1$ as elements of
$\cfq(A_r \klk A_3)[A_2,A_1,A_0]$, and consider the weight $\wt_2$
defined by $\wt_2(A_2):=r-1$, $\wt_2(A_1):=r$, $\wt_2(A_0):=r+1$. We
claim that $G^{\wt_2}$, $\mathcal{S}_1^{\wt_2}$ and
$\mathcal{R}^{\wt_2}$ form a regular sequence in $\cfq(A_r \klk
A_3)[A_2,A_1,A_0]$. Since $G^{\wt_2}=A_1$, we have that $\cfq(A_r
\klk A_3)[A_2,A_1,A_0]/(G^{\wt_2}) \simeq \cfq(A_r \klk
A_3)[A_2,A_0]$ is a domain. Therefore, it suffices to see that
$\mathcal{S}_1^{\wt_2}$ modulo $(A_1)$ and $\mathcal{R}^{\wt_2}$
modulo $(A_1)$ form a regular sequence in $\cfq(A_r \klk
A_3)[A_2,A_0]$. It is easy to see that
$$
\mathcal{S}_1^{\wt_2}\textrm{ modulo }(A_1)=-2(r-1)^{r-1}A_2^r.$$
Further, we have $\mathcal{R}^{\wt_2}\textrm{ modulo
}(A_1,A_2)=(r+1)^{r+1}A_{0}^r$. As a consequence, $G^{\wt_2}$,
$\mathcal{S}_1^{\wt_2}$ and  $\mathcal{R}^{\wt_2}$ form a regular
sequence in $\cfq(A_r \klk A_3)[A_2, A_1, A_0]$. From Lemmas
\ref{lemma: regular sequences 2} and \ref{lemma: regular sequences
1} it follows that $G$, $\mathcal{S}_1$ and $\mathcal{R}$ form a
regular sequence in $\cfq[A_r \klk A_0]$.
\end{proof}

From the first claim we conclude that $\mathcal{D}(W)\cap\{A_0=0\}$
has codimension two in $W$, while the second claim shows that
$\mathcal{S}_1(W)$ has codimension two in $W$. As a consequence,
$\mathcal{D}(W)\cap(A_0\cdot\mathcal{S}_1)(W)$ has codimension two
in $W$, that is, hypothesis $({\sf H}_5)$ is satisfied.

Finally, since $G^{\wt}=G$, we readily deduce that hypothesis $({\sf
H}_6)$ holds.

As a consequence of the fact that the family \eqref{eq: nonlinear
family} satisfies hypotheses $({\sf H}_1)$--$({\sf H}_6)$, we obtain
the following result.
\begin{theorem}
Let $\mathcal{A}_\mathcal{N}$ be the family \eqref{eq: nonlinear
family} and $\bfs\lambda$ a factorization pattern. We have
    \begin{align*}
    \big||\mathcal{A}_{\mathcal{N},\bfs\lambda}^{sq}|
    -\mathcal{T}(\bfs\lambda)\,q^{r-1}\big|& \leq \mathcal{T}(\bfs\lambda)q^{r-2}
    (r^2q^{\frac{1}{2}}+14r^4),\\
    \big||\mathcal{A}_{\mathcal{N},\bfs\lambda}|
    -\mathcal{T}(\bfs\lambda)\,q^{r-1}\big|&\le
    q^{r-2}\big(\mathcal{T}(\bfs\lambda)(r^2q^{\frac{1}{2}}+14 r^4)+r^3\big),
    \end{align*}
    where $\mathcal{A}_{\mathcal{N},\bfs\lambda}$ is the set of elements
    of $\mathcal{A}_\mathcal{N}$
    with factorization pattern $\bfs\lambda$ and
    $\mathcal{A}_{\mathcal{N},\bfs\lambda}^{sq}$ is
    the set of square--free elements of $\mathcal{A}_{\mathcal{N},\bfs\lambda}$.
\end{theorem}
\begin{proof}
This is a simple consequence of Theorem \ref{theorem: estimate fact
patterns} with $m:=1$ and the polynomial
$$R_1:=G(-\Pi_1,\Pi_2,\ldots,(-1)^r\Pi_r).$$
As previously remarked, the weighted degree of $G$ is $r$, which
implies that $\deg R_1=r$. Therefore, we have
$$\delta:=\deg R_1=r\textrm{ and }
D:=\deg R_1-1=r-1.$$
As a consequence, Theorem \ref{theorem: estimate fact patterns}
implies
    \begin{align*}
    \big||\mathcal{A}_{\mathcal{N},\bfs\lambda}^{sq}|
    -\mathcal{T}(\bfs\lambda)\,q^{r-1}\big|& \leq \mathcal{T}(\bfs\lambda)q^{r-2}
    \big((r(r-3)+2)q^{\frac{1}{2}}+14 (r-1)^2 r^2+ r^3\big),\\
    \big||\mathcal{A}_{\mathcal{N},\bfs\lambda}|
    -\mathcal{T}(\bfs\lambda)\,q^{r-1}\big|&\le
    q^{r-2}\Big(\mathcal{T}(\bfs\lambda)
    \big((r(r-3)+2)q^{\frac{1}{2}}+14 (r-1)^2 r^2
    +r^3\big)+r^3\Big).
    \end{align*}
This immediately implies the statement of the theorem.
\end{proof}

%
\section{Average--case analysis of polynomial factorization over $\mathcal{A}$}
\label{sec: average-case complexity of factorization}
In this section we analyze the average--case complexity of the
classical factorization algorithm applied to any family
$\mathcal{A}$ as in \eqref{non-linear family A} satisfying
hypotheses $({\sf H}_1)$--$({\sf H}_6)$.

Given $f \in \fq[T]$, the classical factorization algorithm finds
the complete factorization $f=f_1^{e_1} \dots f_n^{e_n}$, where
$f_1, \dots, f_n$ are pairwise distinct monic irreducible
polynomials in $\fq[T]$ and $e_1, \dots,e_n$ are strictly positive
integers. The algorithm contains three main routines:
\begin{itemize}
    \item \textbf{elimination of repeated factors (ERF)}
    replaces a polynomial by a square--free one that contains
    all the irreducible factors of the original one with exponent
    $1$;
    \item \textbf{distinct--degree factorization (DDF)} splits
    a square--free polynomial into a product of polynomials
    whose irreducible factors have all the same degree;
    \item \textbf{equal--degree factorization (EDF)} splits
    completely a polynomial whose irreducible factors have all the same degree.
\end{itemize}
More precisely, the algorithm works as follows:
\begin{algorithm2}${}$

    \smallskip
    \begin{enumerate}
        \item [] Input: a monic polynomial $f \in \fq[T]$ of degree $r>0$.

        \item [] Output: the complete factorization of $f$ in $\fq[T]$.
        \begin{center}
            \textbf{factor procedure ($f\in \fq[T]$)}
        \end{center}

        \begin{enumerate}
            \item [] $a_f:=\mathrm{ERF}(f)$\quad  [$a_f$  is square--free]
            \item[] $\bfs b_f:=\mathrm{DDF}(a_f)$\quad [$\bfs b_f$ is a partial factorization into distinct degrees]

            \item[]  $F:=1$
            \item[] For $k$ from $1$ to $s$ ($s \leq r$) do
            \begin{enumerate}
                \item []$F:=F\cdot \mathrm{EDF}(b_f[k],k)$\quad [refines the distinct--degree factorization for \linebreak
                ${}$ \hskip 4.4cm polynomials of degree $k$]
            \end{enumerate}
            \item[] end do

            \item[] $c:=\mathrm{factor}(f/a_f)$

            \item[] Return $F\cdot c$.
        \end{enumerate}
    \end{enumerate}

\end{algorithm2}

In \cite{FlGoPa01}, the authors analyze the average--case complexity
of the classical factorization algorithm applied to all the monic
polynomials of degree $r$ of $ \fq [T] $. Unfortunately, the results
of this analysis cannot be directly applied to the family
$\mathcal{A}$, because there is a small probability that a random
monic polynomial of degree $r$ of $ \fq [T] $ belongs to
$\mathcal{A}$. As a consequence, we shall perform an analysis of the
behavior of this algorithm applied to elements of $\mathcal{A}$,
using the results on the distribution of factorization patterns of
Section \ref{the factorization patterns}.

Considering the uniform probability on $\mathcal{A}$, let
$\mathcal{X}: \mathcal{A} \rightarrow \mathbb{N}$ be the random
variable that counts the number $ \mathcal {X} (f) $ of arithmetic
operations in $ \fq $ performed by the classical factorization
algorithm to obtain the complete factorization  in $\fq[T]$ of any
$f\in\mathcal{A}$. We may describe this algorithm as consisting of
four stages, and thus the random variable $\mathcal{X}$ may be
decomposed as the sum of the random variables that count the cost of
each step of the algorithm. More precisely, we consider the random
variable $\mathcal{X}_1: \mathcal{A} \rightarrow \mathbb{N}$ that
counts the number of arithmetic operations in $\fq$ performed in the
ERF step, namely
\begin{equation}
\label{ERF} \mathcal{X}_1(f):=\mathrm{Cost}(\mathrm{ERF}(f)).
\end{equation}
Further, we introduce a random variable $\mathcal{X}_2: \mathcal{A}
\rightarrow \mathbb{N}$ that counts the number of arithmetic
operations in $\fq$ performed during the DDF step, namely
\begin{equation}
\label{DDF} \mathcal{X}_2(f):=\mathrm{Cost}(\mathrm{DDF}(a_f)),
\end{equation}
where $a_f:=\mathrm{ERF}(f)$ denotes the square--free polynomial
obtained after the ERF step on input $f$. Denote by
$$\bfs b_f:=\mathrm{DDF}(a_f)=(b_f(1),\dots, b_f(s))$$
the vector of polynomials obtained by applying the DDF step to the
monic square--free polynomial $a_f:=\mathrm{ERF}(f)$, where $s$ is
the degree of the largest irreducible factor of $a_f$. Each $b_f(k)$
consists of the product of all the monic irreducible polynomials in
$ \fq [T] $ of degree $k$ that divide $f$. With this notation, let
$\mathcal{X}_3: \mathcal{A} \rightarrow \mathbb{N}$ be the random
variable that counts the number of arithmetic operations in $\fq$ of
the EDF step, namely
\begin{equation}
\label{EDF} \mathcal{X}_3(f):=\sum_{k=1}^s
\mathcal{X}_{3,k}(f),\quad
\mathcal{X}_{3,k}(f)\!\!:=\!\!\mathrm{Cost}(\mathrm{EDF}(b_f(k)))\quad(1\le
k\le s).
\end{equation}
%
%
Finally, we introduce a random variable $\mathcal{X}_4:\mathcal{A}
\rightarrow \mathbb{N}$ that counts the number of operations in
$\fq$ performed by the classical factorization algorithm applied to
${f}/\mathrm{ERF}(f)$.
Our aim is to study the expected value of the random variable
$\mathcal{X}$, namely
\begin{equation}\label{esperanza de la variable aleatoria costo}
E[\mathcal{X}]:=\frac{1}{|\mathcal{A}|}\sum_{f \in
\mathcal{A}}\mathcal{X}(f)= \frac{1}{|\mathcal{A}|} \sum_{k=1}^4
\sum_{f \in \mathcal{A}}\mathcal{X}_k(f).
\end{equation}

We denote by $M(r)$ a \textit{multiplication time}, so that the
product of two polynomials in $\fq[T]$ of degree at most $r$ can be
computed with at most $\tau_1 M(r)$ arithmetic operations in $\fq$.
Using fast arithmetic  we can take $M(r):= r\log r\log\log r$ (see,
e.g., \cite{GaGe99}). For $\tau_1$ suitably chosen, a division with
remainder of two polynomials of degree at most $r$ can also be
computed with at most $\tau_1 M(r)$ arithmetic operations in $\fq$.
Further, the cost of computing the greatest common divisor of two
polynomials in $\fq[T]$ of degree at most $r$ is at most $\tau_2\,
\mathcal{U}(r)$ arithmetic operations in $\fq$, where
$\mathcal{U}(r):= M(r) \log r$ (see, e.g., \cite{GaGe99}). Here,
$\tau_1$ and $\tau_2$ are system and implementation dependent
constants.
%
%
\subsection{Elimination of repeated factors}
We consider in detail the step of elimination of repeated factors
(ERF). Let
$$f=f_1^{e_1}\dots f_n^{e_n}=\prod_{p\mid e_i}f_i^{e_i}\prod_{p
\nmid e_i}f_i^{e_i}$$
be the factorization  of $f \in \mathcal{A}$ into monic irreducible
polynomials in $\fq[T]$, where $f_1, \dots,f_n$ are pairwise
distinct, $e_1, \dots,e_n \in \mathbb{N}$ and
$p:=\mathrm{char}(\fq)$. It is clear that $f$ is square--free if and
only if $\gcd (f,f')=1$ (see, e.g., \cite[Corollary 14.25]{GaGe99}).
Assume that $f$ is not square--free. Hence, $u:=\gcd (f,f') \neq 1$.
It follows that $v:=f/u=\prod_{p \nmid e_i} f_i$ is the square--free
part of the product  $\prod_{p \nmid e_i}f_i^{e_i}$ (see, e.g.,
\cite[Theorem 20.4]{Shoup05}).  Since each $e_i \leq r:=\deg(f)$, we
deduce that $\gcd(u,v^r)=\prod_{ p \nmid e_i} f_i^{e_i-1}$.
Therefore,
$$w:=\frac{u}{\gcd(u,v^r)}=\prod_{p \mid e_i} f_i^{e_i}$$
is the part of $f$ which is a power of $ p $. These are the
foundations of the following procedure.
\begin{algorithm3} ${}$
    \begin{enumerate}
        \item []Input: $f \in \fq[T]$ monic of degree $r>0$.

        \item []Output: the square--free part of  $f$, that is, the product of all distinct irreducible factors of  $f$ in $\fq[T]$.

        \begin{center}
            \textbf{procedure ERF (f: polynomial)}
        \end{center}

        \begin{enumerate}
            \item [] Compute $u:=\gcd(f,f')$
            \item [] Compute $v:=\frac{f}{u}$\quad [square--free part of $\prod_{p \nmid e_i}f_i^{e_i}$]
            \item [] Compute $w:=\frac{u}{\gcd(u,v^r)}$\quad [part of $f$ which is a power of $p$]
            \item[] Return $v \cdot \mathrm{ERF}(w^{1/p})$.
        \end{enumerate}
    \end{enumerate}
 \end{algorithm3}
According to \cite[Exercise 14.27]{GaGe99}, for $f\in \fq[T]$ of
degree at most $r$, the number of arithmetic operations in $\fq$
performed by the ERF algorithm to obtain the square--free part of
$f$ is  $\mathcal{O}(M(r)\log r+ r \log (q/p))$. In this section we
analyze the average--case complexity of the ERF algorithm restricted
to elements of the family $\mathcal{A}$. More precisely, we analyze
the expected value $E[\mathcal{X}_1]$ of the random variable
$\mathcal{X}_1$ defined in \eqref{ERF}, namely
\begin{equation}
\label{esperanza ERF}
E[\mathcal{X}_1]:=\frac{1}{|\mathcal{A}|} \sum _{f \in \mathcal{A}} \mathcal{X}_1(f).
\end{equation}

Let $\mathcal{A}^{sq}$ be the set of $f \in \mathcal {A} $ that are
square--free and $\mathcal{A}^{nsq}:=\mathcal{A}\setminus
\mathcal{A}^{sq}$. The probability that a random polynomial  of
$\mathcal{A}$ is square--free is
$$P[\mathcal{A}^{sq}]=\frac{|\mathcal{A}^{sq}|}{|\mathcal{A}|}=1-\frac{|\mathcal{A}^{nsq}|}{|\mathcal{A}|}.
$$
According to \eqref{eq: estimates: upper bound discr locus}, we have
$|\mathcal{A}^{nsq}|\leq r(r-1)\delta_{\bfs G} q^{r-m-1}$. On the
other hand, from Theorem \ref{estimation A} it follows that, if  $q
> 15\delta_{\bfs G}^{13/3}$, then $|\mathcal{A}|\geq\frac{1}{2}
q^{r-m}$,
where $ \bfs G:= (G_1, \dots, G_m) $ are the polynomials defining
the family $ \mathcal{A} $ and $\delta_{\bfs G}:=\deg(G_1)\cdots\deg
(G_m)$. As a consequence,
$$
P[\mathcal{A}^{sq}] \geq 1-\frac{2\,r^2\delta_{\bfs G}
\,q^{r-m-1}}{q^{r-m}} =1-\frac{2\,r^2\delta_{\bfs G}}{q}.
$$
In other words, we have the following result.
\begin{lemma}\label{mas libres de cuadrados}
For $q > 15\delta_{\bfs G}^{13/3}$, the probability that a random
polynomial of  $\mathcal{A}$ is square--free is
$P[\mathcal{A}^{sq}]\ge 1 - 2\,r^2\delta_{\bfs G}/q$. In particular,
if $q >\max\{ 15\delta_{\bfs G}^{13/3},4\,r^2\delta_{\bfs G}\}$,
then $P[\mathcal{A}^{sq}]>1/2$.
\end{lemma}

To estimate $E[\mathcal{X}_1]$, we decompose the family
$\mathcal{A}$ into the sets $\mathcal{A}^{sq}$ and
$\mathcal{A}^{nsq}$. We have
$$
E[\mathcal{X}_1]=\frac{1}{|\mathcal{A}|} \sum _{f \in
\mathcal{A}^{sq}} \mathcal{X}_1(f) +\frac{1}{|\mathcal{A}|} \sum _{f
\in \mathcal{A}^{nsq}} \mathcal{X}_1(f)=:S_1^{sq}+S_1^{nsq}.
$$
First we obtain an upper bound for $S_1^{sq}$. On input
$f\in\mathcal{A}^{sq}$, the ERF algorithm performs the first three
steps. Since $u:=\gcd(f,f')=1$ and $\gcd(u,v^r)=1$, its cost is
dominated by the cost of calculating $u$, which is at most
$\tau_2\,\mathcal{U}(r)$ arithmetic operations in $\fq$, and the
cost of calculating $v^r$, which at most $ \tau_1 \, \mathcal{U}(r)$
arithmetic operations in $\fq$. We conclude that, if $f \in
\mathcal{A}^{sq}$, then $\mathcal{X}_1(f) \leq (\tau_1 +\tau_2)\,
\mathcal{U}(r)$. Therefore,
\begin{equation}\label{cota Suma 1 de X1}
S_1^{sq}:=\frac{1}{|\mathcal{A}|} \sum _{f \in \mathcal{A}^{sq}}
\mathcal{X}_1(f)\leq  (\tau_1 +\tau_2)\, \mathcal{U}(r)
\frac{|\mathcal{A}^{sq}|}{|\mathcal{A}|}.
\end{equation}
On the other hand, if  $f \in \mathcal{A}^{nsq}$, then
\cite[Exercise 14.27]{GaGe99} shows that the number of arithmetic
operations in $ \fq $ which performs the ERF algorithm on input $ f
$ is bounded by $ \mathcal{X}_1(f) \leq c_1 \big( \mathcal{U}(r)+ r
\log \big(\frac{q}{p}\big)\big)$, where $c_1$ is a constant
independent of $q$ and $p:=\mathrm{Char}(\fq)$. Hence, we have
\begin{equation}\label{cota suma S2 de X1}
S_1^{nsq}:=\frac{1}{|\mathcal{A}|} \sum _{f \in \mathcal{A}^{nsq}}
\mathcal{X}_1(f)\leq  c_1 \bigg(\mathcal{U}(r)+ r
\log\Big(\frac{q}{p}\Big)\bigg)
\frac{|\mathcal{A}^{nsq}|}{|\mathcal{A}|}.
\end{equation}
Combining  \eqref{cota Suma 1 de X1} and \eqref{cota suma S2 de X1}
we conclude that
\begin{align*}
E[\mathcal{X}_1] & \leq (\tau_1 +\tau_2) \, \mathcal{U}(r)
\frac{|\mathcal{A}^{sq}|}{|\mathcal{A}|} + c_1  \, \mathcal{U}(r)
\frac{|\mathcal{A}^{nsq}|}{|\mathcal{A}|}+ c_1 \,r
\log\Big(\frac{q}{p}\Big)\frac{|\mathcal{A}^{nsq}|}{|\mathcal{A}|}\\
\notag & \leq c_2  \, \mathcal{U}(r) + c_1\, r
\log\Big(\frac{q}{p}\Big) \frac{|\mathcal{A}^{nsq}|}{|\mathcal{A}|},
\end{align*}
where $c_2:=\max\{ \tau_1 +\tau_2 ,c_1\}$. Hence, if
$q>15\delta_{\bfs G}^{13/3}$, then Lemma \ref{mas libres de
cuadrados} implies
$$E[\mathcal{X}_1] \leq c_2\, \mathcal{U}(r) + 2\,c_1\, r^3
\delta_{\bfs G}\log\Big(\frac{q}{p}\Big)\frac{1}{q}.$$
We obtain the following result.
\begin{theorem}\label{costo paso 1}
Let $q > 15\delta_{\bfs G}^{13/3}$. The average cost $ E [\mathcal
{X} _1] $ of the $\mathrm{ERF}$ algorithm applied to elements of $
\mathcal {A} $ is upper bounded as $ E [\mathcal {X} _1] \leq c_2 \,
\mathcal {U} (r) +c_3 \log  \big (\frac {q} {p} \big) \delta_{\bfs
G}\frac {r ^ 3} {q} $, where $c_2$ and $ c_3$ are constants
independent of $ r$ and $ q $.
\end{theorem}

We may paraphrase this result as saying that the average cost of the
ERF algorithm applied to elements of $\mathcal{A}$ is asymptotically
of order $\mathcal{U}(r)$, which corresponds to the cost of
calculating the greatest common divisor $u:=\gcd (f,f')$. This
generalizes the results of \cite[Section 2]{FlGoPa01}.
%
%
\subsection{Distinct--degree factorization}
Now we analyze the distinct--degree factorization (DDF) step. Recall
that, given a square--free polynomial $a_f:=\mathrm{ERF}(f)$, the
DDF routine outputs a list $(b(1),\ldots,b(s))$, where $b(k)$ is the
product of all the irreducible factors of degree $k$ of the complete
factorization of $a_f$ over $\fq$. The output $(b(1),\dots,b(s))$ is
called the {\em distinct--degree factorization} of $a_f$.

The DDF procedure is based on the following property (see, e.g.,
\cite[Theorem 3.20]{LiNi83}): for $k \geq 1$, the polynomial
$T^{q^k}-T \in\fq[T]$ is the product of all monic irreducible
polynomials in $\fq[T]$ whose degree divides $k$. It follows that
$g_1:=\gcd(T^q-T,f)$ is the product of all the irreducible factors
of $f$ of degree $ 1 $. Then, for $ 1 \leq k \leq r $, the
polynomial $g_k:=\gcd(T^{q^k }-T, f/g_{k-1}) $ is the product of all
the irreducible factors of $f$ of degree $k$. This proves the
correctness of the following procedure.

\begin{algorithm4}\label{algoritmo DDF}${}$
\begin{enumerate}
        \item[] Input: a monic square--free polynomial $a\in\fq[T]$ of degree $r>0$.

        \item[] Output: the distinct--degree factorization $(b(1),\dots,b(s))$ of $a$ in $\fq[T]$.

        \item [] Let  $g:=a$,\, \,  $h:=T$

        \item [] While $g\neq 1$ do
        \begin{enumerate}
            \item [] Compute $h:=h^q \mod g$
            \item[] Compute $b(k):=\gcd(h-T,g)$
            \item[] Compute $g:=\frac{g}{b(k)}$\quad [$a$ without the irreducible factors of degree at most $k$]

            \item[] $k:=k+1$
        \end{enumerate}

        \item[] End while

        \item[] Return $\bfs b$.
    \end{enumerate}
\end{algorithm4}

In \cite[Theorem 14.4]{GaGe99} it is shown that this algorithm
performs $\mathcal{O}(s M(r) \log(rq))$ arithmetic operations in
$\fq$, where $s$ is the maximum degree of the irreducible factors of
the input polynomial $a$. In this section we analyze the
average--case complexity of the DDF routine restricted to
polynomials of the family $\mathcal{A}$. More precisely, we consider
the expected value $E[\mathcal{X}_2]$ of the random variable
$\mathcal{X}_2$ of \eqref{DDF}, namely
$$
E[\mathcal{X}_2]:=\frac{1}{|\mathcal{A}|} \sum _{f \in \mathcal{A}} \mathcal{X}_2(f).
$$
We decompose as before the set of inputs $\mathcal{A}$ into the
disjoint subsets $\mathcal{A}^{sq}$ (elements of $\mathcal{A}$ which
are square--free) and  $\mathcal{A} ^{nsq}:=\mathcal{A} \setminus
\mathcal{A}^{sq}$. Hence, we have
\begin{equation}\label{esperanza DDF 1}
E[\mathcal{X}_2]=\frac{1}{|\mathcal{A}|} \sum _{f \in
\mathcal{A}^{sq}} \mathcal{X}_2(f)+\frac{1}{|\mathcal{A}|} \sum _{f
\in \mathcal{A}^{nsq}} \mathcal{X}_2(f).
\end{equation}

First we obtain an upper bound for the first sum $S_2^{sq}$ in the
right--hand side of \eqref{esperanza DDF 1}. We express
$\mathcal{A}^{sq}$ as a disjoint union as follows:
$$
\mathcal{A}^{sq}=\bigcup_{i=1}^r \mathcal{A}_{i}^{sq},
$$
where $\mathcal{A}_{i}^{sq}$ is the set of elements of
$\mathcal{A}^{sq}$ for which the maximum degree of the irreducible
factors is $i$. Moreover, for $1 \leq i \leq r$, we can express each
$\mathcal{A}_{i}^{sq}$ as the disjoint union
$$
\mathcal{A}_{i}^{sq}=\bigcup_{\bfs \lambda \in \mathcal{P}_i}
\mathcal{A}_{\bfs \lambda}^{sq},
$$
where $\mathcal{P}_i$ is the set of $\bfs \lambda:=(\lambda_1,
\dots, \lambda_i, 0,\dots,0)\in \mathbb{Z}_{\ge 0}^r$ such that
$\lambda_1+\cdots+ i \, \lambda_i=r$ and $ \lambda_i>0$, and
$\mathcal{A}_{ \bfs \lambda}^{sq}$ is the set of elements of
$\mathcal{A}_{i}^{sq}$ with factorization pattern $\bfs \lambda$.
Therefore,
\begin{equation} \label{suma 11 bis}
S_2^{sq}=\frac{1}{|\mathcal{A}|} \sum_{i=1}^r \sum_{\bfs \lambda\in
\mathcal{P}_i}\sum _{f \in \mathcal{A}_{\bfs \lambda }^{sq}}
\mathcal{X}_2(f).
\end{equation}

Fix $i$ with $ 1 \leq i \leq r$, let $\bfs \lambda\in \mathcal{P}_i$
and  $f \in \mathcal{A}_{\bfs \lambda}^{sq}$. To determine the cost
$\mathcal{X}_2(f)$, we observe that the procedure performs $i$
iterations of the main loop. Fix $l$ with $ 1\leq l\leq i$ and we
consider the $l$th iteration of the DDF algorithm. The number of
products modulo $g$ needed to compute $h^q \mod g$ is denoted by
$\lambda(q)$. Using repeated squaring, and denoting by $\nu(q)$ the
number of ones in the binary representation of $q$, the number of
products required to compute $h^q \mod g$ is
$$ \lambda (q):= \lfloor \log q \rfloor + \nu (q) -1.$$
Thus the first step in the $ l $th iteration  of the DDF algorithm
requires at most $2\,\tau_1\, \lambda (q) M(r_l)$ arithmetic
operations in $\fq$, where $r_l:=\deg g$ (note that $r_1=r$ and $r_l
\leq r$ for any $l$). Then the computation $b(k):=\gcd(h-T,g)$
requires at most $\tau_2 M(r_l) \log r_l$ arithmetic operations in
$\fq$. Finally, the division $g/ b(k) $ requires at most $\tau_1
M(r_l)$ arithmetic operations in $\fq$. As a consequence, we see
that
$$
\mathcal{X}_2(f) \leq \sum_{l=1}^i (2\,\tau_1 \lambda (q) + \tau_2
\log r_l + \tau_1) \,  M(r_l).
$$
Observe that, if $a \leq b$, then $M(a) \leq M(b)$ (see, e.g.,
\cite[\S 14.8]{GaGe99})). It follows that
\begin{equation}\label{Costo X2}
\mathcal{X}_2(f) \leq i\,c_{r,q},\quad c_{r,q}:=
M(r)\,\big(2\,\tau_1\lambda(q)+ \tau_1+\tau_2 \log r\big).
\end{equation}
Thus, we obtain
$$S_2^{sq} \leq \frac{c_{r,q}}{|\mathcal{A}|} \sum_{i=1}^r \sum_{\bfs
\lambda \in \mathcal{P}_i}\sum _{f \in \mathcal{A}_{\bfs \lambda
}^{sq}} i= \frac{c_{r,q}}{|\mathcal{A}|}\sum_{i=1}^r i \sum_{\bfs
\lambda \in \mathcal{P}_i} |\mathcal{A}_{\bfs \lambda}^{sq}|.$$
We have the following result.
\begin{lemma}\label{S2sq cota}
For $q > 15\delta_{\bfs G}^{13/3}$, the sum $S_2^{sq}$ is bounded in
the following way:
\begin{equation}\label{suma 11bisbisbisbis}
S_2^{sq} \leq c_{r,q}\bigg(1+\frac{15\delta_{\bfs
G}^{{13}/{6}}}{q^{{1}/{2}}}\bigg)\bigg(1+\frac{M_{r}}{q} \bigg)\xi
(r+1)= c_{r,q}\,\xi (r+1)\big(1+o(1)\big),
\end{equation}
where $M_{r}:=D\delta q^{\frac{1}{2}} +14\,D^2 \delta^2+ r^2\delta$,
$\delta:=\prod_{i=1}^m \wt(G_i)$, $D:=\sum_{i=1}^m (\wt(G_i)-1)$ and
$\xi\sim 0.62432945\dots$ is the Golomb constant.
 \end{lemma}
\begin{proof}
According to Theorem \ref{theorem: estimate fact patterns}, we have
$$|\mathcal{A}_{\bfs \lambda }^{sq}| \leq q^{r-m}\,\mathcal{T}(\bfs
\lambda)\bigg(1 +\frac{M_{ r}}{q}\bigg),$$
where $\mathcal{T}(\bfs \lambda)$ is the probability of the set of
permutations with cycle pattern $\bfs \lambda $ in the symmetric
group $\mathbb{S}_r$ of $r$ elements. Hence,
\begin{align}\label{suma 11bisbisbis}
S_2^{sq} & \leq
\frac{c_{r,q}}{|\mathcal{A}|}\,q^{r-m}\bigg(1+\frac{M_{r}}{q}
\bigg)\sum_{i=1}^r i \sum_{\bfs \lambda \in \mathcal{P}_i}
\mathcal{T}(\bfs \lambda ).
\end{align}

Now we analyze the sum $E_r:=\sum_{i=1}^r i\sum_{\bfs \lambda \in
\mathcal{P}_i} \mathcal{T}(\bfs \lambda )$. Observe that the sum $
\sum_ {\bfs \lambda \in \mathcal {P}_i} \mathcal{T} (\bfs \lambda) $
expresses the probability of the set of permutations whose longest
cycle has length $ i $. It follows that $ E_r $ is the largest
expected length between cycles of a random permutation in $ \mathbb
{S} _r $. In \cite{GoGa98} it is shown that
$$
\frac{E_r}{r+1}\leq \xi,
$$
where $\xi$ is the Golomb constant (see, e.g., \cite{Knuth98}).
Combining this upper bound, Theorem \ref{estimation A} and
\eqref{suma 11bisbisbis}, we readily deduce the statement of the
lemma.
\end{proof}

Next we obtain an upper bound for the second sum $S_2^{nsq}$ of
\eqref{esperanza DDF 1}, namely
$$
S_2^{nsq}:=\frac{1}{|\mathcal{A}|} \sum _{f \in \mathcal{A}^{nsq}} \mathcal{X}_2(f).
$$
Given $f \in \mathcal{A}^{nsq}$, we bound
$\mathcal{X}_2(f):=\mathrm{Cost}(\mathrm{DDF}(a_f))$, where
$a_f:=\mathrm{ERF}(f)$ is the  output square--free polynomial of the
ERF procedure applied to $f$. By \eqref{Costo X2} we have
$$
\mathcal{X}_2(f) \leq c_{N,q}\cdot s_a,
$$
where $c_{N,q}:=M(N)\,\big(2\,\tau_1\lambda(q)+ \tau_1+\tau_2 \log
N\big)$, $N:=\deg(a_f)$ and $s_a$ is the highest degree of the
irreducible factors of $ a_f $. Since $f \in \mathcal{A}^{nsq}$, we
have $N\leq r-1$ and $s_a\le r -2$.  Moreover, it is easy to see
that these bounds are optimal. Therefore we obtain
$$
\mathcal{X}_2(f) \leq c_{r-1,q}\, (r-2).
$$
Combining this bound, Theorem \ref{estimation A} and \eqref{eq:
estimates: upper bound discr locus} we see that, if $q>
15\delta_{\bfs G}^{13/3}$, then
\begin{align}\nonumber
S_{2}^{nsq} \leq c_{r-1,q}\,(r-2) \frac{|A^{nsq}|}{|\mathcal{A}|}
&\leq c_{r-1,q}\,(r-2)\bigg(1+\frac{15\delta_{\bfs
G}^{{13}/{6}}}{q^{{1}/{2}}}\bigg) \frac{r^2\delta_{\bfs G}\,
q^{r-m-1}}{q^{r-m}}\\&\leq \,c_{r-1,q}\bigg(1+\frac{15\delta_{\bfs
G}^{{13}/{6}}}{q^{{1}/{2}}}\bigg)\,\frac{r^3\delta_{\bfs
G}}{q}.\label{suma 2 bis}
\end{align}

From the upper bounds of Lemma \ref{S2sq cota} and \eqref{suma 2
bis} we conclude that
\begin{align*}
E[\mathcal{X}_2]&=\frac{1}{|\mathcal{A}|} \sum _{f \in
\mathcal{A}^{sq}} \mathcal{X}_2(f)+\frac{1}{|\mathcal{A}|} \sum _{f
\in \mathcal{A}^{nsq}} \mathcal{X}_2(f)\\
&\le c_{r,q}\bigg(1+\frac{15\delta_{\bfs
G}^{{13}/{6}}}{q^{{1}/{2}}}\bigg)\bigg(1+\frac{M_{r}}{q} \bigg)\xi
(r+1)+\,c_{r-1,q}\bigg(1+\frac{15\delta_{\bfs
G}^{{13}/{6}}}{q^{{1}/{2}}}\bigg)\,\frac{r^3\delta_{\bfs G}}{q}.
\end{align*}
%
Since $c_{j,q}:=M(j)\,\big(2\, \tau_1 \lambda(q)+\tau_1+\tau_2\log
j\big)$, we have $c_{r-1,q} \leq c_{r,q}$. As a consequence, we
obtain the following result.
\begin{theorem}\label{average de DDF}
For $q> 15\delta_{\bfs G}^{13/3}$, the average cost
$E[\mathcal{X}_2]$ of the $\mathrm{DDF}$ algorithm restricted to
$\mathcal {A} $ is bounded by
\begin{align*}
E[\mathcal{X}_2] &\leq \xi\,(2\, \tau_1 \lambda(q)+\tau_1+\tau_2\log
r) M(r)\,(r+1) \bigg(1 +\frac{M_r+r^2\delta_{\bfs
G}}{q}\bigg)\bigg(1+\frac{15\delta_{\bfs
G}^{{13}/{6}}}{q^{{1}/{2}}}\bigg)\\
&=\xi\,(2\, \tau_1 \lambda(q)+\tau_1+\tau_2\log r)
M(r)\,(r+1)\big(1+o(1)\big),
\end{align*}
where $M_{r}:=D\delta q^{\frac{1}{2}} +14\,D^2 \delta^2+ r^2\delta$,
$\delta:=\prod_{i=1}^m \wt(G_i)$, $D:=\sum_{i=1}^m (\wt(G_i)-1)$ and
$\xi\sim 0.62432945\dots$ is the Golomb constant.
\end{theorem}

In \cite[Theorem 5]{FlGoPa01} the authors prove that the average
cost of the DDF algorithm applied to a random polynomial $ f \in \fq
[T] $ of degree at most $ r $ is of order $0.26689\,
(2\,\tau_1\,\lambda(q) +\tau_2)\,r^3$. We prove that, assuming that
fast arithmetic is used, the average cost of this algorithm
restricted to $\mathcal{A}$ is of order
$\xi(2\,\tau_1\,\lambda(q)+\tau_1 +\tau_2\log r)\,(r+1)\,M(r)$
arithmetic operations in $\fq$, thus improving the result of
\cite{FlGoPa01} (which assumes that standard arithmetic is used).

The DDF algorithm does not completely factor any polynomial $f \in
\mathcal{A}$ having distinct irreducible factors of the same degree.
More precisely, the classical factorization algorithm ends in this
step if the input polynomial $f$ has a factorization pattern $\bfs
\lambda\in \{0,1\}^r$. We conclude this section with a result on the
probability that the DDF algorithm outputs the complete
factorization of the input polynomial of $\mathcal {A}$.

In \cite{FlSe09} it is shown that most factorizations are completed
after the application of the DDF procedure. More precisely, it is
proved that, when $ r $ is fixed and $ q $ tends to infinity, the
probability that the DDF algorithm produces a complete factorization
of a random polynomial of degree at most $ r $ in $ \fq [T] $ is of
order of $e^{-\gamma} \sim 0.5614\dots$, where $\gamma \sim
0.57721\dots$ is the Euler constant (see \cite[Theorem
6]{FlGoPa01}). We generalize this result to the family
$\mathcal{A}$.
\begin {theorem}
The probability that the $\mathrm{DDF}$ algorithm completes the
factorization of a random polynomial of $\mathcal {A}$ is bounded
from above by $\big(e^{-\gamma}+ e^{-{\gamma}}/{r}+O({\log
r}/{r^2})\big)\big(1+o(1)\big)$, where $ \gamma $ is Euler's
constant.
\end {theorem}
\begin{proof}
Let $\mathcal{A}_1$ be set of elements of $ \mathcal {A} $ whose
irreducible factors have all distinct degrees. The probability that
the DDF algorithm outputs the complete factorization of a random
polynomial $f\in \mathcal {A} $ coincides with the probability that
random $f\in \mathcal {A}$ belongs to $\mathcal {A}_1$. We may
express $ \mathcal {A}_1 $ as the following disjoint union:
$$
\mathcal{A}_1=\bigcup_{\bfs \lambda \in \mathcal{P}_{r}} \mathcal{A}_{1,\bfs \lambda},
$$
where $ \mathcal {P}_r $ is the set of all vectors $ \bfs \lambda: =
(\lambda_1, \dots, \lambda_r) \in \{0,1 \}^r $ such that $ \lambda_1
+ \dots + r \, \lambda_r = r $ and $ \mathcal{A}_{1, \bfs \lambda} $
is the set of elements of $ \mathcal {A}_1 $ having factorization
pattern $ \bfs \lambda $. Hence,
\begin{equation}\label{completar factorizacion DDF}
P[\mathcal{A}_1]=\sum_{\bfs \lambda \in \mathcal{P}_r} P
[\mathcal{A}_{1,\bfs \lambda}]=\frac{1}{|\mathcal{A}|} \sum_{\bfs
\lambda \in \mathcal{P}_r} |\mathcal{A}_{1,\bfs \lambda}|.
\end{equation}
Observe that, if $f \in \mathcal{A}_1$, then $f$ is square--free. By
Theorem \ref{theorem: estimate fact patterns}, for $m<r$ we have
$$
|\mathcal{A}_{1,\bfs \lambda}| \leq q^{r-m}\,\mathcal{T}(\bfs
\lambda)\,\bigg(1+\frac{M_{r}}{q}\bigg),
$$
where $M_{r}:=D\delta q^{\frac{1}{2}} +14\,D^2 \delta^2+ r^2\delta$,
$\delta:=\prod_{i=1}^m \wt(G_i)$ and $D:=\sum_{i=1}^m (\wt(G_i)-1)$.
Theorem \ref{estimation A} shows that, if $q > 15\delta_{\bfs
G}^{13/3}$, then
$$
P[ \mathcal{A}_1] \leq \bigg(1+\frac{15\delta_{\bfs
G}^{{13}/{6}}}{q^{{1}/{2}}}\bigg)\bigg(1+\frac{M_{r}}{q}\bigg)\sum_{\bfs
\lambda \in \mathcal{P}_r}\mathcal{T}(\bfs \lambda).
$$
We observe that $\sum_{\bfs \lambda \in \mathcal{P}_r}
\mathcal{T}(\bfs \lambda)$  expresses the probability that a random
permutation of $\mathbb{S}_r$ has a decomposition into cycles of
pairwise different lengths. By \cite[(4.57)]{GrKn90} (see also
\cite[Proposition 1]{FlFuGoPaPo06}), it follows that
$$
\sum _{\bfs \lambda \in \mathcal{P}_r}\mathcal{T}( \bfs \lambda)=
e^{-\gamma}+\frac{e^{-\gamma}}{r}+O\bigg(\frac{\log_2r}{r^2}\bigg).
$$
We deduce that
$$
P[\mathcal{A}_1]\leq \bigg(1+\frac{15\delta_{\bfs
G}^{{13}/{6}}}{q^{{1}/{2}}}\bigg)
\bigg(1+\frac{M_{r}}{q}\bigg)\bigg(e^{-\gamma}+
\frac{e^{-\gamma}}{r}+O\bigg(\frac{\log_2r}{r^2}\bigg)\bigg).
$$
This finishes the proof of theorem.
\end{proof}
%
%
\subsection{Equal--degree factorization}
After the first two steps of the classical factorization algorithm,
the general problem of factorization is reduced to  factorizing a
collection of square--free polynomials $ b (k) $, whose irreducible
factors have all the same degree $ k $. The procedure for
equal--degree factorization (EDF) receives as input a vector $ \bfs
b_f: = \mathrm{DDF} (a_f) = (b_f (1), \dots, b_f (s)) $, where each
$ b_f (k) $ is the product of the irreducible factors of degree $ k
$ of the square--free part $ a_f: = \mathrm{ERF} (f) $ of $ f $. Its
output is the irreducible factorization $b_f (k)=b_f (k, 1) \cdots
b_f ( k, l)$ in $\fq[T]$ of each $b_f (k)$. The probabilistic
algorithm presented here is based on the Cantor--Zassenhaus
algorithm \cite{Zassenhaus69}.
 \begin{algorithm5} ${}$
    \begin{enumerate}
        \item []Input: a monic square--free polynomial $c\in\fq[T]$ whose
        irreducible factors in $\fq[T]$ have all degree $k$.

        \item []    Output: the complete factorization of $c$.

        \begin{center}
            \textbf{procedure EDF(c: square--free polynomial , $k$: integer)}
        \end{center}

        \begin{enumerate}
            \item [] If $\deg c=k$, then return $c$

            \item [] End if

            \item[] Choose  a random $h\in\fq[T]$ of degree $\deg{c}-1$.
            \item[] Compute $g:=h^{(q^k-1)/2}-1 \mod c$
            \item[] Compute $d:=\gcd(g,c)$

            \item [] Return $\mathrm{EDF}(d,k)\cdot \mathrm{EDF}(c/d,k)$.
        \end{enumerate}
    \end{enumerate}
\end{algorithm5}
The EDF algorithm is based on the principle we now briefly explain.
Assume that the irreducible factorization of the input polynomial $
c $ is $c= f_1 \cdots f_j $, with each $f_i$ of degree $ k $. The
Chinese remainder Theorem implies that
$$ \fq [T] /(c) \cong \fq [T] / (f_1) \times \dots \times \fq [T] / (f_j). $$
Under this isomorphism, a random $h\in\fq[T]/(c)$ is associated to a
$j$--tuple $(h_1,\dots,h_j)$, where each $h_i$ is a random element
of $\fq[T]/(f_i)$. Since each $f_i$ is irreducible, the quotient
ring $\fq[T]/(f_i)$ is a finite field, isomorphic to $\fqk$. The
multiplicative group $\fqk^*$ being cyclic, there are the same
number $(q^k-1)/2$ of squares and non--squares (see, e.g.,
\cite[Lemma 14.7]{GaGe99}). Recall that $m \in \fqk^*$ is square if
only if $m^{(q^k-1)/2}=1$. Therefore, testing whether
$h_i^{(q^k-1)/2}=1$ discriminates the squares in $\fqk^*$. Thus, if
$g:=h^{(q^k-1)/2}-1 \mod c$, then $\gcd(g,c)$ is the product of all
the $f_i$ with $h$ a square in $\fq[T]/(f_i)$. From the
probabilistic standpoint, a random element $h_i$ of $\fq[T]/(f_i)$
has probability $\alpha:= 1/2-1/(2q^k)$ of being a square and the
dual probability $\beta:=1/2+1/(2q^k)$ of being a non--square.

Then the EDF algorithm is applied recursively to the polynomials $ d
= \gcd (g, c) $ and $ c / d $. In this way, all the irreducible
factors of $c:=b(k)$ are extracted successively.

Following \cite[Section 5]{FlGoPa01}, in this section we analyze the
average--case complexity of the EDF algorithm applied to the family
$ \mathcal {A} $, namely we consider the expected value
$E[\mathcal{X}_3]$ of the random variable $\mathcal{X}_3$ of
\eqref{EDF}:
$$
E[\mathcal{X}_3]:=\frac{1}{|\mathcal{A}|} \sum _{f \in \mathcal{A}} \mathcal{X}_3(f).
$$
We  decompose $\mathcal {X} _3 $ as in \eqref {EDF} in the form
$$\mathcal{X}_3(f):=\sum_{k=1}^{\lceil r/2
\rceil} \mathcal{X}_{3,k}(f),\quad
\mathcal{X}_{3,k}(f)\!\!:=\!\!\mathrm{Cost}(\mathrm{EDF}(b_f(k)))\quad(1\le
k\le {\lceil r/2 \rceil}),$$
where $ b_f (k) $ is the $k$th coordinate of $\bfs b_f: =
\mathrm{DDF} (a_f) = (b_f (1), \dots, b_f (s)) $. Hence, we have
$$
E[\mathcal{X}_3]=\frac{1}{|\mathcal{A}|}\sum_{k=1}^{\lceil r/2
\rceil}\sum_{f\in
\mathcal{A}}\mathcal{X}_{3,k}(f)=\sum_{k=1}^{\lceil r/2 \rceil}
E[\mathcal{X}_{3,k}].
$$

Fix $k$ with $ 1\leq k \leq \lceil r/2 \rceil$ and write
$E[\mathcal{X}_{3,k}]$ as follows:
$$ E[\mathcal{X}_{3,k}]=\frac{1}{|\mathcal{A}|} \sum_{f \in
\mathcal{A}^{sq}} \mathcal{X}_{3,k}(f)+\frac{1}{|\mathcal{A}|}
\sum_{f \in \mathcal{A}^{nsq}}
\mathcal{X}_{3,k}(f)=:S_{3,k}^{sq}+S_{3,k}^{nsq}.
$$
We first bound $S_{3,k}^{sq}$. For this purpose, we express
$\mathcal{A}^{sq}$ as the disjoint union
$$\mathcal{A}^{sq}=\bigcup_{j=0}^{\lfloor r/k \rfloor} \mathcal{A}_{j,k}^{sq},
$$
where $\mathcal{A}_{j,k}^{sq}$ is the set of all elements $f\in
\mathcal{A}^{sq}$ having $j$ irreducible factors of degree $k$.
Hence, \begin{equation} \label{esperanza X3k libre de cuadrados}
S_{3,k}^{sq}= \frac{1}{|\mathcal{A}|}\sum_{j=0}^{\lfloor r/k \rfloor
} \sum_{f \in \mathcal{A}_{j,k}^{sq}} \mathcal{X}_{3,k}(f).
\end{equation}

We first bound the cost $\mathcal{X}_{3,k}(f)$ of the $\mathrm{EDF}$
algorithm applied to any $f \in \mathcal{A}_{j,k}^{sq}$.
 \begin{lemma}\label{Costo}
For any $f\in \mathcal{A}_{j,k}^{sq}$, we have
$$\mathcal{X}_{3,k}(f) \leq \frac{j(j-1)}{\alpha \beta}\big(\tau_1\,\mu_k
{M(r)}+ \tau_3\,\mathcal{U}(r)\big)\, \frac{k}{r},
$$
where $\mu_k:=\lambda \big(\frac{q^{k} -1} {2} \big): =\lfloor
\log(\frac{q^k-1}{2})\rfloor +\nu(\frac{q^k-1}{2})-1$ and
$\tau_3:=\max\{\tau_1,\tau_2\}$.
\end{lemma}
\begin{proof}
If $j=0$ or $j=1$, then the EDF procedure does not perform any
computation, and the result trivially follows. Therefore, we may
assume that $j\ge 2$.

The cost of a recursive call to the EDF procedure for $ f \in
\mathcal {A}_{j, k} ^{sq}$ is determined by the cost of computing
$h^ {(q ^k-1) / 2} \mod f $, where $ h $ is a random element of $
\fq [T] / (f) $, a greatest common divisor of $f$ with a polynomial
of degree at most $ jk $ and a division of two polynomials of degree
at most $ jk $. Observe that $ \mu_k$ products modulo $f$ are
required to compute $ h^{ (q ^ k-1) / 2} \mod f $ using binary
exponentiation. We conclude that $ h ^ {(q ^ k-1) / 2} \mod f $ can
be computed with at most $2\, \tau_1\, \mu_k M (jk) $ arithmetic
operations in $ \fq $, while the remaining greatest common divisor
and division are computed with at most $ \tau_2\, \mathcal{U} (jk) $
and $ \tau_1\, M(jk) $ arithmetic operations in $ \fq $. In other
words, we have
$$2\, \tau_1\, \mu_k M (jk)+\tau_2\, \mathcal{U} (jk)+\tau_1\,
M(jk)\le \Big(\tau_1\, \mu_k \frac{M (r)}{k\,r}+\tau_2\,
\frac{\mathcal{U} (r)}{2\,k\,r}+\tau_1\,
\frac{M(r)}{2\,k\,r}\Big)(jk)^2$$ arithmetic operations in $\fq$.
Applying \cite[Lemma 4]{FlGoPa01} with $\widetilde{\tau}_1:=
\frac{\tau_1 M(r)}{k\,r}$ and $\widetilde{\tau}_2:=\frac{\tau_3\,
\mathcal{U}(r)}{k\,r}$, we see that
$$\mathcal{X}_{3,k}(f) \leq \bigg(\frac{j(j-1)}{2 \alpha \beta} +
j\sum_{m=0}^{\infty}\sum_{l=0}^m \binom {m}{l} \alpha^{m-l}
\beta^{l}\big(1-(1-\alpha^{m-l} \beta^l)^{j-1}\big)\bigg)\, (\mu_k
\widetilde{\tau}_1 + \widetilde{\tau}_2)\, k^2.
$$
Using the inequality $ 1- (1-u)^{j-1} \leq  (j-1) u $ for $ j \geq 2
$ and $ 0 \leq u \leq 1 $, we obtain
\begin{align*}
\sum_{m=0}^{\infty}\sum_{l=0}^m \binom {m}{l} \alpha^{m-l}
\beta^{l}\big(1-(1-\alpha^{m-l} \beta^l)^{j-1}\big)&\le
(j-1)\sum_{m=0}^{\infty}\sum_{l=0}^m \binom {m}{l} \alpha^{2(m-l)}
\beta^{2l}\\
&\le (j-1) \sum_{m=0}^{\infty}(\alpha^2+
\beta^2)^m=\frac{j-1}{2\alpha \beta}.
\end{align*}
This easily implies the lemma.
\end{proof}

As a consequence of Lemma \ref{Costo}, we have
\begin{equation}
\label{esperanza X3k libre de cuadrados bis} S_{3,k}^{sq}:=
\frac{1}{\vert\mathcal{A}\vert}\sum_{j=2}^{\lfloor r/k \rfloor }
\sum_{f \in \mathcal{A}_{j,k}^{sq}} \mathcal{X}_{3,k}(f) \leq
\sum_{j=2}^{\lfloor r/k \rfloor } \frac{j(j-1)}{\alpha \beta}\,
\big(\tau_1\,\mu_k {M(r)}+ \tau_3\,\mathcal{U}(r)\big)\, \frac{k}{r}
\, \frac{|\mathcal{A}_{j,k}^{sq}|}{|\mathcal{A}|}.
\end{equation}
In the next result we obtain an explicit upper bound for
$S_{3,k}^{sq}$.
\begin{lemma}\label{cota N1}
For $q>15\delta_{\bfs G}^{13/3}$, we have
$$
S_{3,k}^{sq}\leq \frac{1}{\alpha \beta}\bigg(\!\tau_1\mu_k
\frac{M(r)}{k\,r}+ \tau_3\frac{\mathcal{U}(r)}{k\,r}\bigg)
\bigg(1+\frac{15\delta_{\bfs G}^{{13}/{6}}}{q^{{1}/{2}}}\bigg)
\bigg(1+\frac{M_{r}}{q}\bigg),
$$
where $\mu_k$ and $\tau_3$ are as in Lemma \ref{Costo} and $M_{r}$
is defined as in Theorem \ref{average de DDF}.
\end{lemma}
\begin{proof}
According to \eqref{esperanza X3k libre de cuadrados bis}, we
estimate the probability $P[\mathcal {A}_{j, k}^ {sq} ]$ that a
random $ f \in \mathcal { A}$ is square--free and has $ j $
irreducible factors of degree $ k $. In \cite{KnKn90b} it is shown
that, if $ q $ is sufficiently large, then the probability that a
random $ f \in \fq [T] $ of degree at most $r$ has $j$ distinct
irreducible factors of degree $ k $ tends to
$e^{-1/k}\frac{k^{-j}}{j!}$.

We decompose the set $\mathcal{A}_{j,k}^{sq}$ into the disjoint
union
$$
\mathcal{A}_{j,k}^{sq}=\bigcup_{\bfs \lambda \in
\mathcal{P}_r^{j,k}} \mathcal{A}_{j,\bfs \lambda}^{sq},
$$
where $\mathcal{P}_{r}^{j,k}$ is the set of all $ r $--tuples $ \bfs
\lambda: = (\lambda_1, \dots, \lambda_r) \in \mathbb{Z}_{\geq0}^r $
with $\lambda_1 + \dots + r \, \lambda_r = r $ and $ \lambda_k = j
$. Hence, we have
$$P[\mathcal{A}_{j,k}^{sq}]=\frac{1}{|\mathcal{A}|}
\sum_{\bfs \lambda \in \mathcal{P}_{r}^{j,k}} |\mathcal{A}_{j,\bfs
\lambda}^{sq}|.
$$
From Theorem \ref{theorem: estimate fact patterns} we deduce that
$$|\mathcal{A}_{j,\bfs \lambda}^{sq}| \leq q^{r-m}\,\mathcal{T}(\bfs
\lambda)\bigg(1+\frac{M_{r}}{q}\bigg).$$
From Theorem \ref{estimation A} it follows that, for $q>
15\delta_{\bfs G}^{13/3}$,
$$P[\mathcal{A}_{j,k}^{sq}]=\frac{1}{|\mathcal{A}|}
\sum_{\bfs \lambda \in \mathcal{P}_{r}^{j,k}} |\mathcal{A}_{j,\bfs
\lambda}^{sq}| \leq \bigg(1+\frac{15\delta_{\bfs
G}^{{13}/{6}}}{q^{{1}/{2}}}\bigg)\bigg(1+\frac{M_{r}}{q}\bigg)
\sum_{\bfs \lambda \in \mathcal{P}_{r}^{j,k}} \mathcal{T}(\bfs
\lambda).$$
The sum of the right--hand side expresses the probability that a
random permutation in $ \mathbb {S}_r $ has exactly $ j $ cycles of
length $ k $. In \cite{ShLl96} it is shown that
$$
\sum_{\bfs \lambda \in \mathcal{P}_{r}^{j,k}} \mathcal{T}(\bfs
\lambda)=\frac{1}{j!k^j}\sum_{i=0}^{\lfloor r/k-j \rfloor}(-1)^i
\frac{1}{i! k^i}.
$$
We observe that the sum of all probabilities is $ 1 $, that is,
$$
\sum _{j=0}^{\lfloor r/k \rfloor} \frac{1}{j!k^j}\sum_{i=0}^{\lfloor
r/k-j \rfloor}(-1)^i \frac{1}{i! k^i}=1.
$$
As a consequence, by \eqref{esperanza X3k libre de cuadrados bis} we
deduce that
\begin{align*}
S_{3,k}^{sq}& \leq \sum_{j=2}^{\lfloor r/k \rfloor }
\frac{j(j-1)}{\alpha \beta}\bigg(\!\tau_1\mu_k \frac{M(r)}{k\,r}+
\tau_3\frac{\mathcal{U}(r)}{k\,r}\bigg) k^2
\bigg(1+\frac{15\delta_{\bfs
G}^{{13}/{6}}}{q^{{1}/{2}}}\bigg)\bigg(1+\frac{M_{r}}{q}\bigg)
\frac{1}{j!k^j}\!\!\sum_{i=0}^{\lfloor r/k-j \rfloor} \frac{(-1)^i}{i! k^i}\\
& \leq \frac{1}{\alpha \beta}\bigg(\!\tau_1\mu_k \frac{M(r)}{k\,r}+
\tau_3\frac{\mathcal{U}(r)}{k\,r}\bigg) \bigg(1+\frac{15\delta_{\bfs
G}^{{13}/{6}}}{q^{{1}/{2}}}\bigg)
\bigg(1+\frac{M_{r}}{q}\bigg)\sum_{j=2}^{\lfloor r/k \rfloor}
\frac{1}{(j-2)!k^{j-2}}\!\!\sum_{i=0}^{\lfloor r/k-j \rfloor} \frac{(-1)^i}{i! k^i}\\
& \leq \frac{1}{\alpha \beta}\bigg(\!\tau_1\mu_k \frac{M(r)}{k\,r}+
\tau_3\frac{\mathcal{U}(r)}{k\,r}\bigg) \bigg(1+\frac{15\delta_{\bfs
G}^{{13}/{6}}}{q^{{1}/{2}}}\bigg) \bigg(1+\frac{M_{r}}{q}\bigg).
\end{align*}
This shows the lemma.
\end{proof}

Next we obtain an upper bound for
\begin{align}\label{cota N2}
S_{3,k}^{nsq}:=\frac{1}{|\mathcal{A}|} \sum_{f \in
\mathcal{A}^{nsq}} \mathcal{X}_{3,k}(f).
\end{align}
Let $f \in \mathcal{A}^{nsq}$ and $\bfs
b_f:=\mathrm{DDF}(a_f)=(b_f(1),\dots, b_f(s))$. Assume that
$\deg(b_f(k))=m_k$. We have the following bound (see, e.g.,
\cite[Theorem 14.11]{GaGe99}):
$$
\mathcal{X}_{3,k}(f) \leq c\, (k \log q+ \log m_k)M(m_k)
\log\Big(\frac{m_k}{k}\Big),
$$
where $c$ is a constant independent of $k$ and $q$.  Taking into
account the estimate of $|\mathcal{A}^{nsq}|$ of \eqref{eq:
estimates: upper bound discr locus} and Theorem \ref{estimation A},
we conclude that, if $q > 15\delta_{\bfs G}^{13/3}$, then
\begin{align}
    \label{Cota N2 bis}
S_{3,k}^{nsq}& \leq c\, (k \log q+ \log m_k)M(m_k)
\log\bigg(\frac{m_k}{k}\bigg) \frac{2\,r^2\delta_{\bfs G}}{q}.
\end{align}

Now we are able to bound the cost of the EDF procedure.
\begin{theorem}\label{average de EDF}
For $q > 15\delta_{\bfs G}^{13/3}$, the average cost
$E[\mathcal{X}_3]$ of the $\mathrm{EDF}$ algorithm restricted to
$\mathcal{A}$ is upper bounded as
$$E[\mathcal{X}_3] \leq \tau\, M(r) \log q\bigg(\bigg(1+\frac{15\delta_{\bfs
G}^{{13}/{6}}}{q^{{1}/{2}}}\bigg)\bigg(1+\frac{M_{r}}{q}\bigg)+\frac{r^2\delta_{\bfs
G}}{q}\bigg)= \tau\, \mathcal{U}(r) \log q\, (1+o(1)),$$
where $\tau$ is a constant independent of $q$ and $r$ and $M_{r}$ is
defined as in Theorem \ref{average de DDF}.
\end{theorem}
\begin{proof}
Recall that $E[\mathcal{X}_3]=S_{3,k}^{sq}+S_{3,k}^{nsq}$. From
Lemma \ref{cota N1} and \eqref{Cota N2 bis}, we have
\begin{align*}
S_{3,k}^{sq}&\le \bigg(1+\frac{15\delta_{\bfs
G}^{{13}/{6}}}{q^{{1}/{2}}}\bigg)\bigg(1+\frac{M_{r}}{q}\bigg)
\sum_{k=1}^{\lceil r/2 \rceil}  \frac{1}{\alpha
\beta}\bigg(\!\tau_1\mu_k \frac{M(r)}{k\,r}+
\tau_3\frac{\mathcal{U}(r)}{k\,r}\bigg),\\
S_{3,k}^{nsq}&\le \frac{2\,c\,r^2\delta_{\bfs G}}{q}
\sum_{k=1}^{\lceil r/2 \rceil}(k \log q+ \log m_k)M(m_k)
\log\Big(\frac{m_k}{k}\Big).
\end{align*}

We first estimate the sum
$$S_1:=\sum_{k=1}^{\lceil r/2 \rceil}  \frac{1}{\alpha
\beta}\bigg(\!\tau_1\mu_k \frac{M(r)}{k\,r}+
\tau_3\frac{\mathcal{U}(r)}{k\,r}\bigg).$$
Recall that $\mu_k:=\lfloor \log(\frac{q^k-1}{2})\rfloor
+\nu(\frac{q^k-1}{2})-1$, $\alpha:={1}/{2}-{1}/(2 q^k)$ and
$\beta:={1}/{2}+{1}/(2 q^k)$. It is easy to see that
$$
\frac{1}{\alpha \beta}\leq \frac{4q^2}{q^2-1}\leq \frac{16}{3},\quad
\mu_k \leq 2\, k \log q.
$$
As a consequence,
\begin{align*}
S_1\!\leq\! \frac{64 \tau_1}{3}\,\frac{M(r) {\lceil r/2 \rceil} \log
q}{r} + \frac{32 \tau_3}{3}
\frac{\mathcal{U}(r)}{r}\sum_{k=1}^{\lceil r/2 \rceil}
\frac{1}{k}\leq  M(r) \log q\bigg(\frac{64 \tau_1}{3}+\frac{32
\tau_3}{3} \frac{H(\lceil r/2 \rceil)\log r}{r} \bigg),
\end{align*}
where $H(\lceil r/2 \rceil)$ is the $\lceil r/2 \rceil$--th harmonic
number. Since $H(N)\leq 1+\ln N$ (see, e.g., \cite[\S
6.3]{GrKnPa94}), we deduce that, if $r \geq 2$, then ${H(\lceil r/2
\rceil)\log r}/{r} \leq 1$. We conclude that
\begin{align} \label{primera suma de X3 bisbis} S_1\leq M(r) \log
q\bigg(\frac{64 \tau_1}{3}+\frac{32 \tau_3}{3}\bigg).
\end{align}

We now estimate the sum
$$S_2:=\sum_{k=1}^{\lceil r/2 \rceil}(k \log q+ \log m_k)M(m_k)
\log\Big(\frac{m_k}{k}\Big).$$
We have the following inequalities:
$$
\sum_{k=1}^{\lceil r/2 \rceil}k\,
M(m_k)\log\Big(\frac{m_k}{k}\Big)\leq M(r)\sum_{k=1}^{\lceil r/2
\rceil}m_k \frac{\log\big(\frac{m_k}{k}\big)}{\frac{m_k}{k}}\leq
M(r)\sum_{k=1}^{\lceil r/2 \rceil} m_k \leq r\, M(r),
$$
$$
\sum_{k=1}^{\lceil r/2 \rceil} M(m_k)\log(m_k)
\log\Big(\frac{m_k}{k}\Big)  \leq M(r)\sum_{k=1}^{\lceil r/2 \rceil}
\log^2(m_k) \leq M(r)\sum_{k=1}^{\lceil r/2 \rceil} m_k \leq r M(r).
$$
Hence, we deduce that
\begin{align}\label{segunda suma de X3 bis}
S_2& \leq 2\, r M(r) \log q.\\ \notag
\end{align}

From \eqref{primera suma de X3 bisbis} and \eqref{segunda suma de X3
bis} we obtain the following upper bound for $E[\mathcal{X}_3]$:
$$
E[\mathcal{X}_3] \leq M(r)\log q \bigg(\bigg(1+\frac{15\delta_{\bfs
G}^{{13}/{6}}}{q^{{1}/{2}}}\bigg)\bigg(1+\frac{M_{r}}{q}\bigg)\bigg(\frac{64
\tau_1}{3}+\frac{32 \tau_3}{3}\bigg) +\frac{4\,c\,r^3\delta_{\bfs
G}}{q}\bigg).
$$
Defining $ \tau: =\max \{\frac{64 \tau_1} {3}+\frac {32 \tau_3}
{3},4\,c\}$, the statement of the theorem follows.
\end{proof}

In  \cite[Theorem 9]{FlGoPa01}, using the classical multiplication
of polynomials, it is shown that the EDF algorithm requires on
average $ \mathcal {O} (r ^ 2 \log q)$ arithmetic operations in
$\fq$ on the set of elements of $\fq[T]$ of degree at most $r$.
Theorem \ref{average de EDF} proves that, using fast multiplication,
the EDF algorithm performs on average $r\, \log q $ arithmetic
operations in $ \fq $ on $\mathcal{A}$, up to logarithmic terms and
terms which tend to zero as $q$ tends to infinity (for fixed
$\delta_{\bfs G}$ and $r$).

Our analysis improves the worst--case analysis of \cite[Theorem
14.11]{GaGe99}, where it is proved that the EDF algorithm applied to
a polynomial of degree at most $r$ having $j$ irreducible factors of
degree $ k $ requires $ \mathcal {O} ((k \log q+\log r) M(r) \log
j)$ arithmetic operations in $ \fq $, that is, $\mathcal{O}^\sim
(k\,r \log q )$ arithmetic operations in $\fq$.
%
%
\subsection{Average-case analysis of the classical algorithm}
Now we are able to conclude the analysis of the average cost of the
factorization algorithm applied to elements of $ \mathcal {A} $. For
this purpose, it remains to analyze the behavior of the classical
factorization algorithm when the first three steps fail to find the
complete factorization of the input polynomial, namely the expected
value $E[\mathcal{X}_4]$ of the random variable $\mathcal{X}_4$
which counts the number of arithmetic operations in $ \fq$ that the
algorithm performs to factorize $f/\mathrm{ERF}(f) $, when  $ f $
runs over all elements of $ \mathcal {A} $. We can rewrite $ E
[\mathcal {X} _4] $ as follows:
$$
E[\mathcal{X}_4]=\frac{1}{|\mathcal{A}|}\sum_{f \in
\mathcal{A}^{sq}} \mathcal{X}_4(f) + \frac{1}{|\mathcal{A}|}\sum_{f
\in \mathcal{A}^{nsq}} \mathcal{X}_4(f)=:S_4^{sq}+S_4^{nsq}.
$$

We estimate the first sum $S_4^{sq}$. If $f \in \mathcal{A}^{sq}$,
then $f/\mathrm{ERF}(f)=1$ and the algorithm does not perform any
further operation. Hence, the cost of this step is that of dividing
two polynomials of degree at most $r$ at most, namely $\tau_1 M(r)$
arithmetic operations in $\fq$. Thus,
\begin{equation}
\label{Suma 1 de esperanza}
S_4^{sq}
:=\frac{1}{|\mathcal{A}|}\sum_{f \in \mathcal{A}^{sq}} \mathcal{X}_4(f)\leq \tau_1 M(r).
\end{equation}

Now we estimate the second sum $S_4^{nsq}$. For this purpose, we
decompose the set $\mathcal{A}^{nsq}$ into the disjoint union of the
set $\mathcal{A}_{=2}^{nsq}$ of elements having all the irreducible
factors of multiplicity at most $2$, and $\mathcal{A}_{\geq
2}^{nsq}:=\mathcal{A}^{nsq}\setminus \mathcal{A}_{=2}^{nsq}$. If  $f
\in \mathcal{A}_{=2}^{nsq}$, then  $f$ is of the form $f=\prod_i f_i
\prod_j f_j^2$, and we have $f/\mathrm{ERF}(f)=\prod_j f_j$.
Consequently, in this case only the first three steps of the
algorithm are executed, and the worst--case analysis of the
classical algorithm of \cite[Theorem 14.14]{GaGe99} shows that
$\mathcal{X}_4(f) \leq c_3\, r\, M(r)\log(rq)$, where $c_3$ is a
constant independent of $q$ and $r$. On the other hand, if $f \in
\mathcal{A}_{\geq 2}^{nsq}$, then the four steps of the algorithm
are executed. Observe that the last step is executed as many times
as the highest multiplicity arising in the irreducible factors of
$f/\mathrm{ERF}(f)$. Thus the worst--case analysis of \cite[Theorem
14.14]{GaGe99} implies that $\mathcal{X}_4(f) \leq c_4\, r^2 M(r)
\log(rq)$, where $c_4$ is a constant independent of $q$ and $r$. It
follows that
\begin{align}
\label{Suma 2 esperanza} S_4^{nsq}&\leq c_3\, r\, M(r) \log(rq)
\frac{|\mathcal{A}_{=2}^{nsq}|}{|\mathcal{A}|}+ c_4\, r^2 M(r)
\log(rq)\frac{|\mathcal{A}_{\geq 2}^{nsq}|}{|\mathcal{A}|}
\end{align}
Since $\mathcal{A}_{=2}^{nsq}$ is a subset of $\mathcal{A}^{nsq}$,
from \eqref{eq: estimates: upper bound discr locus} we have that
\begin{equation}
\label{cota de no libres de cuadrados con un factor de grado 2}
|\mathcal{A}_{=2}^{nsq}|\leq r(r-1)\delta_{\bfs G}q^{r-m-1} \leq
r^2\delta_{\bfs G} q^{r-m-1}.
\end{equation}
On the other hand, if $f \in \mathcal{A}_{\geq 2}^{nsq}$, then
$\deg(\gcd(f,f')) \geq 2$. We deduce that
$\mathrm{Res}(f,f')=\mathrm{Subres}(f,f')=0$. Hence,
$\mathcal{A}_{\geq 2}^{nsq}$  is a subset of $\mathcal{S}_1(W)$,
where $W\subset \A^r$ is the affine variety defined by $G_1, \dots,
G_m$, $\mathcal{D}(W)$ is the discriminant locus of $W$ and
$\mathcal{S}_1(W)$ is the first subdiscriminant locus of $W$. We
deduce that
\begin{align}
\label{cota de no libres de cuadrados con otros factores}
|\mathcal{A}_{\geq 2}^{nsq}| \leq r(r-1)^2(r-2)\delta_{\bfs G}
q^{r-m-2}\leq r^4\delta_{\bfs G} q^{r-m-2}.
\end{align}
Further, if $q >15\delta_{\bfs G}^{13/3}$, then Theorem
\ref{estimation A} implies $|\mathcal{A}|\geq\frac{1}{2} q^{r-m}.$
Replacing \eqref{cota de no libres de cuadrados con un factor de
grado 2}, \eqref{cota de no libres de cuadrados con otros factores}
in \eqref{Suma 2 esperanza} we obtain
\begin{align}
\label{Suma 2 esperanza bis} \mathcal{S}_4^{nsq} & \leq 2\,c_3 M(r)
\log(rq) \frac{r^3\delta_{\bfs G}}{q}+ 2\,c_4 M(r) \log(rq)
\frac{r^6\delta_{\bfs G}}{q^2}.
\end{align}
Combining \eqref{Suma 1 de esperanza} and \eqref{Suma 2 esperanza bis} we obtain the following result.
\begin{theorem}\label{costo en promedio cuarto paso}
Let $q > 15\delta_{\bfs G}^{13/3}$. The average cost
$E[\mathcal{X}_4]$ of the fourth step of the classical factorization
algorithm on $\mathcal{A}$ is bounded in the following way:
$$E[\mathcal{X}_4] \leq
\tau_1 M(r)+ \frac{c\, r^6\delta_{\bfs G} M(r) \log(rq)}{q}=\tau_1
M(r)(1+o(1)),$$ where $c$ is a constant independent of $q$ and $r$.
\end{theorem}
Theorem \ref{costo en promedio cuarto paso} shows that the average
cost of the last step of the classical factorization algorithm
applied to elements of $\mathcal {A} $ is $\tau_1\, M(r)(1+o(1))$
arithmetic operations in $ \fq $, which asymptotically coincides
with the cost of dividing two polynomials of degree at most $r$.
%
%



\newcommand{\etalchar}[1]{$^{#1}$}
\providecommand{\bysame}{\leavevmode\hbox
to3em{\hrulefill}\thinspace}
\providecommand{\MR}{\relax\ifhmode\unskip\space\fi MR }
\providecommand{\MRhref}[2]{%
  \href{http://www.ams.org/mathscinet-getitem?mr=#1}{#2}
} \providecommand{\href}[2]{#2}

\end{document}